\documentclass[reqno]{amsart}
\usepackage{amscd,amsmath,xypic,amssymb,combelow,tikz-cd,etoolbox,calligra,mathrsfs,enumitem,mathtools,epigraph}

\DeclareMathOperator{\sHom}{\mathscr{H}\text{\kern -3pt {\calligra\large om}}\,}
\DeclareMathOperator{\sRHom}{\mathscr{RH}\text{\kern -3pt {\calligra\large om}}\,}
\DeclareMathOperator{\sQuot}{\mathscr{Q}\text{\kern -3pt {\calligra\large uot}}\,}

\makeatletter
\patchcmd{\@settitle}{\uppercasenonmath\@title}{}{}{}
\makeatother

\xyoption{all}
\CompileMatrices

\newcommand{\nc}{\newcommand}

\makeatletter
\@addtoreset{equation}{section}
\makeatother

\newtheorem{theorem}[subsection]{Theorem}
\newtheorem{proposition}[subsection]{Proposition}
\newtheorem{lemma}[subsection]{Lemma}
\newtheorem{corollary}[subsection]{Corollary}
\newtheorem{conjecture}[subsection]{Conjecture}
\newtheorem{definition}[subsection]{Definition}

\newtheorem{claim}[subsection]{Claim}

\newtheorem{remark}[subsection]{Remark}

\nc{\fa}{{\mathfrak{a}}}
\nc{\fb}{{\mathfrak{b}}}
\nc{\fg}{{\mathfrak{g}}}
\nc{\fh}{{\mathfrak{h}}}
\nc{\fj}{{\mathfrak{j}}}
\nc{\fn}{{\mathfrak{n}}}
\nc{\fm}{{\mathfrak{m}}}
\nc{\fu}{{\mathfrak{u}}}
\nc{\fp}{{\mathfrak{p}}}
\nc{\fr}{{\mathfrak{r}}}
\nc{\ft}{{\mathfrak{t}}}
\nc{\fsl}{{\mathfrak{sl}}}
\nc{\fgl}{{\mathfrak{gl}}}
\nc{\hsl}{{\widehat{\mathfrak{sl}}}}
\nc{\hgl}{{\widehat{\mathfrak{gl}}}}
\nc{\hg}{{\widehat{\mathfrak{g}}}}
\nc{\chg}{{\widehat{\mathfrak{g}}}{}^\vee}
\nc{\hn}{{\widehat{\mathfrak{n}}}}
\nc{\chn}{{\widehat{\mathfrak{n}}}{}^\vee}

\nc{\Mod}{{\textrm{Mod}}}

\nc{\wGL}{{\widehat{GL}^+}}

\nc{\BA}{{\mathbb{A}}}
\nc{\BC}{{\mathbb{C}}}
\nc{\BG}{{\mathbb{G}}}
\nc{\BM}{{\mathbb{M}}}
\nc{\BK}{{\mathbb{K}}}
\nc{\BN}{{\mathbb{N}}}
\nc{\BF}{{\mathbb{F}}}
\nc{\BH}{{\mathbb{H}}}
\nc{\BP}{{\mathbb{P}}}
\nc{\BQ}{{\mathbb{Q}}}
\nc{\BR}{{\mathbb{R}}}
\nc{\BZ}{{\mathbb{Z}}}

\nc{\ff}{{\mathbb{F}}}
\nc{\kk}{{\mathbb{K}}}
\nc{\kko}{{\mathbb{K}}}

\nc{\coh}{{\text{Coh}}}
\nc{\perf}{{\text{Perf}}}

\nc{\op}{{\text{op}}}

\nc{\ecoh}{{\emph{Coh}}}
\nc{\eperf}{{\emph{Perf}}}

\nc{\CA}{{\mathcal{A}}}
\nc{\CC}{{\mathcal{C}}}
\nc{\CB}{{\mathcal{B}}}
\nc{\DD}{{\mathcal{D}}}
\nc{\CE}{{\mathcal{E}}}
\nc{\CF}{{\mathcal{F}}}
\nc{\tCF}{{\widetilde{\CF}}}
\nc{\tCM}{{\widetilde{\CM}}}
\nc{\tCT}{{\widetilde{\CT}}}
\nc{\oCF}{{\bar{\CF}}}
\nc{\CG}{{\mathcal{G}}}
\nc{\CL}{{\mathcal{L}}}
\nc{\CK}{{\mathcal{K}}}
\nc{\CI}{{\mathcal{I}}}
\nc{\CM}{{\mathcal{M}}}
\nc{\CH}{{\mathcal{H}}}
\nc{\CN}{{\mathcal{N}}}
\nc{\CO}{{\mathcal{O}}}
\nc{\CP}{{\mathcal{P}}}
\nc{\CR}{{\mathcal{R}}}
\nc{\CQ}{{\mathcal{Q}}}
\nc{\CS}{{\mathcal{S}}}
\nc{\CT}{{\mathcal{T}}}
\nc{\tCU}{{\widetilde{\CU}}}
\nc{\CU}{{\mathcal{U}}}
\nc{\CV}{{\mathcal{V}}}
\nc{\CW}{{\mathcal{W}}}

\nc{\tpsi}{{\widetilde{\Psi}}}
\nc{\wpi}{{\widetilde{\pi}}}
\nc{\Ker}{{\text{Ker }}}
\nc{\Coker}{{\text{Coker }}}

\nc{\CX}{{\mathcal{X}}}
\nc{\tCX}{{\widetilde{\mathcal{X}}}}
\nc{\CY}{{\mathcal{Y}}}
\nc{\tCY}{{\widetilde{\mathcal{Y}}}}

\nc{\tN}{{\widetilde{\CN}}}
\nc{\pN}{{\BP\widetilde{\CN}}}

\nc{\tT}{{T}}

\nc{\fC}{{\mathfrak{C}}}
\nc{\fY}{{\mathfrak{Y}}}
\nc{\fZ}{{\mathfrak{Z}}}
\nc{\fU}{{\mathfrak{U}}}
\nc{\fV}{{\mathfrak{V}}}
\nc{\fW}{{\mathfrak{W}}}
\nc{\fS}{{\mathfrak{S}}}

\nc{\bfY}{{\bar{\fY}}}
\nc{\bfZ}{{\bar{\fZ}}}

\nc{\od}{{\overline{d}}}
\nc{\rg}{{\textrm{R}\Gamma}}
\nc{\erg}{{\emph{R}\Gamma}}
\nc{\id}{{\textrm{Id}}}
\nc{\rhom}{{\textrm{RHom}}}

\def\Tor{\mathscr{T}or}
\def\Ext{\textrm{Ext}}

\def\Hom{\textrm{Hom}}

\def\e{\varepsilon}

\def\and{\textrm{ }\&\textrm{ }}

\def\sym{\textrm{Sym}}

\def\red{\text{red}}
\def\ered{\emph{red}}

\def\tCF{\widetilde{\CF}}

\def\te{\widetilde{e}}
\def\tf{\widetilde{f}}

\def\loccit{\emph{loc. cit. }}
\def\loccitt{\emph{loc. cit.}}

\def\Tan{\text{Tan}}
\def\eTan{\emph{Tan}}

\def\km{K_\CM}

\def\kms{K_{\CM \times S}}
\def\kmss{K_{\CM \times S \times S}}
\def\kmsss{K_{\CM \times S \times S \times S}}

\def\ks{K_S}
\def\kss{K_{S \times S}}

\def\fr{\text{Frac }\BF}

\def\split{\text{split}}

\def\stack{\text{Stack}}
\def\estack{\emph{Stack}}

\def\comm{\text{Comm}}
\def\ecomm{\emph{Comm}}

\def\quot{\text{Quot}}
\def\equot{\emph{Quot}}

\def\tor{\text{Tor}}

\def\End{{\text{End}}}
\def\taut{{\text{taut}}}

\def\ccur{{\underline{\curvearrowright}}}

\def\Maps{{\text{Maps}}}
\def\const{{\text{const}}}
\def\econst{{\emph{const}}}

\def\pr{{\text{pr}}}

\def\first{{\text{first}}}
\def\efirst{{\emph{first}}}
\def\last{{\text{last}}}
\def\elast{{\emph{last}}}

\def\defect{{\text{def }}}
\def\edefect{{\emph{def }}}

\def\te{{\widetilde{e}}}

\def\qis{{\stackrel{\text{q.i.s.}}\cong}}

\usepackage{etoolbox}

\makeatletter
\patchcmd{\epigraph}{\@epitext{#1}}{\itshape\@epitext{#1}}{}{}
\makeatother

\def\Hilb{{\text{Hilb}}}

\begin{document}
	
	\title[Hecke correspondences for smooth moduli spaces of sheaves]{\large{\textbf{Hecke correspondences for smooth moduli spaces of sheaves}}}
	\author[Andrei Negu\cb t]{Andrei Negu\cb t}
	\address{MIT, Department of Mathematics, Cambridge, MA, USA}
	\address{Simion Stoilow Institute of Mathematics, Bucharest, Romania}
	\email{andrei.negut@gmail.com}
	
	\maketitle
	
	\renewcommand{\thefootnote}{\fnsymbol{footnote}} 
	\renewcommand{\thefootnote}{\arabic{footnote}} 
	
	\begin{abstract} We define functors on the derived category of the moduli space $\CM$ of stable sheaves on a smooth projective surface (under Assumptions A and S below), and prove that these functors satisfy certain commutation relations. These relations allow us to prove that the given functors induce an action of the elliptic Hall algebra on the $K$--theory of the moduli space $\CM$, thus generalizing the action studied by Nakajima, Grojnowski and Baranovsky in cohomology. 
		
		\text{} \newline
		
	\end{abstract}
	
	\epigraph{\epigraphsize Lunei, pentru cinci luni minunate}
	
	\section{Introduction} 
	\label{sec:introduction}
	
\medskip

\subsection{} 
\label{sub:heis intro}

Let $S$ be a smooth projective surface over an algebraically closed field of characteristic 0 (henceforth denoted by $\BC$). An important family of moduli spaces associated to $S$ are the Hilbert schemes of points on $S$, which we will denote by $\Hilb_n(S)$: these are smooth $2n$-dimensional algebraic varieties, which parameterize colength $n$ ideal sheaves $\CI \subset \CO_S$. The Betti numbers of Hilbert schemes were computed in \cite{ES, Go}, and this lead to a very interesting observation: the generating series of the Betti numbers of $\Hilb_n(S)$, as $n$ goes from 0 to $\infty$, matches the Poincar\'e polynomial of (an appropriately graded version of) the Fock space of the infinite-dimensional Heisenberg algebra. Coupled with expectations from mathematical physics, the way to understand this interesting connection is to consider all Hilbert schemes together:
$$
\Hilb(S) = \bigsqcup_{n = 0}^{\infty} \Hilb_n(S)
$$
and construct an action
\begin{equation}
\label{eqn:heis intro}
\text{Heisenberg algebra} \curvearrowright H^*(\Hilb(S)) = \bigoplus_{n=0}^{\infty} H^*(\Hilb_n(S))
\end{equation}
This was achieved by Grojnowski in \cite{G} and Nakajima in \cite{Nak}. We will mostly use the formulation of \emph{loc. cit.}, in which the generators $\{a_{\pm n}\}_{n \in \mathbb{N}}$ of the Heisenberg algebra act on $H^*(\Hilb(S))$ by the so-called \textbf{Hecke correspondences}
\begin{equation}
\label{eqn:corr intro}
\fC_n^\bullet = \Big\{(\CI' \subset_{nx} \CI)\Big\}
\end{equation}
where $\CI' \subset_{nx} \CI$ means that the ideal sheaves $\CI, \CI' \subset \CO_S$ are contained inside each other, and their quotient is a length $n$ sheaf supported at a single (but arbitrary) closed point $x \in S$. Nakajima considered the projection maps

	$$
	\xymatrix{& \fC_n^\bullet \ar[ld]_{p_+} \ar[d]^{p_S} \ar[rd]^{p_-} & \\
		\Hilb(S) & S & \Hilb(S)} \qquad \qquad \xymatrix{& \CI' \subset_{nx} \CI \ar@{|->}[ld] \ar@{|->}[d] \ar@{|->}[rd] & \\
		\CI' & x & \CI}
	$$
and defined	
\begin{equation}
\label{eqn:ops intro}
a_{\pm n} = (p_{\mp} \times p_S)_* \circ p_{\pm}^* : H^*(\Hilb(S)) \rightarrow H^*(\Hilb(S)) \otimes H^*(S)
\end{equation}
The main result of \cite{Nak} is that the operators above satisfy the defining relations of the Heisenberg algebra, with $H^*(S)$ as a ``parameter space". Under this action, $H^*(\Hilb(S))$ is isomorphic to the Fock space of the Heisenberg algebra, which explains in a representation theoretic way the classical formulas for the Betti numbers. \\

\subsection{} 
\label{sub:k intro}

After constructing the action \eqref{eqn:heis intro}, one might wonder how to generalize it. One direction would be to replace singular cohomology by other homology theories. While the construction extends almost verbatim to Chow groups, in algebraic $K$--theory things are not so simple. The reason for this is that the correspondences $\fC_n^\bullet$ are very badly behaved for $n>1$ (one does not even know if they are equidimensional in general) so the initial question is to find suitable replacements.
\begin{align*}
\textbf{Problem:} &\text{ define analogues of the operators \eqref{eqn:ops intro} in }K\text{--theory, which} \\ 
&\text{ satisfy the defining relations of the deformed Heisenberg algebra}
\end{align*}
One approach toward this problem was undertaken in \cite{FT, SV} for $S = \BA^2$, where the authors used equivariant localization computations to construct an action
\begin{equation}
\label{eqn:a intro}
\CA \curvearrowright K_{\Hilb(S)} = \bigoplus_{n=0}^{\infty} K_{\Hilb_n(S))}
\end{equation}
(see Subsection \ref{sub:k-theory} for our conventions on $K$--theory). In the elliptic Hall algebra formulation of \cite{SV}, the algebra $\CA$ is generated by elements $\{e_{\pm n,k}\}_{(n,k) \in \mathbb{N} \times \mathbb{Z}}$, with the subalgebra generated by the $e_{\pm n,0}$'s isomorphic to the deformed Heisenberg algebra, thus explaining our interest in constructing an action \eqref{eqn:a intro}. In \cite{Mod}, we realized the operators $e_{\pm n,k}$ using certain $K$-theory classes on the correspondence
\begin{equation}
\label{eqn:z intro}
\fZ_n^\bullet = \Big\{(\CI_0 \subset_{x} \CI_1 \subset_x \dots \subset_x \CI_n)\Big\}
\end{equation}
which can thus be thought of as a resolution of $\fC_n^\bullet$ of \eqref{eqn:corr intro}. While \emph{loc. cit.} only proved results for $S = \BA^2$, the correspondence \eqref{eqn:z intro} is well-defined for any smooth surface $S$, and we are thus poised to solve the \textbf{Problem} stated above. Even more so, the operators $e_{\pm n,k}$ naturally lift to the derived category, and are thus related to various categorifications of Hecke algebras (see \cite{GNR} for an overview). \\
	
	\subsection{} 
	\label{sub:main intro}
	
	Another direction for generalizing \eqref{eqn:heis intro} is to replace Hilbert schemes by other moduli spaces of coherent sheaves on $S$. An important example is given by moduli spaces of stable sheaves (see Subsection \ref{sub:first mod} for more details), in which the analogue of the action \eqref{eqn:heis intro} was worked out in \cite{Ba}. The main purpose of the present paper is to construct an action \eqref{eqn:a intro} in the aforementioned setup of stable sheaves. Thus, let $S$ be a smooth projective surface with an ample divisor $H$, and also fix $(r,c_1) \in \BN \times H^2(S, \BZ)$. Consider the moduli space $\CM$ of $H$--stable sheaves on the surface $S$ with the numerical invariants $r,c_1$ and any $c_2 \in \BZ$. We make the following two assumptions throughout the present paper:
	\begin{align}
	&\textbf{Assumption A:} \qquad  \gcd(r, c_1 \cdot H) = 1 \label{eqn:assumption a} \\
	&\textbf{Assumption S:} \qquad \begin{cases} \text{either } \omega_S \cong \CO_S, \\ \text{or } c_1(\omega_S) \cdot H < 0 \end{cases} \label{eqn:assumption s}
	\end{align}
	Assumption A implies that $\CM$ is representable, i.e. there exists a universal sheaf $\CU$ on $\CM \times S$. Assumption S implies that $\CM$ is smooth, which allows us to define
	$$
	D_\CM = D^b(\coh(\CM)) = \perf(\CM)
	$$
In the following formulas, we will denote by $\CF$ the coherent sheaves on $S$ that are parameterized by the moduli space $\CM$. With this in mind, let us recall the following moduli spaces of flags of sheaves that we studied in \cite{Univ} and \cite{W univ}, respectively:
	\begin{align*}
	&\fZ_1 = \Big\{ (\CF_0 \subset_x \CF_1) \Big\} \\
	&\fZ_2^\bullet = \Big\{ (\CF_0 \subset_x \CF_1 \subset_x \CF_2) \Big\} 
	\end{align*}
	where $\CF' \subset_x \CF$ means that $\CF' \subset \CF$ and $\CF/\CF'$ is the length 1 skyscraper sheaf at the closed point $x \in S$. We recall the scheme structures of $\fZ_1$ and $\fZ_2^\bullet$ in Section \ref{sec:mod}, and in particular, the fact that they are both smooth. This implies that the maps
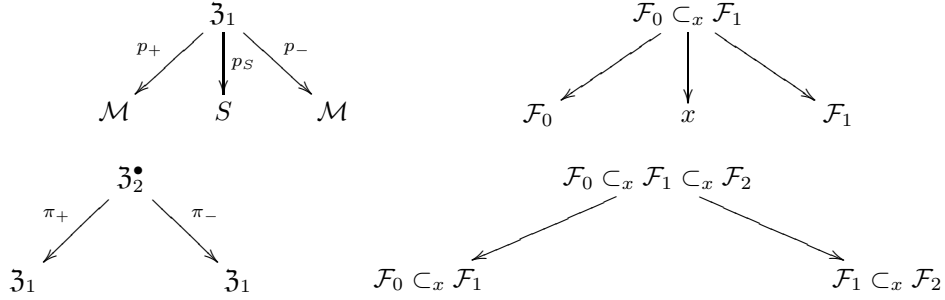
	$$
	\xymatrix{& \fZ_1 \ar[ld]_{p_+} \ar[d]^{p_S} \ar[rd]^{p_-} & \\
		\CM & S & \CM} \qquad \qquad \qquad  \xymatrix{& \CF_0 \subset_x \CF_1 \ar@{|->}[ld] \ar@{|->}[d] \ar@{|->}[rd] & \\
		\CF_0 & x & \CF_1}
	$$
	$$
	\xymatrix{& \fZ_2^\bullet \ar[ld]_{\pi_+} \ar[rd]^{\pi_-} & \\
		\fZ_1 & & \fZ_1} \qquad \qquad \xymatrix{& \CF_0 \subset_x \CF_1 \subset_x \CF_2 \ar@{|->}[ld] \ar@{|->}[rd] & \\
		\CF_0 \subset_x \CF_1 & & \CF_1 \subset_x \CF_2}
	$$
	induce direct and inverse image functors between the derived categories $D_\CM$, $D_{\fZ_1}$, $D_{\fZ_2^\bullet}$. We may combine these spaces into more complicated diagrams of the form
	$$
	
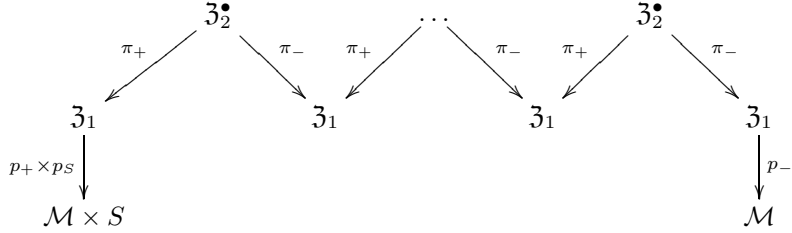
\xymatrix{& \fZ_2^\bullet \ar[ld]_{\pi_+} \ar[rd]^{\pi_-} & & \dots \ar[ld]_{\pi_+} \ar[rd]^{\pi_-} & & \fZ_2^\bullet \ar[ld]_{\pi_+} \ar[rd]^{\pi_-} & & \\
		\fZ_1 \ar[d]_{p_+ \times p_S} & & \fZ_1 & & \fZ_1 & & \fZ_1 \ar[d]^{p_-} \\
		\CM \times S & & & & & & \CM}
	$$
	where the number of copies of $\fZ_1$ contained in the middle row is denoted by $n$. Then for an arbitrary sequence $d_1,\dots,d_n \in \BZ$, we consider the functor
	\begin{equation}
	\label{eqn:def functor}
	\te_{(d_1,\dots,d_n)} : D_{\CM} \rightarrow D_{\CM \times S}
	\end{equation}
	given by the composition
	$$
	\xymatrix{D_{\fZ_1} \ar[d]^{(p_+ \times p_S)_*} & D_{\fZ_1} \ar[l]_{\otimes \CL^{d_1}} & D_{\fZ_1} \ar[l]_{\pi_{+*} \pi_-^*} & D_{\fZ_1} \ar[l]_-{\otimes \CL^{d_2}} \quad \dots \quad D_{\fZ_1} & D_{\fZ_1} \ar[l]_-{\pi_{+*} \pi_-^*} & D_{\fZ_1} \ar[l]_{\otimes \CL^{d_n}} \\
		D_{\CM \times S} & & & & & D_{\CM} \ar[u]^{p_-^*}}
	$$
	where $\CL$ is the line bundle on $\fZ_1$ with fibers $\Gamma(S,\CF_{1}/\CF_{0})$ \footnote{More precisely, consider the universal sheaves $\CU_0 \subset \CU_1$ on $\fZ_1 \times S$; then $\CL$ is defined as the push-forward of the quotient $\CU_1/\CU_0$ to $\fZ_1$ via the first projection.}. Note that $\te_d := \te_{(d)}$ are the derived category versions of the operators $E_{1,d}$ of Subsection \ref{sub:k intro}, but it is crucial to us that $\te_{(d_1,\dots,d_n)}$ is different from the composition $\te_{d_1} \circ \dots \circ \te_{d_n}$ (see Propositions \ref{prop:xx} and \ref{prop:base change excess}). With this in mind, our main result is the following. \\
	
	\begin{theorem}
		\label{thm:main}
		
		For any $n \in \BN$ and any $d_1,\dots,d_n,k \in \BZ$, consider the functors
		\begin{equation}
		\label{eqn:main}
		\te_{(d_1,\dots,d_n)} \circ \te_k \quad \text{and} \quad \te_k \circ \te_{(d_1,\dots,d_n)} : D_{\CM} \rightarrow D_{\CM \times S \times S}
		\end{equation}
of \eqref{eqn:composition 1},\eqref{eqn:composition 2}. There are explicit functors $g_0,\dots,g_n : D_{\CM} \rightarrow D_{\CM \times S \times S}$ such that \\
		
		\begin{enumerate}[leftmargin=*]
			
			\item $g_0 = \te_{(d_1,\dots,d_n)} \circ \te_k$ and $g_n = \te_k \circ \te_{(d_1,\dots,d_n)}$; \\
			
			\item for all $i \in \{1,\dots,n\}$, there exist explicit natural transformations
			\begin{equation}
			\label{eqn:natural transformation}
			\begin{cases} g_{i-1} \rightarrow g_i & \text{ if } d_i > k \\ g_{i-1} \leftarrow g_i & \text{ if } d_i < k \\ g_{i-1} \cong g_i & \text{ if } d_i = k \end{cases}
			\end{equation}
			
			\item for all $i \in \{1,\dots,n\}$, the cone of the natural transformation in the previous item has a filtration with the associated graded object given by the functor
$$
\begin{cases}
\displaystyle \bigoplus_{a=k}^{d_i-1} \Delta_*(\te_{(d_1,\dots,d_{i-1},a,d_i+k-a,d_{i+1},\dots,d_n)}) &\text{if } d_i > k \\
\displaystyle \bigoplus_{a=d_i}^{k-1} \Delta_*(\te_{(d_1,\dots,d_{i-1},a,d_i+k-a,d_{i+1},\dots,d_n)}) &\text{if } d_i < k
\end{cases}
$$			
Above, we abuse notation by writing $\Delta : \CM \times S \rightarrow \CM \times S \times S$ for the identity on $\CM$ times the diagonal embedding $S \hookrightarrow S \times S$. \\
			
		\end{enumerate}
		
	\end{theorem}
	
	\subsection{}  To gain more insight on the meaning of the categorical relations in (1)--(3) above, let us consider the maps induced by \eqref{eqn:def functor} at the level of $K$--theory groups:
	\begin{equation}
	\label{eqn:def functor k}
	e_{(d_1,\dots,d_n)} : K_{\CM} \rightarrow K_{\CM \times S}
	\end{equation}
	Items (1)--(3) in Theorem \ref{thm:main} imply that these maps satisfy the following relations:
	\begin{equation}
	\label{eqn:new relations}
	[e_{(d_1,\dots,d_n)}, e_k] = \Delta_* \sum_{i=1}^n \begin{cases} - \sum_{k \leq a < d_i} e_{(d_1,\dots,d_{i-1},a,d_i+k-a,d_{i+1},\dots,d_n)}  & \text{if } d_i > k \\ \\ \sum_{d_i \leq a < k} e_{(d_1,\dots,d_{i-1},a,d_i+k-a,d_{i+1},\dots,d_n)}  & \text{if } d_i < k \end{cases}
	\end{equation}
	We show in Section \ref{sec:a infty acts} that these relations are sufficient to conclude that the maps \eqref{eqn:def functor k} induce an action of the elliptic Hall algebra $\CA$ (see \cite{BS, S}, and Section \ref{sec:a infty acts} for a review) on the $K$--theory groups of the moduli space $\CM$. We thus conclude: \\
	
	\begin{corollary}
		\label{cor:main}
		
(see Theorem \ref{thm:a infty acts} for details) There exists an action $\CA \curvearrowright \km$, in the sense of Definition \ref{def:action}. This notion of action is defined so that composing maps $\km \rightarrow \kms$ treats the second factor of $S$ as a ``parameter space". \\
		
	\end{corollary}
	
	\noindent We prove the Corollary only upon tensoring $\km$ with $\BQ$, but we expect it remains true over $\BZ$. The parameters $q_1, q_2$ of the elliptic Hall algebra $\CA$ are such that
	\begin{align*} 
	&q_1+q_2 = [\Omega_S^1] \\ 
	&q = q_1q_2 = [\omega_S]
	\end{align*}
	as elements of $\ks$. Therefore, $\ks$ plays the role of ground ring of the algebra $\CA$. \\
	
	\subsection{} As a consequence of Corollary \ref{cor:main} and the defining relations in $\CA$, the maps
	$$
	\underbrace{e_{(0,\dots,0)}}_{n \text{ zeroes}} : K_{\CM} \rightarrow K_{\CM \times S}
	$$
	satisfy the relations of the deformed Heisenberg algebra. This solves the Problem in Subsection \ref{sub:k intro} in the context of moduli spaces of stable sheaves (strictly speaking, the operators above only give half of the Heisenberg algebra, with the other half provided by the transposed correspondences). It is natural to propose the following. \\
	
	\begin{conjecture}
		\label{conj:main}
		
		For any $n,m \in \BN$, we have
		$$
		\underbrace{\te_{(0,\dots,0)}}_{n \text{ zeroes}} \circ \underbrace{\te_{(0,\dots,0)}}_{m \text{ zeroes}} \cong \underbrace{\te_{(0,\dots,0)}}_{m \text{ zeroes}} \circ \underbrace{\te_{(0,\dots,0)}}_{n \text{ zeroes}}
		$$
		as functors $D_{\CM} \rightarrow D_{\CM \times S \times S}$. \\
		
	\end{conjecture}
	
	\noindent Theorem \ref{thm:main} proves the $m=1$ case of Conjecture \ref{conj:main}, for any $n$. It would be interesting to compare Conjecture \ref{conj:main} with other categorifications of Heisenberg algebras in the literature. For example, such a categorification was proposed in \cite{CL} when the surface is of ADE type (generalized to all surfaces in \cite{K}), but using the language of derived categories of symmetric powers of the surface $S$. It is thus not easy, but would be very interesting, to compare our results with those of \emph{loc. cit.} \\
	
	\noindent One can generalize Conjecture \ref{conj:main} by recalling that the generators $e_{-n,k} \in \CA$ satisfy
	\begin{equation}
	\label{eqn:am}
	e_{-n,k} = q^{\gcd(n,k)-1} e_{(d_1,\dots,d_n)}, \quad \text{where } d_i = \left \lceil \frac {ki}n \right \rceil - \left \lceil \frac {k(i-1)}n \right \rceil + \delta_i^n - \delta_i^1
	\end{equation}
Corollary \ref{cor:main} implies the following formula, for all $(n,k), (n',k') \in \BN \times \BZ$:
$$
[e_{-n,k}, e_{-n',k'}] = \Delta_*(\text{a certain linear combination of products of } e_{-n'',k''}\text{'s})
$$
where the coefficients of the linear combination match the analogous structure constants in the algebra $\CA$ (see \eqref{eqn:comm} and \eqref{eqn:comm exp}). We expect such formulas to categorify. \\
	
	\begin{conjecture}
	\label{conj:small}
	
	Suppose $(n,k), (n',k') \in \BN \times \BZ$ are such that $kn' - k'n = 1$. Then there exists a natural transformation between the functors
	$$
	D_\CM \xrightarrow{\te_{-n,k} \circ \te_{-n',k'}} D_{\CM \times S \times S} \qquad \text{and} \qquad D_\CM \xrightarrow{\te_{-n',k'} \circ \te_{-n,k}} D_{\CM \times S \times S}
	$$
	(where $\te_{-n,k} = \te_{(d_1,\dots,d_n)}$ with the $d_i$'s as in \eqref{eqn:am}) whose cone is $\Delta_*(\te_{-n-n',k+k'})$. \\
	
	\end{conjecture} 

	\subsection{} Working with $\te_{(d_1,\dots,d_n)} : D_\CM \rightarrow D_{\CM \times S}$ instead of $e_{(d_1,\dots,d_n)} : \km \rightarrow \kms$ reveals new features, such as the presence of categorified knot invariants (namely Khovanov homology). Indeed, following the principles laid out in \cite{GNR}, the natural transformations that appear in Theorem \ref{thm:main} can be interpreted as yielding skein exact triangles between certain objects in the category $\mathscr{C}$ of affine Soergel bimodules (in \cite{GN}, we make this precise by realizing Proposition \ref{prop:zero} in $\mathscr{C}$). \\
	
	\noindent The construction of the functors \eqref{eqn:def functor} was given in terms of derived schemes in \cite{W univ}, and Corollary \ref{cor:main} was conjectured therein. In the present paper, we show that Assumption S implies that many of the spaces featured in \loccit are local complete intersections. Therefore, one can work in classical algebraic geometry rather than derived algebraic geometry. By combining our Corollary \ref{cor:main} with the results of \cite{AGT Univ}, we obtain the identification of the Carlsson-Okounkov operator $\km \rightarrow \km$ with a $W$--algebra intertwiner (see \loccit for definitions). When the surface is $S = \BA^2$ (and $K$--theory is understood equivariantly), this is enough to prove a mathematical incarnation of the  deformed AGT correspondence for $U(r)$ gauge theory with matter, see \cite{W} (inspired by \cite{MO, SV3} and the physics literature). \\
	
	\noindent We observe that Theorem \ref{thm:main} also holds when the moduli space of stable sheaves is replaced by the Hilbert scheme of points on an arbitrary smooth quasiprojective surface (we do not need the surface to be projective, as all the push-forward maps in Subsection \ref{sub:main intro} are proper). The modifications required to make the argument work are minimal, and we leave the details to the interested reader. When the surface $S$ is $K3$, we revisit some of the methods in the present paper in \cite{MN} in order to prove that the cycle class map
	\begin{equation}
	\label{eqn:bv}
	A^*_{\text{taut}} \rightarrow H^*(\text{Hilb}(K3))
	\end{equation}
	is injective, where $A^*_{\text{taut}}$ denotes the subring of the Chow ring of the Hilbert scheme of arbitrarily many points on $S$ that is generated by the tautological classes. This is a partial version of the Beauville-Voisin conjecture \cite{V} for hyperk\"ahler manifolds. \\
	
	\subsection{} The structure of the present paper is the following. In Section \ref{sec:mod} we recall the moduli space of stable sheaves $\CM$, introduce the moduli spaces of flags of sheaves $\fZ_\lambda$ and several versions of the moduli space of quadruples $\fY$, and state a number of geometric properties of $\fZ_\lambda$ and $\fY$ (such as dimension, irreducibility, smoothness, l.c.i.-ness, Cohen-Macaulay-ness and normality). In Section \ref{sec:derived categories}, we use these properties to set up and prove Theorem \ref{thm:main}. In Section \ref{sec:a infty acts}, we go from derived categories to $K$--theory and prove Corollary \ref{cor:main}. Finally, in Sections \ref{sec:app} and \ref{sec:geom} we prove all the aforementioned geometric properties: dimension estimates and irreducibility in Section \ref{sec:app} and all other properties in Section \ref{sec:geom}. \\
	
	\noindent I would like to thank Mina Aganagic, Roman Bezrukavnikov, Tom Bridgeland, Eugene Gorsky, Sergei Gukov, Tamas Hausel, Ivan Loseu, Alina Marian, Davesh Maulik, Alexander Minets, Georg Oberdieck, Alexei Oblomkov, Andrei Okounkov, Hiraku Nakajima, Francesco Sala, Olivier Schiffmann, Richard Thomas, Alexander Tsymbaliuk and the anonymous referee for many interesting discussions on the subject. I would like to thank MSRI, Berkeley, for their hospitality while this paper was being written in the Spring semester of 2018. I gratefully acknowledge the support of NSF grants DMS-1600375 and DMS-1440140. \\
	
	\section{The moduli space of sheaves}
	\label{sec:mod}
	
	\medskip
	
	\subsection{} 
	\label{sub:notation}
	
All our schemes will be projective over an algebraically closed field of characteristic 0, henceforth denoted by $\BC$. A scheme will be called smooth if all of its local rings are regular. A closed embedding $Z \hookrightarrow X$ is called regular if the ideal of $Z$ in $X$ is locally generated by a regular sequence. In the present paper, we will often encounter closed embeddings which arise as the zero loci of sections
	$$
	\sigma : \CO_X \longrightarrow W
	$$
	where $W$ is a locally free sheaf on the scheme $X$. We will often abuse terminology and refer to either $\sigma$ or $\sigma^\vee : W^\vee \rightarrow \CO_X$ as ``the section", with the latter having the advantage that the zero locus is defined as the scheme
	\begin{equation}
	\label{eqn:zero locus}
	Z(\sigma) = \text{Spec}_X \left( \frac {\CO_X}{\text{Im } \sigma^\vee} \right)
	\end{equation}
	If $X$ is a Cohen-Macaulay scheme, we have the well-known inequality
	\begin{equation}
	\label{eqn:inequality}
	\dim Z(\sigma) \geq \dim X - \text{rank } W
	\end{equation}
	Since all schemes in our paper will be Cohen-Macaulay, we give the following. \\

	\begin{definition}
		\label{def:regular}
		
		The section $\sigma$ is called \textbf{regular} if equality holds in \eqref{eqn:inequality}. \\
		
	\end{definition}
	
	\noindent Indeed, the usual definition of regularity (that the coordinates of $\sigma$ in any local trivialization of $W$ form a regular sequence) is equivalent to that of Definition \ref{def:regular} over Cohen-Macaulay schemes. In this case, we have a quasi-isomorphism
	$$
	\CO_{Z(\sigma)} = \frac {\CO_X}{\text{Im }\sigma^\vee} \qis \left[ \dots \longrightarrow \wedge^2 W^\vee \longrightarrow W^\vee \xrightarrow{\sigma^\vee} \CO_X \right]
	$$
(the maps in the complex above are given by contraction with $\sigma^\vee$)	and we call $Z(\sigma) \hookrightarrow X$ a complete intersection. \\
	
	\subsection{} 
	\label{sub:derived}
	
	Given a map of schemes $\eta : X' \rightarrow X$, we may form the fiber square
	\begin{equation}
	\label{eqn:lci diagram}
	\xymatrix{Z(\eta^*(\sigma)) \ar[d] \ar@{^{(}->}[r] & X' \ar[d]^\eta \\
		Z(\sigma) \ar@{^{(}->}[r] & X}
	\end{equation}
	where
	\begin{equation}
	\label{eqn:sigma prime}
	\eta^*(W^\vee) \xrightarrow{\eta^*(\sigma^\vee)} \CO_{X'}
	\end{equation}
	is the pull-back of $\sigma^\vee$. However, we note the very important fact that
	$$
	\sigma \text{ regular} \quad \text{does not imply} \quad \eta^*(\sigma) \text{ regular}
	$$
	in general, because regular sequences are not preserved under pull-back. \\
	
	\begin{definition}
		\label{def:lci}
		
		If the sections $\sigma$ and $\eta^*(\sigma)$ are both regular, then we call \eqref{eqn:lci diagram} a \textbf{derived fiber square}. We will also use this terminology when the horizontal arrows in \eqref{eqn:lci diagram} are embeddings of zero loci followed by smooth maps. \\
		
	\end{definition}
	
	\noindent If $\sigma$ is a regular section and $X$, $X'$ are Cohen-Macaulay schemes, then \eqref{eqn:inequality} implies that $\eta^*(\sigma)$ is regular if and only if
	\begin{equation}
	\label{eqn:eq lci}
	\dim Z(\eta^*(\sigma)) - \dim X' = \dim Z(\sigma)  - \dim X
	\end{equation}
	
\medskip	
	
\begin{definition} 
\label{def:excess square}

With the notation as above, assume that there exists a subsheaf 
$$
L \subset \emph{Ker } \eta^*(\sigma^\vee)
$$
such that both $L$ and $\frac {\eta^*(W^\vee)}{L}$ are locally free. In this case, if the induced section
\begin{equation}
\label{eqn:excess}
\frac {\eta^*(W^\vee)}{L} \xrightarrow{{\sigma'}^\vee} \CO_{X'}
\end{equation}
is regular, then we call \eqref{eqn:lci diagram} a \textbf{derived fiber square} with \textbf{excess}. \\

\end{definition}	
	
\noindent In the setting of Definition \ref{def:excess square}, the locally free sheaf $L$ is called the \textbf{excess bundle}, and its rank is precisely the difference between the two sides of \eqref{eqn:eq lci}. \\
	
	\subsection{} 
	
	Given a locally free sheaf $V$ on a scheme $X$, consider the projective bundle
	$$
	\BP_X(V) = \text{Proj} \left( \text{symmetric algebra of }V \right)
	$$
	Projective bundles are among the easiest examples of moduli spaces in algebraic geometry, in the sense that they represent the functor of line bundle quotients of the vector bundle $V$. More specifically, this means that there is a natural identification
	\begin{multline}
	\text{Maps}(T, \BP_X(V)) \xleftrightarrow{1\text{-to-}1} \\ 
	\Big\{ T \xrightarrow{\phi} X, \text{ line bundle } \CL \text{ on }T, \text{ surjection } \phi^*(V) \twoheadrightarrow \CL \Big\} \label{eqn:proj represent}
	\end{multline}
	Assume that we have a map $W \rightarrow V$ of locally free sheaves on $X$. In the present paper, we will encounter local complete intersection morphisms of the form
	\begin{equation}
	\label{eqn:iota pi}
	Z(\sigma) \stackrel{\iota}\hookrightarrow \BP_X(V) \stackrel{\rho}\twoheadrightarrow X
	\end{equation}
	where the map $\iota$ is cut out by the section
	\begin{equation}
	\label{eqn:abusive section}
	\sigma : \rho^* (W) \longrightarrow \rho^* (V) \xrightarrow{\taut} \CO(1)
		\end{equation}
	and the map denoted by $\taut$ is the tautological morphism on $\BP_X(V)$. \\
	
	\begin{remark}
	
	We note a convenient abuse of terminology in calling \eqref{eqn:abusive section} a ``section", as it will appear again throughout the present paper. Given a line bundle $L$ and a vector bundle $E$ on a scheme $X$, whenever we refer to a map:
$$
L \xrightarrow{\sigma} E
$$
as a section, we are implicitly referring to the induced section $\CO \rightarrow E \otimes L^{-1}$. The same terminology will apply to the dual map $E^\vee \rightarrow L^{-1}$ in relation to $E^\vee \otimes L \rightarrow \CO$. \\

	\end{remark}
	
\noindent Since the map $\rho$ in \eqref{eqn:iota pi} is smooth, the considerations of the preceding Subsection apply to the composed map $Z(\sigma) \rightarrow X$. More specifically, given any map $\eta : X' \rightarrow X$, the fiber square analogous to \eqref{eqn:lci diagram} is derived iff $\sigma$ and $\eta^*(\sigma)$ are regular. \\
	
	\subsection{} 
	\label{sub:first mod}
	
	Consider a smooth projective surface $S$ and an ample divisor $H \subset S$. The Hilbert polynomial of a coherent sheaf $\CF$ on $S$ is given by
	$$
	P_\CF(n) := \chi(S, \CF \otimes \CO(nH)) = a n^2 + b n + c
	$$
	where $a,b,c$ are rational numbers that one can compute from the Hirzebruch-Riemann-Roch theorem. One can find formulas for these numbers in the Appendix to \cite{Univ}, but the only thing we will need in the present paper is that they can be expressed in terms of $S,H$ and the rank $r$ and Chern classes $c_1, c_2$ of $\CF$. If $\CF$ is torsion free, then $a > 0$, and one defines the reduced Hilbert polynomial as
	$$
	p_\CF(n) = \frac {P_\CF(n)}{a}
	$$ 
	A torsion-free coherent sheaf $\CF$ on $S$ is called (Gieseker $H$--) \textbf{stable} if for all proper subsheaves $\CG \subset \CF$, we have the inequality
	$$
	p_\CG(n) < p_\CF(n) 
	$$
	whenever $n \gg 0$. Since the reduced Hilbert polynomials are monic and quadratic, stability is determined by checking certain inequalities for the linear term and constant term coefficients. Note that stability depends on the ample divisor $H$, but we will fix a choice throughout this paper. \\
	
	\begin{definition} (see \cite{HL}): Let $\CM_{(r,c_1,c_2)}$ denote the quasiprojective variety which corepresents the moduli functor of stable sheaves on $S$ with the invariants $r,c_1,c_2$. \\
		
	\end{definition}
	
	\noindent It is straightforward to compute the tangent spaces to the variety $\CM_{(r,c_1,c_2)}$, and the interested reader can refer to Subsection \ref{sub:tangent} in the Appendix. A simple consequence of this computation is the following well-known fact. \\
	
	\begin{proposition}
		\label{prop:m smooth}
		
		Under Assumption $S$, $\CM_{(r,c_1,c_2)}$ is smooth of dimension
		\begin{equation}
		\label{eqn:dim m}
		\emph{const} + 2rc_2
		\end{equation}
		where $\econst$ only depends on $S,H,r,c_1$ (the specific formula can be found in \cite{Univ}). \\
		
	\end{proposition}
	
	\subsection{}
	\label{sub:ass a}
	
	Recall Assumption A of \eqref{eqn:assumption a}, which states that $\gcd(r, c_1 \cdot H) = 1$. An important consequence of this assumption is the existence of a universal sheaf
	\begin{equation}
	\label{eqn:universal}
	\xymatrix{
		\CU_{(r,c_1,c_2)} \ar@{.>}[d] \\
		\CM_{(r,c_1,c_2)} \times S}
	\end{equation}
which is flat over $\CM_{(r,c_1,c_2)}$, and has the universal property \eqref{eqn:universal property}. In what follows, given a coherent sheaf $\CF$ on a scheme $T \times S$, we will denote by $\CF_t$ its fiber over $\{t\} \times S$, for any closed point $t \in T$. Then we have the following fact (see \cite{HL}). \\
	
	\begin{proposition} 
		\label{prop:represents}
		
		Under Assumption A, $\CM_{(r,c_1,c_2)}$ is a projective variety, which represents the moduli functor of stable sheaves on $S$ with the invariants $r,c_1,c_2$. \\
		
	\end{proposition}
	
	\noindent The functorial description of $\CM_{(r,c_1,c_2)}$ entails the existence of natural bijections
\begin{multline}
\label{eqn:bijection}
\quad	\Maps(T,\CM_{(r,c_1,c_2)}) \xleftrightarrow{1\text{-to-}1} \Big \{ \text{coherent sheaves } \CF \text{ on } T \times S, \text{ flat over }T, \\
\text{such that } \CF_t \text{ is stable with invariants } (r,c_1,c_2), \forall t \in T \Big\} \Big / \sim   \qquad
\end{multline}
for all schemes $T$. Above, we write $\CF \sim \CF'$ if $\CF' = \CF \otimes \pi_1^*(\CL)$ for some line bundle $\CL$ on $T$, and $\pi_1 : T \times S \rightarrow T$ denotes the standard projection. The word ``represents" in Proposition \ref{prop:represents} means that the 1-to-1 correspondence \eqref{eqn:bijection} is given by
	\begin{equation}
	\label{eqn:universal property}
	\Big\{ T \xrightarrow{\phi} \CM_{(r,c_1,c_2)} \Big\} \leadsto \Big\{ \CF = (\phi \times \text{Id}_S)^* (\CU_{(r,c_1,c_2)}) \Big\}
	\end{equation}
Because of the equivalence relation $\sim$ in \eqref{eqn:bijection}, the universal sheaf $\CU_{(r,c_1,c_2)}$ is not unique, but may be tensored with any line bundle pulled back from $\CM_{(r,c_1,c_2)}$. \\

\subsection{} From now on, we will fix $(r, c_1) \in \BN \times H^2(S,\BZ)$ and set
$$
\CM = \bigsqcup_{c_2 = - \infty}^{\infty} \CM_{(r,c_1,c_2)}
$$
Note that the Bogomolov inequality implies that $c_2$ is bounded below by $\frac {r-1}{2r} c_1^2$. \\

\noindent We will write $\CU$ for the universal sheaf on $\CM \times S$ obtained as the disjoint union of the universal sheaves \eqref{eqn:universal} over all $c_2$; we assume that the latter are compatible with each other as $c_2$ varies, as explained in \cite[Subsection 5.9]{Univ}. This implies that the moduli spaces of flags of sheaves that we will introduce in the next Subsection also have universal sheaves which are contained inside each other in the obvious way. \\

\noindent Because the universal sheaf $\CU$ is flat over $\CM$, it inherits certain properties from the stable sheaves it parameterizes, such as having homological dimension 1. Indeed, any stable sheaf of rank $r>0$ is torsion free, and any torsion free sheaf on a smooth projective surface has homological dimension 1, see \cite[Example 1.1.16]{HL}. The fact that this also holds in families is made explicit by the following result. \\
	
	\begin{proposition}
		\label{prop:length 1}
		
		(\cite{Univ}) There exists a short exact sequence
		\begin{equation}
		\label{eqn:length 1}
		0 \rightarrow \CW \rightarrow \CV \rightarrow \CU \rightarrow 0
		\end{equation}
		with $\CW$ and $\CV$ locally free sheaves on $\CM \times S$. In fact, we can take
		\begin{equation}
		\label{eqn:we can take}
		\CV = \pi_1^* \Big[ \pi_{1*} \Big( \CU \otimes \pi_2^*(\CO(nH) ) \Big) \Big] \otimes \pi_2^*(\CO(-n H))
		\end{equation}
		for $n \gg 0$, where $\pi_1 :  \CM \times S \rightarrow \CM$ and $\pi_2 :  \CM \times S \rightarrow S$ are the standard projections. \\
		
	\end{proposition}
	
	\subsection{}
	\label{sub:schemes}
	
	A \textbf{set partition} is an equivalence relation on a finite ordered set. We will represent set partitions symbolically, for example $(x,y,z)$ will refer to the partition of a 3-element set into distinct 1-element subsets, while $(x,y,x)$ (respectively $(x,x,x)$) refers to the equivalence relation which sets the first and the last element (respectively all elements) equivalent to each other. The size of a partition $\lambda$, which is denoted by $|\lambda|$, is the number of elements of the underlying set. \\
	
	\begin{definition}
		\label{def:z}
		
		For a set partition $\lambda$ of size $n$, we will consider the scheme
		\begin{multline}
		\fZ_\lambda = \Big \{(\CF_0 \subset_{x_1} \CF_1 \subset_{x_2} \dots \subset_{x_n} \CF_n) \text{ stable coherent sheaves} \\
		\text{ for some }x_1,\dots,x_n \in S \text{ such that } x_i = x_j \text{ if } i \sim j \text{ in } \lambda \Big\} \label{eqn:z}
		\end{multline}
		where $\CF' \subset_x \CF$ means that $\CF' \subset \CF$ and $\CF/\CF' \cong \BC_x$. If $n = 0$, we write $\fZ_\emptyset = \CM$. \\
		
	\end{definition}
	
	\noindent We will often write
	$$
	\fZ_n = \fZ_{(x_1,\dots,x_n)} \qquad \text{and} \qquad \fZ_n^\bullet = \fZ_{(x,\dots,x)}
	$$
	Definition \ref{def:z} should be read as ``the scheme $\fZ_\lambda$ represents the functor"
	\begin{multline}
	\label{eqn:functor z}
	\Maps(T, \fZ_\lambda) \xleftrightarrow{1\text{-to-}1} \Big \{ \text{flags of coherent sheaves } (\CF_0 \subset \CF_1 \subset \dots \subset \CF_n) \text{ on } T \times S, \ \ \\ \text{ flat over }T, \text{ such that } \CF_{0,t},\dots, \CF_{n,t} \text{ are stable with invariants }(r,c_1), \forall t\in T, \\ \text{together with the data in (a) and (b) below} \Big\} \Big/\sim  \qquad \qquad
	\end{multline} 
where $\sim$ is the equivalence relation induced by tensoring $\CF_0,\dots,\CF_n$ by one and the same line bundle pulled back from $T$. In the formula above, (a) and (b) refers to \\
	
	\begin{enumerate}
		\item[(a)] maps $x_1,\dots,x_n : T \rightarrow S$ such that $x_i = x_j$ if $i \sim j$ in $\lambda$, \\
		
		\item[(b)] line bundles $\CL_1,\dots,\CL_n$ on $T$ such that $\CF_i/\CF_{i-1} \cong \Gamma^{i}_*(\CL_i)$, where $\Gamma^i : T \hookrightarrow T \times S$ denotes the graph of the map $x_i : T \rightarrow S$ from (a) \\
	\end{enumerate}
	
	\noindent Proposition \ref{prop:tower} will establish (by induction on $n$) that the functor \eqref{eqn:functor z} is indeed representable, which implies the existence of the scheme stipulated in Definition \ref{def:z}. However, to keep our notation simple, we will denote points of this scheme as \eqref{eqn:z} instead of the more complicated notation \eqref{eqn:functor z}. In other words, all the constructions from now on will be done at the level of closed points, and the family versions are left as exercises to the interested reader. \\ 
	
	\subsection{} If $\lambda$ is a partition on an $n$ element ordered set, then we will use the notation
	$$
	|\lambda \qquad \text{and} \qquad \lambda|
	$$
	for the partition on the $n-1$ element ordered set obtained by dropping the first (respectively last) element of $\lambda$. Then we may consider the maps
	\begin{equation}
	\label{eqn:projection maps}
	\xymatrix{
		\fZ_\lambda \ar[d]_{\pi_-} \\
		\fZ_{|\lambda} \times S^\#} \qquad \text{and} \qquad \xymatrix{
		\fZ_\lambda \ar[d]^{\pi_+} \\
		\fZ_{\lambda|} \times S^\#}
	\end{equation}
	given by forgetting the first (respectively last) sheaf in the flag \eqref{eqn:z}. The number $\#$ is 0 or 1 depending on whether the first (respectively last) element of $\lambda$ is or is not equivalent to some other element of $\lambda$. Therefore, the map $\pi_-$ (resp. $\pi_+$) remembers the point $x_1 \in S$ (resp. $x_n \in S$) if and only if $\# = 1$. For example:
	$$
	\xymatrix{
		\fZ_{(x,y,x,z)} \ar[d]_{\pi_-} \\
		\fZ_{(y,x,z)}} \qquad \text{but} \qquad \xymatrix{
		\fZ_{(y,x,x,z)} \ar[d]_{\pi_-} \\
		\fZ_{(x,x,z)} \times S}
	$$
	As is clear from the example above, in the case of the arrow on the left, there is no reason to remember the point $x \in S$ ``forgotten" from the set partition $\lambda = (x,y,x,z)$, since it can be recovered from $|\lambda = (y,x,z)$. In other words, our convention when $\# = 0$ is  that $\pi_-$ (resp. $\pi_+$) composed with the graph of the map that remembers $x_1$ (resp. $x_n$) yields the analogous map $\fZ_\lambda \rightarrow \fZ_{|\lambda} \times S$ (resp. $\fZ_\lambda \rightarrow \fZ_{\lambda|} \times S$) to the one we would have defined in the $\# = 1$ case. \\
		
	\subsection{}
	
	There exists a natural map $p_i : \fZ_\lambda \rightarrow \CM$ which only remembers the sheaf $\CF_i$ in the flag \eqref{eqn:z}. We will write $\CU_i = (p_i \times \text{Id}_S)^* (\CU)$, and note that Proposition \ref{prop:length 1} implies that there exists an exact sequence
	\begin{equation}
	\label{eqn:length 1 non flat}
	0 \rightarrow \CW_i \rightarrow \CV_i \rightarrow \CU_i \rightarrow 0 
	\end{equation}
	on $\fZ_\lambda \times S$, where $\CW_i$ and $\CV_i$ are locally free. While the map $p_i$ is not flat, \eqref{eqn:length 1 non flat} is short exact because all three of the coherent sheaves in \eqref{eqn:length 1} are flat over $\CM$. \\
	
	\noindent In the following Proposition, we consider any set partition $\lambda$ of size $n$, and let $\fZ_\lambda$ denote the moduli space \eqref{eqn:z}, whose points will be denoted by $\CF_0 \subset \dots \subset \CF_n$. \\
	
	\begin{proposition}
		\label{prop:tower}
		
The map $\pi_-$ of \eqref{eqn:projection maps} can be realized as the diagonal arrow in
		\begin{align}
		&\xymatrix{
			\fZ_\lambda \ar[rd]_{\pi_-} \ar@{^{(}->}[r] & \BP_{\fZ_{|\lambda} \times S} (\CV_1) \ar[d]^\rho \\
			& \fZ_{|\lambda} \times S} & &\text{if } \# = 1, \text{ and} \label{eqn:pi minus v1} \\
		&\xymatrix{
			\fZ_\lambda \ar[rd]_{\pi_-} \ar@{^{(}->}[r] & \BP_{\fZ_{|\lambda}} (\Gamma^{x*} (\CV_1)) \ar[d]^\rho \\
			& \fZ_{|\lambda}} & &\text{if } \# = 0 \label{eqn:pi minus v2}
		\end{align}
		(in the latter case, $\Gamma^x : \fZ_{|\lambda} \rightarrow \fZ_{|\lambda} \times S$ is the graph of the function $\fZ_{|\lambda} \rightarrow S$ that remembers the support point we forget when going from $\lambda$ to $|\lambda$) where $\hookrightarrow$ is the closed embedding cut out by the following composition of maps of locally free sheaves on the projectivization:
\begin{align}
&\rho^*(\CW_1) \rightarrow \rho^*(\CV_1) \stackrel{\emph{taut}}\twoheadrightarrow \CO(1) & &\text{if } \# = 1, \text{ and} \label{eqn:section minus 1} \\
&\rho^*(\Gamma^{x*}(\CW_1)) \rightarrow \rho^*(\Gamma^{x*}(\CV_1)) \stackrel{\emph{taut}}\twoheadrightarrow \CO(1) & &\text{if } \# = 0 \label{eqn:section minus 0} 
\end{align}
The line bundle $\CL_1$ on $\fZ_\lambda$ is the restriction of $\CO(1)$ from the projectivization. \\
		
		\noindent Similarly, the map $\pi_+$ of \eqref{eqn:projection maps} can be realized as the diagonal arrow in
		\begin{align}
		&\xymatrix{
			\fZ_\lambda \ar[rd]_{\pi_+} \ar@{^{(}->}[r] & \BP_{\fZ_{\lambda|} \times S} (\CW^\vee_n \otimes \omega_S) \ar[d]^\rho \\
			& \fZ_{\lambda|} \times S} & &\text{if } \# = 1, \text{ and} \label{eqn:pi plus v1} \\
		&\xymatrix{
			\fZ_\lambda \ar[rd]_{\pi_+} \ar@{^{(}->}[r] & \BP_{\fZ_{\lambda|}} (\Gamma^{x*} (\CW_n^\vee \otimes \omega_S)) \ar[d]^\rho \\
			& \fZ_{\lambda|}} & &\text{if } \# = 0 \label{eqn:pi plus v2}
		\end{align}
		(we write $\omega_S$ for the canonical line bundle on $S$, and its various pullbacks) where $\hookrightarrow$ is the closed embedding cut out by the following composition of maps of locally free sheaves on the projectivization:
\begin{align}
& \rho^*(\CV_n^\vee \otimes \omega_S) \rightarrow  \rho^*(\CW^\vee_n \otimes \omega_S) \stackrel{\emph{taut}}\twoheadrightarrow \CO(1) & &\text{if } \# = 1, \text{ and}\label{eqn:section plus 1} \\
& \rho^*(\Gamma^{x*}(\CV^\vee_n \otimes \omega_S)) \rightarrow  \rho^*(\Gamma^{x*}(\CW^\vee_n \otimes \omega_S)) \stackrel{\emph{taut}}\twoheadrightarrow \CO(1) & &\text{if } \# = 0 \label{eqn:section plus 0}
\end{align}
The line bundle $\CL_n$ on $\fZ_\lambda$ is the restriction of $\CO(-1)$ from the projectivization. \\
		
	\end{proposition}
	
	\begin{proof} When $\lambda$ has size 1, the Proposition was proved in \cite[Definition 2.5, Proposition 2.8, Proposition 2.10]{Univ}, and the general case will follow the same logic. We will prove the statements pertaining to $\pi_-$, and leave those pertaining to $\pi_+$ as exercises to the interested reader. Moreover, we will only prove the case $\# = 1$, as $\# = 0$ is analogous. We interpret the Proposition as follows: having constructed $\fZ_{|\lambda}$ which represents the functor \eqref{eqn:functor z} for the partition $|\lambda$, we must show that the scheme $\fZ_\lambda$ defined as the closed embedding \eqref{eqn:pi minus v1} represents the functor \eqref{eqn:functor z} for the partition $\lambda$. \\
		
		\noindent Recall that a map $\phi : T \rightarrow \fZ_{|\lambda} \times S$ consists of a flag 
		\begin{equation}
		\label{eqn:america}
		\CF_1 \subset \CF_2 \subset \dots \subset \CF_n
		\end{equation}
of coherent sheaves on $T \times S$ which are flat over $T$, together with maps 
$$
x_1,x_2,\dots,x_n : T \rightarrow S
$$
and line bundles $\CL_2,\dots,\CL_n$ on $T$ such that $\CF_i/\CF_{i-1} \cong \Gamma^{i}_*(\CL_i)$ for all $i \in \{2,\dots,n\}$, where $\Gamma^i : T \hookrightarrow T \times S$ denotes the graph of $x_i$. We write $x_1$ for the map $T \rightarrow S$ arising from the second factor of $\fZ_{|\lambda} \times S$ and $\Gamma^1$ for the graph of $x_1$. Therefore, we may write
		$$
		\phi = (\bar{\phi} \times \text{Id}_S) \circ \Gamma^{1}
		$$
		where $\bar{\phi} : T \rightarrow \fZ_{|\lambda}$ represents the flag \eqref{eqn:america}. Then a map
		$$
		T \rightarrow \BP_{\fZ_{|\lambda} \times S}(\CV_1)
		$$
		consists of the data of $\bar{\phi}$ and $x_1$ as above, together with a line bundle $\CL_1$ on $T$ and a surjective homomorphism
		$$
		\tau : \phi^*(\CV_1) \twoheadrightarrow \CL_1
		$$
		Finally, a map from $T$ to the subscheme $\fZ_\lambda$ of \eqref{eqn:pi minus v1} consists of $\bar{\phi}$, $x_1$, $\CL_1$ and $\tau$ as above, such that the following composition vanishes:
		$$
		\phi^*(\CW_1) \rightarrow \phi^*(\CV_1) \stackrel{\tau}\twoheadrightarrow \CL_1
		$$
		Since pull-back is right-exact, this is equivalent to a surjective map 
		$$
		\phi^*(\CU_1) \twoheadrightarrow \CL_1 \qquad \Leftrightarrow \qquad \Gamma^{1*} \circ (\bar{\phi} \times \text{Id}_S)^*(\CU_1) \twoheadrightarrow \CL_1 \qquad \Leftrightarrow \qquad \Gamma^{1*}(\CF_1) \twoheadrightarrow \CL_1
		$$
		By adjunction, this datum is equivalent to a surjection $\CF_1 \twoheadrightarrow \Gamma^{1}_*(\CL_1)$ on $T \times S$. Letting $\CF_0$ be the kernel of this surjection, this precisely completes the flag \eqref{eqn:america} to the flag \eqref{eqn:functor z}. By \cite[Proposition 5.5]{Univ}, $\CF_0$ is stable if and only if $\CF_1$ is stable.
		
	\end{proof}
	
	\subsection{} 
	\label{sub:tor}
	
	Consider an arbitrary set partition $\lambda$ of size $n$ and the scheme $\fZ_\lambda$ of Definition \ref{def:z}. Let $\pi_1 : \fZ_\lambda \times S \rightarrow \fZ_\lambda$ and $\pi_2 : \fZ_\lambda \times S \rightarrow S$ denote the standard projections. For each $i \in \{1,\dots,n\}$, we have the following short exact sequence on $\fZ_\lambda \times S$:
	\begin{equation}
	\label{eqn:small ses}
	0 \rightarrow \CU_{i-1} \rightarrow \CU_i \rightarrow \CA_i \otimes \CO_{\Gamma^i} \rightarrow 0
	\end{equation}
	where $\CA_i = \pi_1^*(\CL_i)$, and $\Gamma^i \subset \fZ_\lambda \times S$ denotes the graph of the map 
	$$
	p_S^i : \fZ_\lambda \rightarrow S
	$$
	that remembers the $i$--th support point $x_i$. If $i$ and $j$ are equivalent elements of the set partition $\lambda$, then $\Gamma^i = \Gamma^j$. We may upgrade \eqref{eqn:small ses} to the following commutative diagram of sheaves on $\fZ_\lambda \times S$ with all rows and columns exact:
	\begin{equation}
	\label{eqn:big diagram}
	\xymatrix{ & 0 & 0 & 0  & \\
		0 \ar[r] & \CU_{i-1} \ar[u] \ar[r] & \CU_i \ar[u] \ar[r] & \CB_i \otimes \CO_{\Gamma^i} \ar[u] \ar[r] & 0 \\
		0 \ar[r] & \CV_{i-1} \ar[u] \ar[r] & \CV_i \ar[u] \ar[r] & \CB_i \ar[u] \ar[r] & 0 \\ 
		0 \ar[r] & \CW_{i-1} \ar[u] \ar[r] & \CW_i \ar[u] \ar[r] & \CB_i \otimes \CI_{\Gamma^i} \ar[u] \ar[r] & 0 \\
		& 0 \ar[u] & 0 \ar[u] & 0 \ar[u] & }
	\end{equation}
	where $\CI_{\Gamma_i}$ denotes the ideal of the graph $\Gamma^i$ inside $\fZ_\lambda \times S$, and the line bundle
	\begin{equation}
	\label{eqn:new line bundle}
	\CB_i = \pi_1^* \Big (\CL_i \otimes p_S^{i*} (\CO(nH)) \Big) \otimes \pi_2^*(\CO(-nH))
	\end{equation}
	is isomorphic to the quotient $\CV_i/\CV_{i-1}$ due to \eqref{eqn:we can take}. Note that $\CB_i|_{\Gamma^i} \cong \CA_i|_{\Gamma^i} \cong \CL_i$. In the Proposition below, we will restrict the diagram \eqref{eqn:big diagram} to $\Gamma^1$ and $\Gamma^n$. \\
	
\begin{proposition}
\label{prop:tor}
		
Let $\lambda$ be a set partition of size $n$. \\
		
\begin{enumerate}[leftmargin=*]

\item[$(-)$] If $p_S^1$ is flat, then there exists an injective map of sheaves $\CL_1 \otimes p_S^{1*}(\omega_S) \hookrightarrow \CW_0|_{\Gamma^1}$ with locally free quotient, such that the following composition vanishes:
\begin{equation}
\label{eqn:finally 1}
\CL_1 \otimes p_S^{1*}(\omega_S) \hookrightarrow \CW_0|_{\Gamma^1} \rightarrow \CV_0|_{\Gamma^1} 
\end{equation}

\item [$(+)$] If $p_S^n$ is flat, then there exists an injective map of sheaves $\CL_n^{-1}  \hookrightarrow (\CV_n|_{\Gamma^n})^\vee$ with locally free quotient, such that the following composition vanishes:
\begin{equation}
\label{eqn:finally 2}
\CL_n^{-1}  \hookrightarrow (\CV_n|_{\Gamma^n})^\vee \rightarrow (\CW_n|_{\Gamma^n})^\vee
\end{equation}
	
\end{enumerate}		

\noindent In both formulas above, the second arrows are induced by \eqref{eqn:length 1 non flat}. \\

\end{proposition}
	
	\begin{remark}
	
The flatness of $p_S^1,\dots,p_S^n$ is proved in Proposition \ref{prop:flat morphism} for any $\lambda$ of size $n\leq 3$, but we expect it to hold for any set partition $\lambda$. \\

	\end{remark}
	
	\begin{proof} Let us first prove $(-)$. Upon restricting \eqref{eqn:big diagram} (for $i=1$) to $\Gamma^1 : \fZ_\lambda \hookrightarrow \fZ_\lambda \times S$, we obtain the following diagram of Tor sheaves with exact rows and columns: 
\begin{equation}
\label{eqn:tor diagram 1}
		\xymatrix{& \CU_{0}|_{\Gamma^1} \ar[r] & \CU_1|_{\Gamma^1} \ar@{->>}[r] & \CB_1|_{\Gamma^1} \\
			& \CV_{0}|_{\Gamma^1} \ar@{^{(}->}[r] \ar@{->>}[u] & \CV_1|_{\Gamma^1} \ar@{->>}[r] \ar@{->>}[u] & \CB_1|_{\Gamma^1} \ar@{=}[u] \\ 
			\CB_1 \otimes \Tor_1(\CI_{\Gamma^1},\CO_{\Gamma^1}) \ar@{^{(}->}[r] & \CW_{0}|_{\Gamma^1} \ar[u] \ar[r] & \CW_1|_{\Gamma^1} \ar[u] \ar@{->>}[r] & \CB_1 \otimes \CI_{\Gamma^1}|_{\Gamma^1} \ar[u]^0 \\
			\CB_1 \otimes \Tor_2(\CO_{\Gamma^1},\CO_{\Gamma^1})  \ar@{=}[u] \ar@{^{(}->}[r] & \Tor_1(\CU_{0},\CO_{\Gamma^1}) \ar@{^{(}->}[u] \ar[r] & \Tor_1(\CU_1,\CO_{\Gamma^1}) \ar@{^{(}->}[u] \ar[r] & \CB_1 \otimes \Tor_1(\CO_{\Gamma^1},\CO_{\Gamma^1}) \ar@{=}[u]}
\end{equation}
The injectivity of the maps denoted by $\hookrightarrow$ in the diagram above is due to the fact that $\CW_1, \CV_0,\CV_1, \CB_1$ are locally free, except for the injectivity of the horizontal map in the bottom left corner, which follows from the fact that $\CU_1$ has homological dimension 1 (see \eqref{eqn:length 1 non flat}). We have the following fiber square:
		$$
		\xymatrix{ \fZ_\lambda  \ar@{^{(}->}[r]^-{\Gamma^1} \ar[d]_-{p_S^1} & \fZ_\lambda \times S \ar[d]^{p_S^1 \times \text{Id}_S} \\
			S \ar[r]^-{\Delta} & S \times S}
		$$
		The assumption that $p_S^1$ is a flat morphism implies that
		$$
		\Tor_2(\CO_{\Gamma^1}, \CO_{\Gamma^1}) = p_S^{1*} \left( \Tor_2(\CO_\Delta, \CO_\Delta) \right) = p_S^{1*}(\omega_S)
		$$
		on $\fZ_\lambda$. The latter equality in the equation above is a standard and straightforward exercise, which follows from the fact that $S \hookrightarrow S \times S$ is a smooth embedding of codimension 2. If we recall the fact that $\CB_1|_{\Gamma^1} \cong \CL_1$, then we conclude that
		\begin{equation}
		\label{eqn:desert}
		\CB_1 \otimes \Tor_2(\CO_{\Gamma^1}, \CO_{\Gamma^1}) \cong \CL_1 \otimes p_S^{1*}(\omega_S)
		\end{equation}
		Since $\CB_1 \otimes \Tor_2(\CO_{\Gamma^1}, \CO_{\Gamma^1})$ injects into the kernel of $\CW_0|_{\Gamma^1} \rightarrow \CV_0|_{\Gamma_1}$ (as can be seen from diagram \eqref{eqn:tor diagram 1}), this establishes the vanishing of the composition \eqref{eqn:finally 1}. The fact that the sheaf
		$$
		\frac {\CW_0|_{\Gamma^1}}{\CL_1 \otimes p_S^{1*}(\omega_S)} \cong \Ker \Big(\CW_1|_{\Gamma^1} \twoheadrightarrow \CB_1 \otimes \CI_{\Gamma^1}|_{\Gamma^1} \Big)
		$$
		is locally free follows from the facts that $\CW_1$ is locally free, and $\CI_{\Gamma^1}|_{\Gamma^1}$ has homological dimension 1 (since it is the pull-back of the defining ideal of the codimension 2 regular embedding $\Delta : S \hookrightarrow S \times S$ under the flat morphism $p_S^{1*}$). \\
		
		\noindent To deal with the $(+)$ case, let us restrict \eqref{eqn:big diagram} (for $i=n$) to $\Gamma^n$ and then dualize:
		\begin{equation}
		\label{eqn:tor diagram 2}
		\xymatrix{
			(\CV_{n-1}|_{\Gamma^n})^\vee \ar[d] & (\CV_n|_{\Gamma^n})^\vee \ar@{->>}[l] \ar[d] & (\CB_n|_{\Gamma^n})^\vee \ar@{_{(}->}[l] \ar[d] \\ 
			(\CW_{n-1}|_{\Gamma^n})^\vee & (\CW_n|_{\Gamma^n})^\vee \ar[l] & (\CB_n \otimes \CI_{\Gamma^n}|_{\Gamma^n})^\vee \ar[l]}
		\end{equation}
		The top-most row is exact, due to the fact that restriction and dualization preserve short exact sequences of locally free sheaves. Because $p_S^n : \fZ_\lambda \rightarrow S$ is flat, we have
		$$
		\CI_{\Gamma^n}|_{\Gamma^n} \cong p_S^{n*}(\Omega_S^1) \qquad \Rightarrow \qquad (\CI_{\Gamma^n}|_{\Gamma^n})^\vee \cong  p_S^{n*}(\CT_S^1)
		$$
		and the right-most vertical map in \eqref{eqn:tor diagram 2} is zero. Therefore, the rank 1 sub-bundle $(\CB_n|_{\Gamma^n})^\vee \cong \CL_n^{-1}$ embeds in the kernel of $(\CV_n|_{\Gamma^n})^\vee \rightarrow (\CW_n|_{\Gamma^n})^\vee$, thus implying the vanishing of the composition \eqref{eqn:finally 2}. Moreover, the quotient
$$
\frac {(\CV_n|_{\Gamma^n})^\vee}{\CL_n^{-1}} \cong (\CV_{n-1}|_{\Gamma^n})^\vee
$$
is locally free by construction.
		
	\end{proof}
	
\begin{remark} 

A straightforward analogue of the proof given above shows that statement $(-)$ (respectively $(+)$) of Proposition \ref{prop:tor} would still hold with the numbers $0,1$ (respectively $n-1,n$) replaced by $i-1,i$, for any $i \in \{1,\dots,n\}$. \\

\end{remark}	
	
\subsection{} 
	
We will now state certain geometric properties of several schemes $\fZ_\lambda$ with $n = |\lambda| \leq 4$, to be proved in Section \ref{sec:geom}. Let $k$ be the number of distinct elements of $\lambda$. \\

\begin{definition}
\label{def:exp dim}

The expected dimension of (a connected component of) $\fZ_\lambda$ is 
$$
\econst + c_2(\CF_0) + c_2(\CF_n) + k
$$
for any $(\CF_0 \subset_{x_1} \dots \subset_{x_n} \CF_n) \in \fZ_\lambda$. \\

\end{definition}

\noindent Moreover, consider the map $\pi : \fZ_\lambda \rightarrow \CM \times S^n$ which takes $(\CF_0 \subset_{x_1} \dots \subset_{x_n} \CF_n)$ to $(\CF_n, x_1,\dots,x_n)$. In what follows, whenever we say that $\fZ_\lambda$ is irreducible, we actually mean that $\pi^{-1}(C)$ is irreducible for every connected component $C$ of $\CM \times S^n$. \\

\begin{proposition}
\label{prop:x}

$\fZ_{(x)}$ is smooth and irreducible of expected dimension. \\
		
\end{proposition}	

\begin{proposition}
\label{prop:xx}

$\fZ_{(x,x)}$ is smooth and irreducible of expected dimension. Moreover, the following fiber square is derived with excess:
\begin{equation}
\label{eqn:square xx}
		\xymatrix{\fZ_{(x,x)} \ar[r]^-{\pi_-} \ar[d]_-{\pi_+} & \fZ_{(x)} \ar[d]  \\
			\fZ_{(x)} \ar[r] & \CM \times S} \qquad \qquad \xymatrix{\CF_0 \subset_x \CF_1 \subset_x \CF_2 \ar@{|->}[r] \ar@{|->}[d] & \CF_1 \subset_x \CF_2 \ar@{|->}[d]  \\
			\CF_0 \subset_x \CF_1 \ar@{|->}[r] & (\CF_1,x)}
\end{equation}
The excess bundle is $\CL_2 \otimes \CL_1^{-1} \otimes p_S^*(\omega_S)$, where $\fZ_{(x,x)} \xrightarrow{p_S} S$ is the map that records $x$. \\
		
\end{proposition}
	
\begin{proposition}
\label{prop:xy}

$\fZ_{(x,y)}$ is l.c.i. and irreducible of expected dimension. Moreover, the following fiber square is derived:
\begin{equation}
\label{eqn:square xy}
		\xymatrix{\fZ_{(x,y)} \ar[r] \ar[d] & \fZ_{(y)} \times S \ar[d]  \\
			\fZ_{(x)} \times S \ar[r] & \CM \times S \times S} \qquad \xymatrix{\CF_0 \subset_x \CF_1 \subset_y \CF_2 \ar@{|->}[r] \ar@{|->}[d] & (\CF_1 \subset_y \CF_2,x) \ar@{|->}[d]  \\
			(\CF_0 \subset_x \CF_1,y) \ar@{|->}[r] & (\CF_1,x,y)}
\end{equation}
		
	\end{proposition}

\medskip
	
\begin{proposition}
\label{prop:xxx}

$\fZ_{(x,x,x)}$ is l.c.i. and irreducible of expected dimension. Moreover, the following fiber square is derived:
$$
		\xymatrix{\fZ_{(x,x,x)} \ar[r] \ar[d] & \fZ_{(x,x)} \ar[d]  \\
			\fZ_{(x,x)} \ar[r] & \fZ_{(x)}} \qquad \qquad \xymatrix{\CF_0 \subset_x \CF_1 \subset_x \CF_2 \subset_x \CF_3 \ar@{|->}[r] \ar@{|->}[d] & \CF_1 \subset_x \CF_2 \subset_x \CF_3 \ar@{|->}[d]  \\
			\CF_0 \subset_x \CF_1 \subset_x \CF_2 \ar@{|->}[r] & \CF_1 \subset_x \CF_2}
		$$
		
\end{proposition}

\medskip
	
\begin{proposition}
\label{prop:xxy and yxx}

$\fZ_{(x,x,y)}$ and $\fZ_{(y,x,x)}$ are l.c.i. and irreducible of expected dimension. Moreover, the following fiber squares are derived:
$$
\xymatrix{\fZ_{(x,x,y)} \ar[r] \ar[d] & \fZ_{(y)} \times S \ar[d]  \\
			\fZ_{(x,x)} \times S \ar[r] & \CM \times S \times S} \qquad \qquad \xymatrix{\CF_0 \subset_x \CF_1 \subset_x \CF_2 \subset_y \CF_3 \ar@{|->}[r] \ar@{|->}[d] & (\CF_2 \subset_y \CF_3,x) \ar@{|->}[d]  \\
			(\CF_0 \subset_x \CF_1 \subset_x \CF_2,y) \ar@{|->}[r] & (\CF_2, x, y)}
$$
$$
		\xymatrix{\fZ_{(y,x,x)} \ar[r] \ar[d] & \fZ_{(x,x)} \times S \ar[d]  \\
			\fZ_{(y)} \times S \ar[r] & \CM \times S \times S} \qquad \xymatrix{\CF_0 \subset_y \CF_1 \subset_x \CF_2 \subset_x \CF_3 \ar@{|->}[r] \ar@{|->}[d] & (\CF_1 \subset_x \CF_2 \subset_x \CF_3,y) \ar@{|->}[d]  \\
			(\CF_0 \subset_y \CF_1,x) \ar@{|->}[r] & (\CF_1, x, y)}
$$
		
\end{proposition}	

\medskip

\begin{proposition}
\label{prop:xyx}

$\fZ_{(x,y,x)}$ is Cohen-Macaulay and irreducible of expected dimension. \\
		
\end{proposition}
	
\begin{proposition}
\label{prop:xxxx}

$\fZ_{(x,x,x,x)}$ is l.c.i. and has two irreducible components of expected dimension. Moreover, the following fiber squares are derived:
$$
\xymatrix{\fZ_{(x,x,x,x)} \ar[r] \ar[d] & \fZ_{(x,x,x)} \ar[d]  \\
			\fZ_{(x,x)} \ar[r] & \fZ_{(x)}} \xymatrix{\CF_0 \subset_x \CF_1 \subset_x \CF_2 \subset_x \CF_3 \subset_x \CF_4 \ar@{|->}[r] \ar@{|->}[d] & \CF_1 \subset_x \CF_2 \subset_x \CF_3 \subset_x \CF_4 \ar@{|->}[d]  \\
			\CF_0 \subset_x \CF_1 \subset_x \CF_2 \ar@{|->}[r] & \CF_1 \subset_x \CF_2}
$$
$$
\xymatrix{\fZ_{(x,x,x,x)} \ar[r] \ar[d] & \fZ_{(x,x)} \ar[d]  \\
			\fZ_{(x,x,x)} \ar[r] & \fZ_{(x)}} \qquad \xymatrix{\CF_0 \subset_x \CF_1 \subset_x \CF_2 \subset_x \CF_3 \subset_x \CF_4 \ar@{|->}[r] \ar@{|->}[d] & \CF_2 \subset_x \CF_3 \subset_x \CF_4 \ar@{|->}[d]  \\
			\CF_0 \subset_x \CF_1 \subset_x \CF_2 \subset_x \CF_3\ar@{|->}[r] & \CF_2 \subset_x \CF_3}
$$
	
\end{proposition}
	
\begin{proposition}
\label{prop:xxyx and xyxx}

$\fZ_{(x,x,y,x)}$ and $\fZ_{(x,y,x,x)}$ are Cohen-Macaulay and irreducible of expected dimension. Moreover, the following fiber squares are derived:
		$$
		\xymatrix{\fZ_{(x,x,y,x)} \ar[r] \ar[d] & \fZ_{(x,y,x)} \ar[d]  \\
			\fZ_{(x,x)} \ar[r] & \fZ_{(x)}} \xymatrix{\CF_0 \subset_x \CF_1 \subset_x \CF_2 \subset_y \CF_3 \subset_x \CF_4 \ar@{|->}[r] \ar@{|->}[d] & \CF_1 \subset_x \CF_2 \subset_y \CF_3 \subset_x \CF_4 \ar@{|->}[d]  \\
			\CF_0 \subset_x \CF_1 \subset_x \CF_2 \ar@{|->}[r] & \CF_1 \subset_x \CF_2}
		$$
		$$
		\xymatrix{\fZ_{(x,y,x,x)} \ar[r] \ar[d] & \fZ_{(x,x)} \ar[d]  \\
			\fZ_{(x,y,x)} \ar[r] & \fZ_{(x)}} \qquad \xymatrix{\CF_0 \subset_x \CF_1 \subset_y \CF_2 \subset_x \CF_3 \subset_x \CF_4 \ar@{|->}[r] \ar@{|->}[d] & \CF_2 \subset_x \CF_3 \subset_x \CF_4 \ar@{|->}[d]  \\
			\CF_0 \subset_x \CF_1 \subset_y \CF_2 \subset_x \CF_3 \ar@{|->}[r] & \CF_2 \subset_x \CF_3}
		$$
		
\end{proposition}

\bigskip
	
\noindent Finally, we will need the normality of some of the varieties $\fZ_\lambda$. \\

	\begin{proposition}
		\label{prop:normal}
		
		The scheme $\fZ_\lambda$ is normal for any
		\begin{equation}
		\label{eqn:lambda}
		\lambda \in \Big\{ (x,y), (x,x,y), (x,y,x), (y,x,x), (x,x,y,x), (x,y,x,x) \Big\}
		\end{equation}
		
	\end{proposition}
	
	\medskip	
	
	\subsection{}
	\label{sub:def y}
	
	Let us consider the spaces $\fY$, $\fY_-$, $\fY_+$, $\fY_{-+}$ which parameterize diagrams
	\begin{equation}
	\label{eqn:y}
	\xymatrix{& \CF_1 \ar@{^{(}->}[rd]^{y} & \\
		\CF_0 \ar@{^{(}->}[ru]^{x} \ar@{^{(}->}[rd]_{y}
		& & \CF_2 \\
		& \CF_1' \ar@{^{(}->}[ru]_{x} &} 
	\end{equation}
	\begin{equation}
	\label{eqn:y minus}
	\xymatrix{& \CF_1 \ar@{^{(}->}[rd]^{y} & & \\
		\CF_0 \ar@{^{(}->}[ru]^{x} \ar@{^{(}->}[rd]_{y}
		& & \CF_2 \ar@{^{(}->}[r]^{x} & \CF_3 \\
		& \CF_1' \ar@{^{(}->}[ru]_{x} & &} 
	\end{equation}
	\begin{equation}
	\label{eqn:y plus}
	\xymatrix{& & \CF_2 \ar@{^{(}->}[rd]^{y} & \\
		\CF_0 \ar@{^{(}->}[r]^{x} & \CF_1 \ar@{^{(}->}[ru]^{x} \ar@{^{(}->}[rd]_{y}
		& & \CF_3 \\
		& & \CF_2' \ar@{^{(}->}[ru]_{x} &} 
	\end{equation}
	\begin{equation}
	\label{eqn:y bullet}
	\xymatrix{& & \CF_2 \ar@{^{(}->}[rd]^{y} & & \\
		\CF_0 \ar@{^{(}->}[r]^{x} & \CF_1 \ar@{^{(}->}[ru]^{x} \ar@{^{(}->}[rd]_{y}
		& & \CF_3 \ar@{^{(}->}[r]^{x} & \CF_4 \\
		& & \CF_2' \ar@{^{(}->}[ru]_{x} & &} 
	\end{equation}
	respectively, of stable coherent sheaves where each successive inclusion is colength 1 and supported at the point indicated on the diagram. Strictly speaking, $\fY$, $\fY_-$, $\fY_+$, $\fY_{-+}$ are functors which associate to a scheme $T$ diagrams of flat families of stable coherent sheaves \eqref{eqn:y}, \eqref{eqn:y minus}, \eqref{eqn:y plus}, \eqref{eqn:y bullet} on $T \times S$, satisfying all the standard properties. On these schemes, we have the line bundles with fibers
	$$
	\CL_i = \Gamma(S,\CF_i/\CF_{i-1})
	$$
	and where the notation is applicable, line bundles
	$$
	\CL_i' = \Gamma(S,\CF_i'/\CF_{i-1}) \quad \text{or} \quad \CL_i' = \Gamma(S,\CF_i/\CF_{i-1}')
	$$
	Note that
	\begin{align}
	&\CL_1\CL_2 = \CL_1'\CL_2' \in \text{Pic}(\fY), \text{Pic}(\fY_-) \label{eqn:picard relation 1} \\
	&\CL_2\CL_3 = \CL_2'\CL_3' \in \text{Pic}(\fY_+), \text{Pic}(\fY_{-+}) \label{eqn:picard relation 2}
	\end{align}
	
	\subsection{}
	
	Consider the maps
	\begin{equation}
	\label{eqn:pi top}
	\xymatrix{\fY \ar@{.>}[d]^-{\pi^\uparrow} \\ \fZ_{(x,y)}} \qquad \xymatrix{\fY_- \ar@{.>}[d]^-{\pi^\uparrow} \\ \fZ_{(x,y,x)}} \qquad \xymatrix{\fY_+ \ar@{.>}[d]^-{\pi^\uparrow} \\ \fZ_{(x,x,y)}} \qquad \xymatrix{\fY_{-+} \ar@{.>}[d]^-{\pi^\uparrow} \\ \fZ_{(x,x,y,x)}}
	\end{equation}
	obtained by remembering only the middle and top part of \eqref{eqn:y}--\eqref{eqn:y bullet}, and
	\begin{equation}
	\label{eqn:pi bot}
	\xymatrix{\fY \ar@{.>}[d]^-{\pi^\downarrow} \\ \fZ_{(y,x)}} \qquad \xymatrix{\fY_- \ar@{.>}[d]^-{\pi^\downarrow} \\ \fZ_{(y,x,x)}} \qquad \xymatrix{\fY_+ \ar@{.>}[d]^-{\pi^\downarrow} \\ \fZ_{(x,y,x)}} \qquad \xymatrix{\fY_{-+} \ar@{.>}[d]^-{\pi^\downarrow} \\ \fZ_{(x,y,x,x)}}
	\end{equation}
	obtained by remembering only the middle and bottom part of \eqref{eqn:y}--\eqref{eqn:y bullet}. In Proposition \ref{prop:representable y} we will show that the maps \eqref{eqn:pi top} and \eqref{eqn:pi bot} are representable, which together with Proposition \ref{prop:tower} shows that the functors $\fY$, $\fY_- ,\fY_+$, $\fY_{-+}$ are themselves representable. Let us first introduce some notation pertaining to $\fY$. The short exact sequence
	$$
	0 \longrightarrow \CU_1/\CU_0 = \Gamma^x_*(\CL_1) \longrightarrow \CU_2/\CU_0 \longrightarrow \CU_2/\CU_1 = \Gamma^y_*(\CL_2) \longrightarrow 0
	$$
	consists of coherent sheaves on $\fZ_2 \times S$ which are flat over $\fZ_2 = \fZ_{(x,y)}$. Here and in what follows, $\Gamma^x$ and $\Gamma^y$ denote the graphs of the maps
	$$
	p_S^x, p_S^y : \fZ_2 \rightarrow S
	$$
	which record the points $x$ and $y$, respectively. Let us write
	$$
	\fZ_2 \times S \xrightarrow{\pr} \fZ_2 
	$$
	for the standard projection. Then the short exact sequence
	\begin{equation}
	\label{eqn:ses push}
	0 \longrightarrow \CL_1 \longrightarrow \CE := \pr_*(\CU_2/\CU_0) \longrightarrow \CL_2 \longrightarrow 0
	\end{equation}
	consists of locally free sheaves of ranks $1,2,1$, respectively, on $\fZ_2$. The composition
	$$
	\pr^*(\CE) \otimes \CI_{\Gamma^y} \hookrightarrow \pr^*(\CE) \longrightarrow \CU_2/\CU_0 \longrightarrow \Gamma^y_*(\CL_2)
	$$
	vanishes (on account of $\CI_{\Gamma^y}$ being the ideal sheaf of functions which vanish on any coherent sheaf of the form $\Gamma^y_*(\dots)$), and therefore induces a map
	$$
	\pr^* (\CE) \otimes \CI_{\Gamma^y} \longrightarrow \CU_1/\CU_0 = \Gamma^x_*(\CL_1)
	$$
	By adjunction, this gives rise to a map
	$$
	\Gamma^{x*} ( \pr^* (\CE)) \otimes \Gamma^{x*}(\CI_{\Gamma^y}) \longrightarrow \CL_1
	$$
	which can be rewritten as
	\begin{equation}
	\label{eqn:the map}
	\CE \otimes (p_S^x \times p_S^y)^*(\CI_\Delta) \longrightarrow \CL_1
	\end{equation}
	where $\CI_\Delta$ is the ideal of the diagonal $\Delta : S \hookrightarrow S \times S$. Indeed, the identification $\Gamma^{x*}(\CI_{\Gamma^y}) \cong (p_S^x \times p_S^y)^*(\CI_\Delta)$ stems from the fiber square
	\begin{equation}
	\label{eqn:stem fiber}
	\xymatrix{\fZ_2 \ar@{^{(}->}[r]^-{\Gamma^y} \ar[d]_{p_S^y} & \fZ_2 \times S \ar[d]^{p_S^y \times \text{Id}_S} \\
		S \ar@{^{(}->}[r]^-\Delta & S \times S }
	\end{equation}
	Proposition \ref{prop:flat morphism} asserts that the vertical maps are flat, hence $\CI_{\Gamma^y} = (p_S^y \times \text{Id})^*(\CI_\Delta)$. \\
	
\noindent The preceding discussion applies equally well to the spaces $\fY_-$, $\fY_+$, $\fY_{-+}$ instead of $\fY$ (the sheaves $\CU_0, \CU_1, \CU_2, \CL_1, \CL_2$ must be replaced by $\CU_1, \CU_2, \CU_3, \CL_2, \CL_3$ in the case of $\fY_+$ and $\fY_{-+}$, see the notation in \eqref{eqn:y}, \eqref{eqn:y minus}, \eqref{eqn:y plus}, \eqref{eqn:y bullet}). \\
	
	\begin{proposition}
		\label{prop:representable y} 
		
		With the notation above, we have
		\begin{equation}
		\label{eqn:proj y}
		\xymatrixcolsep{3pc}\xymatrix{\fY \ar[rd]_-{\pi^\uparrow \emph{ or } \pi^\downarrow} \ar@{^{(}->}[r]^-{\iota^\uparrow \emph{ or } \iota^\downarrow} & \BP_{\fZ_2}(\CE) \ar[d]^\rho \\ & \fZ_2}
		\end{equation}
		where $\iota^\uparrow$ and $\iota^\downarrow$ are cut out by the following section:
		$$
		\sigma : \CH \otimes \rho^* \left( (p_S^x \times p_S^y)^*(\CI_\Delta) \right) \longrightarrow \rho^* \left(\CE \otimes (p_S^x \times p_S^y)^*(\CI_\Delta) \right) \xrightarrow{\eqref{eqn:the map}} \rho^*(\CL_1) 
		$$
		with $\CH = \emph{Ker } \rho^*(\CE) \twoheadrightarrow \CO(1)$. The same formulas hold for the spaces $\fY_-$, $\fY_+$, $\fY_{-+}$ instead of $\fY$, but replacing $\fZ_2$ by the corresponding spaces in \eqref{eqn:pi top} and \eqref{eqn:pi bot}. \\
	\end{proposition}
	
	\begin{proof} We will only prove the required statement for $\fY$, as the cases of $\fY_-$, $\fY_+$, $\fY_{-+}$ are analogous (equivalently, they follow from base change in the derived fiber squares \eqref{eqn:derived y 1}--\eqref{eqn:derived y 2}). Also, we will only prove the case of the map $\pi^\uparrow$, as the situation is symmetric by replacing $(\uparrow,x,y,\CL_1,\CL_2)$ with $(\downarrow,y,x,\CL_1',\CL_2')$. If we define the scheme $\fY$ by \eqref{eqn:proj y}, then maps $T \rightarrow \fY$ are in one-to-one correspondence with \\
		
		\begin{enumerate}[leftmargin=*]
			
			\item a map $\rho : T \rightarrow \fZ_2$, i.e. a flat family of stable coherent sheaves $(F_0 \subset F_1 \subset F_2)$ on $T \times S$, together with line bundles $L_1$, $L_2$ on $T$ and points $x,y : T \rightarrow S$ such that $F_1/F_0 \cong \Gamma^x_*(L_1)$ and $F_2/F_1 \cong \Gamma^y_*(L_2)$. \\
			
			\item a line bundle $\CO(1)$ on $T$ and a surjective homomorphism
			\begin{equation}
			\label{eqn:phi}
			E \twoheadrightarrow \CO(1)
			\end{equation}
			where $E = \pr_*(F_2/F_0)$ with $\pr : T \times S \rightarrow T$ the standard projection. If we let $H$ denote the kernel of \eqref{eqn:phi}, then one requires that the composition
			\begin{equation}
			\label{eqn:condition 1}
			H \otimes (x \times y)^*(\CI_\Delta) \rightarrow L_1
			\end{equation}
			be 0, where $x \times y : T \rightarrow S \times S$ is the product of the maps in item (1). \\
		\end{enumerate}
		
		\noindent We must show that any datum as in items (1)--(2) above gives rise to a square \eqref{eqn:y}. Item (1) yields the commutative diagram with exact rows and column
$$
\xymatrix{
			0 \ar[r] & \pr^*(L_1) \ar[r] \ar@{->>}[d] & \pr^*(E) \ar[r] \ar@{->>}[d] & \pr^*(L_2) \ar[r] \ar@{->>}[d] & 0 \\
			0 \ar[r] & \Gamma_*^x(L_1) \ar[r]^-\iota & F_2/F_0 \ar[r]^-\pi & \Gamma_*^y(L_2) \ar[r] & 0 }
$$
		on $T \times S$. The kernel of \eqref{eqn:phi} gives rise to a sub line bundle $H \subset E$, and the question is when $\Gamma^y_*(H)$ is a subsheaf of $F_2/F_0$ (this subsheaf would be $F_1'/F_0$, and this would complete the square \eqref{eqn:y} by constructing the bottom sheaf). It is easy to see that this happens precisely when the following composition vanishes:
		$$
		\nu : \pr^*(H) \otimes \CI_{\Gamma^y} \hookrightarrow \pr^*(E) \twoheadrightarrow F_2/F_0
		$$
Since $\pi \circ \nu = 0$, the map $\nu$ takes values in $\Gamma^x_*(L_1)$, so the question is when the map
		$$
		\pr^*(H) \otimes \CI_{\Gamma^y} \rightarrow \Gamma^x_*(L_1)
		$$
		vanishes. By adjunction, this happens when
		\begin{equation}
		\label{eqn:condition 2}
		H \otimes \Gamma^{x*}(\CI_{\Gamma^y}) \rightarrow L_1
		\end{equation}
		vanishes. Since $(x \times y)^*(\CI_\Delta) = \Gamma^{x*}(\CI_{\Gamma^y})$ (as a consequence of the sentence preceding diagram \eqref{eqn:stem fiber}), the vanishing of \eqref{eqn:condition 2} is equivalent to the vanishing of \eqref{eqn:condition 1}. 
		
	\end{proof}
	
	\subsection{} Now that we have defined the schemes $\fY$, $\fY_-$, $\fY_+$, $\fY_{-+}$, let us consider their basic properties. An important fact is the following geometric observation. \\
	
	\begin{proposition}
		\label{prop:zero}
		
		The composition $\CF_1/\CF_0 \hookrightarrow \CF_2/\CF_0 \twoheadrightarrow \CF_2/\CF_1'$ (see the notation in \eqref{eqn:y}) induces the following map of line bundles on $\fY$:
		\begin{equation}
		\label{eqn:zero}
		\CL_1 \rightarrow \CL_2'
		\end{equation}
		If we interpret this map as a section of $\CL_2' \otimes \CL^{-1}_1$, then its zero locus consists of
\begin{equation}
		\label{eqn:point zero}
		\Big\{ (\CF_1,x) = (\CF_1',y) \Big \} \hookrightarrow \fY,
\end{equation}
		and is isomorphic to $\fZ_{(x,x)}$. The analogous result holds for $\fY_-, \fY_+, \fY_{-+}$. \\
		
	\end{proposition}
	
	\begin{proof} By the functor-of-points description, a map from a scheme $T$ into the zero locus of \eqref{eqn:zero} consists of a square \eqref{eqn:y} (of coherent sheaves on $T \times S$, flat over $T$) such that the corresponding induced map
		\begin{equation}
		\label{eqn:fgh}
		\pr_*(\CF_1/\CF_0) \rightarrow \pr_*(\CF_2/\CF_1')
		\end{equation}
		vanishes, where $\pr : T \times S \rightarrow T$ denotes the projection (here we have used the facts that $\CF_1/\CF_0 \cong \Gamma^x_*(\CL_1)$, $\CF_2/\CF_1' \cong \Gamma^x_*(\CL_2')$ and $\pr \circ \Gamma^x = \text{Id}$). Because the sheaves $\CF_1/\CF_0$, and $\CF_2/\CF_1'$ are flat of length 1 over $T$, the map $\pr_*$ is an isomorphism on local sections, hence the induced map $\CF_1/\CF_0 \rightarrow \CF_2/\CF_1'$ vanishes. This implies that $\CF_1 \subset \CF_1'$ as subsheaves of $\CF_2$, hence we obtain short exact sequences
$$
0 \rightarrow \CF_1/\CF_0 \xrightarrow{f} \CF_1'/\CF_0 \xrightarrow{g} \CF_1'/\CF_1 \rightarrow 0
$$
$$
0 \rightarrow \CF_1'/\CF_1 \xrightarrow{f'} \CF_2/\CF_1 \xrightarrow{g'} \CF_2/\CF_1' \rightarrow 0
$$
The homomorphism $f$ is an injection $\Gamma^x_*(\CL_1) \rightarrow \Gamma^y_*(\CL_1')$, which implies that $x = y$ as maps $T \rightarrow S$ (the statement is local on both $S$ and $T$). Since $\Gamma^x$ is a closed embedding, $\Gamma^x_*$ is an exact functor and so the homomorphisms $f$ and $g'$ induce line bundle homomorphisms $\CL_1 \hookrightarrow \CL_1'$ and $\CL_2 \twoheadrightarrow \CL_2'$ on $T$. Since $\CL_1 \CL_2 \cong \CL_1'\CL_2' \cong \det \CE$ with $\CE = \pr_*(\CF_2/\CF_0)$, this is only possible if the aforementioned line bundle homomorphisms are isomorphisms (in other words, if $ab$ is a unit in a certain ring, then both $a$ and $b$ are units). Therefore, we must have $\CL_1 = \CL_1'$ as subsheaves of $\CE$, hence $\CF_1 = \CF_1'$ as subsheaves of $\CF_2$, which precisely yields a point of $\fZ_{(x,x)}$. The analogous results for $\fY_-, \fY_+, \fY_{-+}$ instead of $\fY$ are proved by the same argument.
		
	\end{proof}
	
	\subsection{} The following results will be proved in Section \ref{sec:geom}. The expected dimension of the varieties $\fY$, $\fY_-$, $\fY_+$, $\fY_{-+}$ is defined to be the dimension of the respective spaces on the bottom of \eqref{eqn:pi top} or \eqref{eqn:pi bot} (which is equal to the actual dimension of these spaces by Propositions \ref{prop:xy}, \ref{prop:xxy and yxx}, \ref{prop:xyx}, \ref{prop:xxyx and xyxx}). \\
	
	\begin{proposition}
	\label{prop:y}
	
$\fY$ is smooth and irreducible of expected dimension. \\

\end{proposition}

	\begin{proposition}
	\label{prop:y- and y+}
	
$\fY_-$ and $\fY_+$ are l.c.i. and irreducible of expected dimension. \\

\end{proposition}

	\begin{proposition}
	\label{prop:y-+}
	
$\fY_{-+}$ is l.c.i. and has two irreducible components of expected dimension. \\

\end{proposition}
	
\noindent Using the results above, we will prove the following fact in Section \ref{sec:geom}. \\	

\begin{proposition}
		\label{prop:reduced}
		
		The schemes $\fY$, $\fY_-$, $\fY_+$, $\fY_{-+}$ are reduced. \\
		
	\end{proposition}
	
\noindent Our main interest in the above geometric properties is the following result. \\
	
	\begin{proposition}
		\label{prop:birational}
		
		The map $\pi^\uparrow : \fY \rightarrow \fZ_2 = \fZ_{(x,y)}$ has the property that
		\begin{equation}
		\label{eqn:birational}
		R^i\pi^\uparrow_*(\CO_\fY) = \begin{cases} \CO_{\fZ_2} & \text{if } i = 0 \\ 0 & \text{if } i>0 \end{cases}
		\end{equation}
		The analogous properties hold for $\pi^\downarrow$. Moreover, the analogous properties hold with the scheme $\fY$ replaced by the schemes $\fY_-$, $\fY_+$, $\fY_{-+}$ of \eqref{eqn:pi top} and \eqref{eqn:pi bot}. \\
		
	\end{proposition}
	
	\begin{proof} Recall from Proposition \ref{prop:representable y} that $\pi^\uparrow = \rho \circ \iota^{\uparrow}$, where $\iota^{\uparrow}$ is a closed embedding and $\rho : \BP_{\fZ_2}(\CE) \twoheadrightarrow \fZ_2$ is a $\BP^1$--bundle. Therefore, we have a short exact sequence
		\begin{equation}
		\label{eqn:shorty}
		0 \rightarrow \text{Kernel} \rightarrow \CO_{\BP_{\fZ_2}(\CE)} \rightarrow \iota^{\uparrow}_*(\CO_{\fY}) \rightarrow 0
		\end{equation}
		of coherent sheaves on $\BP_{\fZ_2}(\CE)$. Because $\BP_{\fZ_2}(\CE)$ is a $\BP^1$--bundle over $\fZ_2$, we have
		\begin{align*}
		&R^i\rho_*(\CO_{\BP_{\fZ_2}(\CE)}) = 0 \qquad \forall i \geq 1 \\
		&R^i\rho_*(\text{Kernel}) = 0 \qquad \ \forall i \geq 2
		\end{align*}
(the latter equality would actually hold for any coherent sheaf on $\BP_{\fZ_2}(\CE)$). Therefore, in the long exact sequence in cohomology associated to \eqref{eqn:shorty}
		$$
\dots \rightarrow R^i \rho_*(\CO_{\BP_{\fZ_2}(\CE)}) \rightarrow R^i\rho_* \circ \iota^\uparrow_*(\CO_{\fY}) \rightarrow R^{i+1} \rho_*(\text{Kernel}) \rightarrow \cdots
		$$
		the spaces on the left and on the right are 0 for any $i \geq 1$. This implies \eqref{eqn:birational} for $i\geq 1$. As for the case $i = 0$, it follows from Stein factorization and the facts that \\
		
		\begin{enumerate}[leftmargin=*]
			
			\item $\fZ_2$ is normal (Proposition \ref{prop:normal}), \\
			
			\item $\fY$ is reduced (Proposition \ref{prop:reduced}), \\
			
			\item the map $\pi^\uparrow$ is proper (Proposition \ref{prop:representable y}) and all its fibers are either a point or $\BP^1$. \\
			
		\end{enumerate}
		
		\noindent Indeed, it is easy to observe that the fiber of $\pi^\uparrow$ above a closed point
		$$
		(\CF_0 \subset_x \CF_1 \subset_y \CF_2)
		$$
		is isomorphic to $\BP^1$ if $x = y$ and $\CF_2/\CF_0 \cong \BC_x^{\oplus 2}$, and is a point in all other cases.
		
	\end{proof}
	
\begin{remark}

All the results in Section \ref{sec:mod} hold with the moduli space $\CM$ replaced by the Hilbert schemes of points $\emph{Hilb}(S)$: all one needs to do is to replace stable sheaves $\CF$ by ideal sheaves $\CI$ everywhere. In particular, Proposition \ref{prop:birational} shows that the nested Hilbert scheme $\fZ_2 = \{(\CI_0 \subset \CI_1 \subset \CI_2)\}$ has rational singularities, and $\fY$ is the blow-up of its singular locus (see the proof of Proposition \ref{prop:y geom}). \\

\end{remark}
	
	\section{Derived categories}
	\label{sec:derived categories}
	
	\medskip
	
	\subsection{}
	\label{sub:derived categories}
	
	Recall that all our schemes are projective over $\BC$. Consider the following:
	\begin{align*}
	&D^b(\coh(X)) = \text{the bounded derived category of coherent sheaves on } X \\
	&\perf(X) = \text{the derived category of perfect complexes on } X
	\end{align*}
	We recall that a complex is perfect if it is quasi-isomorphic to a finite complex of locally free sheaves of finite rank on $X$ (this notion is usually called ``strictly perfect", although it is equivalent with the more general notion of ``perfect" on projective varieties). The natural inclusion functor
	\begin{equation} 
	\label{eqn:perf coh}
	\perf(X) \hookrightarrow D^b(\coh(X))
	\end{equation}
	is fully faithful, and is an equivalence if the scheme $X$ is smooth. We will write
	$$
	D_X = \perf(X) = D^b(\coh(X)) \qquad \text{if } X \text{ is smooth}
	$$
	The functors between derived categories associated to a morphism $f : X \rightarrow Y$ are \\
	
	\begin{enumerate}
		
		\item $D^b(\coh(X)) \xrightarrow{f_*} D^b(\coh(Y))$ if $f$ is proper \\
		
		\item $\perf(X) \xrightarrow{f_*} \perf(Y)$ if $f$ is proper and l.c.i. \\
		
		\item $D^b(\coh(Y)) \xrightarrow{f^*} D^b(\coh(X))$ if $f$ has finite Tor dimension \\
		
		\item $\perf(Y) \xrightarrow{f^*} \perf(X)$ if $f$ is arbitrary \\
		
	\end{enumerate}
	
	\noindent The existence of the push-forward in the second line is non-trivial, and we refer the reader to \cite{T} for an overview. As for the pull-back in the third line, we note that we always have a left derived functor
	$$
	D^-(\coh(Y)) \xrightarrow{f^*} D^-(\coh(X))
	$$
	on the derived categories of bounded above complexes of coherent sheaves over projective schemes. However, in order to ensure that the functor $f^*$ takes bounded complexes to bounded complexes, one needs a strong assumption, such as $f$ having finite Tor dimension, i.e. $\Tor_i^{\CO_Y-\text{mod}}(\CO_X,-) = 0$ for all $i$ large enough. \\
	
	\begin{proposition}
		\label{prop:base change}
		
		Consider a derived fiber square as in Definition \ref{def:lci}:
		\begin{equation}
		\label{eqn:base change}
		\xymatrix{Z' \ar@{^{(}->}[r]^{\iota'} \ar[d]_{\eta'} & X' \ar[d]^\eta \\
			Z \ar@{^{(}->}[r]^\iota & X}
		\end{equation}
		where $\eta$ is proper. Suppose that $Z \hookrightarrow X$ and $Z' \hookrightarrow X'$ are closed embeddings cut out by regular sections $\sigma$ and $\eta^*(\sigma)$ of locally free sheaves $W$ and $\eta^*(W)$, respectively. Then we have equivalences
		\begin{equation}
		\label{eqn:base change 1}
		\eta^* \iota_* \cong \iota'_* {\eta'}^* : \eperf(Z) \rightarrow \eperf(X')
		\end{equation}
		\begin{equation}
		\label{eqn:base change 2}
\qquad \qquad	\iota^* \eta_* \cong {\eta'}_* {\iota'}^*  : D^b(\ecoh(X')) \rightarrow D^b(\ecoh(Z))
		\end{equation}
		The same equivalences hold if the regular embeddings $\iota$ and $\iota'$ are replaced by smooth morphisms such as the projectivization of a locally free sheaf coming from $X$. \\
		
	\end{proposition}
	
	\begin{proof} The equivalences \eqref{eqn:base change 1} and \eqref{eqn:base change 2} are well-known to hold in the derived category of unbounded complexes, as long as the maps $\eta$ and $\iota$ are Tor independent:
		\begin{equation} 
		\label{eqn:tor independent}
		\tor_i^{\CO_{X,x}-\text{mod}}(\CO_{Z,z}, \CO_{X',x'}) = 0 \qquad \forall \ i \geq 1
		\end{equation}
		for all points $z \in Z$ and $x' \in X'$ which map to the same point $x \in X$. Since $\iota$ is a regular embedding, we may replace $\CO_{Z,z}$ by the Koszul complex of the section $\sigma^\vee : W^\vee_x \rightarrow \CO_{X,x}$, and so \eqref{eqn:tor independent} is equivalent to the complex
		$$
		\left[ \dots \xrightarrow{\sigma^\vee} \wedge^2 (W^\vee_x) \otimes_{\CO_{X,x}} \CO_{X',x'} \xrightarrow{\sigma^\vee} W^\vee_x \otimes_{\CO_{X,x}} \CO_{X',x'} \xrightarrow{\sigma^\vee} \CO_{X',x'} \right]
		$$
		being exact everywhere except at the right-most place. However, the above is none other than the Koszul complex of the section $\eta^*(\sigma)$ of the locally free sheaf $\eta^*(W)$, and its exactness follows from our assumption that the section $\eta^*(\sigma)$ is regular. \\
		
		\noindent Having established \eqref{eqn:base change 1} and \eqref{eqn:base change 2} at the level of unbounded complexes, the desired conclusion follows from the fact that $\perf$ and $D^b(\coh)$ are full subcategories of the derived category of unbounded complexes, together with the fact that the pull-back and push-forward functors associated to $\iota, \iota', \eta, \eta'$ are well-defined on these subcategories due to items (1)--(4) that precede the statement of Proposition \ref{prop:base change}.
		
	\end{proof}
	
	\subsection{} 
	\label{sub:k-theory}
	
	Associated to a scheme $X$, we have the $K$--theory groups
	\begin{align*}
	&K_0(X) = \text{Grothendieck group of } D^b(\coh(X)) \\
	&K^0(X) = \text{Grothendieck group of } \perf(X)
	\end{align*}
	If $X$ is smooth, these groups are isomorphic and we will write
	$$
	K_X = K_0(X) = K^0(X) \qquad \text{if } X \text{ is smooth}
	$$
	Associated to a morphism $f : X \rightarrow Y$, we have group homomorphisms $f_*$, $f^*$ between the various $K$--theory groups, as in the items (1)--(4) of Subsection \ref{sub:derived categories}. Moreover, \eqref{eqn:base change 1} and \eqref{eqn:base change 2} yield equalities of maps between $K$--theory groups, under the assumptions in Proposition \ref{prop:base change}. However, the following only holds in $K$--theory.  \\
	
	\begin{proposition}
		\label{prop:base change excess}
		
(\cite[Proposition 2.2]{FN}) Consider a derived fiber square with excess, as in Definition \ref{def:excess square}. Then
		\begin{equation}
		\label{eqn:base change excess}
\iota^* \eta_* = \eta'_* \left( \wedge^\bullet(L) \cdot {\iota'}^* \right) : K_0(X') \rightarrow K_0(Z)		
		\end{equation}
		where $\wedge^\bullet(L) = \sum_{i=0}^{\emph{rank } L} (-1)^i [\wedge^i(L)]$ for the excess bundle $L$. \\
		
	\end{proposition}
	
	\noindent Equality \eqref{eqn:base change excess} is called the ``excess intersection formula", and it does not generally lift to the derived category. One reason for this is that formula \eqref{eqn:excess} does not naturally give rise to a map $s : L \rightarrow \CO$ whose Koszul complex categorifies $\wedge^\bullet(L)$. The case when \eqref{eqn:base change excess} lifts to the derived category, with the section being $s = 0$, is treated in \cite{ACH} and linked with the relation between $Z'$ and the derived fiber product $Z \times_X^L X'$. \\
	
	\subsection{}
	\label{sub:proof}
	
	Let us recall the schemes $\fZ_1$ and $\fZ_2^\bullet$ of Subsection \ref{sub:schemes}. It was shown in \cite{Univ, W univ} that these schemes are smooth (see also Propositions \ref{prop:x} and \ref{prop:xx}). Because of Proposition \ref{prop:tower}, the maps $p_{\pm}$ and $\pi_{\pm}$ of Subsection \ref{sub:main intro} are all proper and l.c.i., which allow us to define functors
	\begin{equation}
	\label{eqn:category e}
	\te_{(d_1,\dots,d_n)} : D_{\CM} \rightarrow D_{\CM \times S}
	\end{equation}
for all $d_1,\dots,d_n \in \BZ$, according to the composition immediately after \eqref{eqn:def functor}. \\

	\begin{remark} 
		
		Analogously, one can define the transposed functor
		\begin{equation}
		\label{eqn:category f}
		\tf_{(d_1,\dots,d_n)} : D_{\CM} \rightarrow D_{\CM \times S}
		\end{equation}
		by the composition going the other way (left-to-right):
		$$
		\xymatrix{D_{\fZ_1} \ar[r]^{\otimes \CL^{d_1-r}} & D_{\fZ_1} \ar[r]^-{\pi_{-*} \pi_+^*} & D_{\fZ_1} \ar[r]^-{\otimes \CL^{d_2-r}} & D_{\fZ_1} \quad \dots \quad D_{\fZ_1} \ar[r]^-{\pi_{-*} \pi_+^*} & D_{\fZ_1} \ar[r]^-{\otimes \CL^{d_n-r}} & D_{\fZ_1}   \ar[d]_{(p_- \times p_S)_*} \\
			D_{\CM} \ar[u]_{p_+^*} & & & & & D_{\CM \times S} }
		$$
		followed by tensoring with $\left(\omega_S^{\otimes (1-r)}\det \CU \right)^{\otimes n}$ and shifting complexes by $rn$. \\
		
	\end{remark}

	\noindent In Section \ref{sec:a infty acts}, we will show that the functors $\te_{(d_1,\dots,d_n)}$ categorify the elliptic Hall algebra introduced in \cite{BS}. Our main tool is the commutation relation between the functors \eqref{eqn:category e} that we stipulated in Theorem \ref{thm:main}. To set up the relation, let us consider the compositions (we will often write $\te_k = \te_{(k)}$)
	\begin{align}
	& D_{\CM} \xrightarrow{\te_k} D_{\CM \times S_2} \xrightarrow{\te_{(d_1,...,d_n)} \times \text{Id}_{S_2}} D_{\CM \times S_1 \times S_2} \label{eqn:composition 1} \\
	& D_{\CM} \xrightarrow{\te_{(d_1,...,d_n)}} D_{\CM \times S_1} \xrightarrow{\te_k \times \text{Id}_{S_1}} D_{\CM \times S_1 \times S_2} \label{eqn:composition 2} 
	\end{align} 
	and denote them by $\te_{(d_1,\dots,d_n)} \circ \te_k$ and $\te_k \circ \te_{(d_1,\dots,d_n)}$, respectively. In formulas \eqref{eqn:composition 1} and \eqref{eqn:composition 2}, we set $S_1 = S_2 = S$, but we use different notations for the two copies of $S$ to emphasize the fact that the functor $\te_{(d_1,\dots,d_n)}$ takes values in the first factor of $S \times S$, while the functor $\te_k$ takes values in the second factor. \\
	
	\begin{proof} \emph{of Theorem \ref{thm:main}:} We will begin with the case $n=1$, then $n=2$ and finally general $n$. The fact that the fiber square \eqref{eqn:square xy} is derived, combined with Proposition \ref{prop:base change}, implies that the compositions $\te_d \circ \te_k$ and $\te_k \circ \te_d$ are given by the following correspondences, respectively:
		$$
		\xymatrix{& \CL_1^d \CL_2^k \ar@{.>}[d] & \\
			& \fZ_{(x,y)} \ar[rd]^-{p_-} \ar[ld]_-{p_+ \times p_S^x \times p_S^y} & \\
			\CM \times S \times S & & \CM} \qquad \qquad \xymatrix{& \CL_1^k \CL_2^d \ar@{.>}[d] & \\
			& \fZ_{(y,x)} \ar[rd]^-{p_-} \ar[ld]_-{p_+ \times p_S^x \times p_S^y} & \\
			\CM \times S \times S & & \CM}
		$$
		Above, we write $p_S^x, p_S^y$ for the maps $\fZ_2 \rightarrow S$ that remember the points $x,y$, respectively. The way to interpret the diagrams above as functors is the following: start from $D_{\CM}$, pull-back to $\fZ_2$, tensor by the line bundle on top, and then push-forward to $D_{\CM \times S \times S}$. For example, in the case of the diagram on the left, $\te_d \circ \te_k$ equals
		$$
		D_{\CM \times S \times S} \xleftarrow{(p_+ \times p_S^x \times p_S^y)_*} D^b(\coh(\fZ_{(x,y)})) \xleftarrow{\otimes \CL_1^d\CL_2^k} \perf(\fZ_{(x,y)}) \xleftarrow{p_-^*} D_\CM
		$$	
		Note that the middle arrow consists of tensoring with a line bundle, followed by the fully faithful map \eqref{eqn:perf coh}. This aspect is a necessary technicality, because the scheme $\fZ_2$ is not smooth (in the case at hand, we could have gotten away with only using the category $D^b(\coh)$ since the map $p_-$ is an l.c.i. morphism, but we will apply the notation above in situations where the morphism which plays the role of $p_-$ will not be l.c.i.). Consider the scheme $\fY$, together with the maps \eqref{eqn:pi top} and \eqref{eqn:pi bot}:
		$$
		\xymatrix{& \fY \ar[ld]_{\pi^\uparrow} \ar[rd]^{\pi^\downarrow} & \\ \fZ_{(x,y)} & & \fZ_{(y,x)}}
		$$
		Because of Proposition \ref{prop:birational}, we have $R^\bullet \pi^\uparrow_*(\CO_\fY) = \CO_{\fZ_{(x,y)}}$ and $R^\bullet \pi^\downarrow_*(\CO_\fY) = \CO_{\fZ_{(y,x)}}$ in the derived category. Therefore, the compositions $\te_d \circ \te_k$ and $\te_k \circ \te_d$ can be alternatively given by the following correspondences:
		$$
		\xymatrix{& \CL_1^d \CL_2^k \ar@{.>}[d] & \\
			& \fY \ar[rd] \ar[ld] & \\
			\CM \times S \times S & & \CM} \qquad \qquad \xymatrix{& {\CL_1'}^k {\CL_2'}^d \ar@{.>}[d] & \\
			& \fY \ar[rd] \ar[ld] & \\
			\CM \times S \times S & & \CM}
		$$
		(the map that points left, respectively right, remembers the sheaf $\CF_0$, respectively $\CF_2$ in the notation of \eqref{eqn:y}). Recall from \eqref{eqn:picard relation 1} the identity $\CL_1\CL_2 = \CL_1'\CL_2'$ on $\fY$. The map of line bundles $\CL_1 \rightarrow \CL_2'$ of \eqref{eqn:zero} gives rise to maps
		\begin{equation}
		\label{eqn:cases}
		\begin{cases}
		\CL_1^d \CL_2^k \rightarrow {\CL_1'}^k {\CL_2'}^d & \text{if } d > k \\
		\CL_1^d \CL_2^k \leftarrow {\CL_1'}^k {\CL_2'}^d & \text{if } d < k \\
		\CL_1^d \CL_2^k \cong {\CL_1'}^k {\CL_2'}^d & \text{if } d = k
		\end{cases}
		\end{equation}
		which induces a natural transformation between the functors $\te_d \circ \te_k$ and $\te_k \circ \te_d$, as required in item (2) of Theorem \ref{thm:main}. In general, the cone of a composition of maps in any triangulated category has a filtration with associated graded given by the sum of the cones of the individual maps \footnote{The analogous statement in an abelian category would be that, given a flag of objects $V_0 \subset \dots \subset V_n$, the quotient $V_n/V_0$ has a filtration with associated graded given by the sum of the individual quotients $V_i/V_{i-1}$.}. Thus, the cone of the maps denoted by $\rightarrow$ and $\leftarrow$ in \eqref{eqn:cases} has a filtration with associated graded given by the sum of the cones
		$$
		\begin{cases}
		\Big( \CL_1^{a+1-k} {\CL_2'}^{d-a-1} \rightarrow \CL_1^{a-k} {\CL_2'}^{d-a} \Big) \otimes (\CL_1\CL_2)^k & \text{with } a \in \{k,\dots,d-1\} \text{ if } d > k \\
		\Big( {\CL_1'}^{a-d} \CL_2^{k-a} \leftarrow {\CL_1'}^{a-d+1} \CL_2^{k-a-1} \Big) \otimes (\CL_1\CL_2)^d & \text{with } a \in \{d,\dots,k-1\} \text{ if } d < k 
		\end{cases}
		$$
		By Proposition \ref{prop:zero}, the cone of the maps denoted by $\rightarrow$ and $\leftarrow$ above is precisely the line bundle $\CL_1^a \CL_2^{d+k-a}$ on the scheme $\fZ_{(x,x)}$, where $a$ goes over the indexing sets featured in either of the two situations above. The correspondence
		$$
		\xymatrix{& & \CL_1^a \CL_2^{d+k-a} \ar@{.>}[d] & \\
			& & \fZ_{(x,x)} \ar[rd]^-{p_-} \ar[ld]_-{p_+ \times p_S^x} & \\
			\CM \times S \times S & \CM \times S \ar@{_{(}->}[l]_-{\Delta} & & \CM}
		$$
		is precisely $\Delta_*(\te_{(a,d+k-a)})$, and this implies item (3) of Theorem \ref{thm:main}. The proof of Theorem \ref{thm:main} in the case $n=1$ is now complete. \\
		
		\noindent Let us now deal with the case $n=2$. The fact that the fiber squares in Proposition \ref{prop:xxy and yxx} are derived, combined with Proposition \ref{prop:base change}, implies that the compositions $g_0 = \te_{(d_1,d_2)} \circ \te_k$ and $g_2 = \te_k \circ \te_{(d_1,d_2)}$ are given by the following correspondences, respectively:
		$$
		\xymatrix{& \CL_1^{d_1} \CL_2^{d_2} \CL_3^k \ar@{.>}[d] & \\
			& \fZ_{(x,x,y)} \ar[rd]^-{p_-} \ar[ld]_-{p_+ \times p_S^x \times p_S^y} & \\
			\CM \times S \times S & & \CM} \qquad \xymatrix{& \CL_1^k \CL_2^{d_1} \CL_3^{d_2} \ar@{.>}[d] & \\
			& \fZ_{(y,x,x)} \ar[rd]^-{p_-} \ar[ld]_-{p_+ \times p_S^x \times p_S^y} & \\
			\CM \times S \times S & & \CM}
		$$
The maps denoted by $p_+$ and $p_-$ only remember the sheaves $\CF_0$ and $\CF_3$, respectively, from points $(\CF_0 \subset \dots \subset \CF_3)$ of either $\fZ_{(x,x,y)}$ or $\fZ_{(y,x,x)}$. We will also consider the functor $g_1 : D_\CM \rightarrow D_{\CM \times S \times S}$ given by the following correspondence
		$$
		
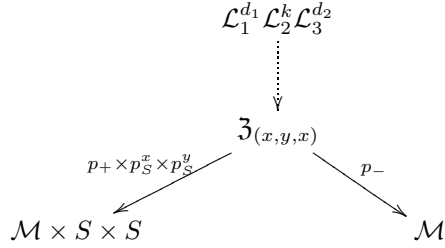
\xymatrix{& \CL_1^{d_1} \CL_2^k \CL_3^{d_2}  \ar@{.>}[d] & \\
			& \fZ_{(x,y,x)} \ar[rd]^-{p_-} \ar[ld]_-{p_+ \times p_S^x \times p_S^y} & \\
			\CM \times S \times S & & \CM}
		$$
		Because of Proposition \ref{prop:birational}, the functors $g_0$ and $g_1$ can be given by the following correspondences, respectively:
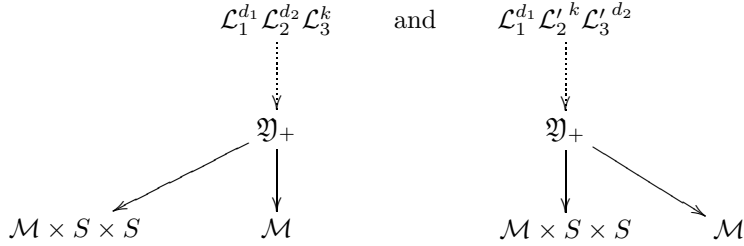
		$$
		\xymatrix{& \CL_1^{d_1} \CL_2^{d_2} \CL_3^k \ar@{.>}[d] \\
			& \fY_+ \ar[ld] \ar[d] \\
			\CM \times S \times S & \CM} \qquad \text{and} \qquad \xymatrix{ \CL_1^{d_1} {\CL_2'}^{k} {\CL_3'}^{d_2}  \ar@{.>}[d] & \\
			\fY_+ \ar[d] \ar[rd] & \\
			\CM \times S \times S & \CM}
		$$
		while the functors $g_1$ and $g_2$ can be given by the following correspondences, respectively:
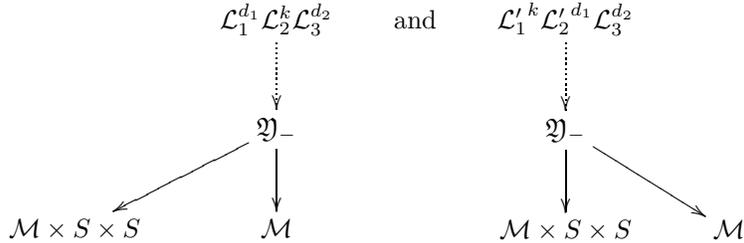
		$$
		\xymatrix{& \CL_1^{d_1} \CL_2^k \CL_3^{d_2} \ar@{.>}[d] \\
			& \fY_- \ar[ld] \ar[d] \\
			\CM \times S \times S & \CM} \qquad \text{and} \qquad \xymatrix{ {\CL_1'}^{k} {\CL_2'}^{d_1} \CL_3^{d_2}  \ar@{.>}[d] & \\
			\fY_- \ar[d] \ar[rd] & \\
			\CM \times S \times S & \CM}
		$$
		Depending on whether $d_2 > k$ or $d_2 < k$ or $d_2 = k$, we obtain natural transformations $g_0 \rightarrow g_1$ or $g_0 \leftarrow g_1$ or $g_0 \cong g_1$ by using the map \eqref{eqn:zero} on the space $\fY_+$ (cf. \eqref{eqn:cases}). Similarly, depending on whether $d_1 > k$ or $d_1 < k$ or $d_1 = k$, we obtain natural transformations $g_1 \rightarrow g_2$ or $g_1 \leftarrow g_2$ or $g_1 \cong g_2$ by using the map \eqref{eqn:zero} on the space $\fY_-$. This establishes item (2) of Theorem \ref{thm:main}. To prove item (3), one needs to proceed as in the $n=1$ case, and we will present the detailed argument for the natural transformation $g_0 \rightarrow g_1$ in the case $d_2 > k$. The cone of this natural transformation arises from the cone of the map of line bundles
		$$
		\CL_1^{d_1} \otimes \Big(\CL_2^{d_2} \CL_3^k \rightarrow {\CL_2'}^{k} {\CL_3'}^{d_2} \Big)
		$$
		on $\fY_+ $, which in turn has a filtration whose associated graded object is the direct sum of the cones
		\begin{equation}
		\label{eqn:cone}
		\CL_1^{d_1} \CL_2^a \CL_3^k {\CL_3'}^{d_2-a} \otimes \left(\frac {\CL_2}{\CL_3'} \rightarrow \CO \right) \qquad \text{as } a \in \{k,\dots,d_2-1\}
		\end{equation}
		The analogue of Proposition \ref{prop:zero} for $\fY_+$ instead of $\fY$ implies that the cone of the map \eqref{eqn:cone} is the structure sheaf of the subscheme $\fZ_{(x,x,x)} \hookrightarrow \fY_+$ given by the condition $x=y$ and $\CF_2 = \CF_2'$. Therefore, the cone in \eqref{eqn:cone} is the correspondence $\CL_1^{d_1} \CL_2^{a} \CL_3^{d_2+k-a}$ on $\fZ_{(x,x,x)}$. Since the fiber square in Proposition \ref{prop:xxx} is derived, the functor $D_\CM \rightarrow D_{\CM \times S \times S}$ induced by this correspondence is $\Delta_*(\te_{(d_1,a,d_2+k-a)})$, precisely as stipulated by item (3) of Theorem \ref{thm:main}. \\
		
		\noindent Now that we have established Theorem \ref{thm:main} for $n=1$ and $n=2$, let us tackle the case of general $n$. Consider the functors $g_0 = \te_{(d_1,\dots,d_n)} \circ \te_k$ and $g_n = \te_k \circ \te_{(d_1,\dots,d_n)}$. For any $i \in \{1,\dots,n-1\}$ let $g_i$ denote the following composition:

		$$
		\xymatrix{& & & & D_{\CM} \ar[d]^{p_-^*} \\
			& \ar[dl]_{\rho_-^*} D_{\fZ_1} & D_{\fZ_1} \ar[l]_-{\otimes \CL^{d_{n-i+1}}} \ \dots \ D_{\fZ_1} & D_{\fZ_1} \ar[l]_-{\pi_{+*} \pi_-^*} & D_{\fZ_1} \ar[l]_{\otimes \CL^{d_n}} \\
			\perf(\fZ_{(x,y,x)}) \ar[drrrr]^{\otimes \CL_{\text{mid}}^k} & & & & \\
			& & & & D^b(\coh(\fZ_{(x,y,x)})) \ar[ld]_{\rho_{+*}} \\
			D_{\fZ_1} \ar[d]^{(p_+ \times p_S)_*} & D_{\fZ_1} \ar[l]_{\otimes \CL^{d_1}} & D_{\fZ_1} \ar[l]_-{\pi_{+*} \pi_-^*} \ \dots \ D_{\fZ_1} & \ar[l]_-{\otimes \CL^{d_{n-i}}} D_{\fZ_1} & \\
			D_{\CM \times S} & & & &}
		$$
		where we consider the maps $\rho_\pm$ and the line bundle $\CL_{\text{mid}}$ as in the following diagram:
		$$
		\xymatrix{& \CL_{\text{mid}} \ar@{.>}[d] & \\ & \fZ_{(x,y,x)} \ar[ld]_{\rho_+} \ar[rd]^{\rho_-} & \\
			\fZ_1 & & \fZ_1} \qquad \xymatrix{& \Gamma(S,\CF_2/\CF_1) \ar@{.>}[d] & \\ & \CF_0 \subset_x \CF_1 \subset_y \CF_2 \subset_x \CF_3 \ar@{|->}[ld] \ar@{|->}[rd] & \\
			\CF_0 \subset_x \CF_1 & & \CF_2 \subset_x \CF_3}
		$$
		The required statements about the natural transformations $g_0 \leftrightarrow g_1$ and $g_{n-1} \leftrightarrow g_n$ are proved just like in the case $n=2$ in the previous paragraph. Similarly, the required statements about the natural transformations $g_{i-1} \leftrightarrow g_i$ are all proved just like in the case $n=3$ and $i=2$, which we will now explain. According to the diagram above, $g_1$ is given by the correspondence
		
	$$
		\xymatrix{& & & \CL_{\text{mid}}^k \ar@{.>}[d] & \\
			\CL^{d_1} \ar@{.>}[d] & \fZ_{(x,x)} \ar[ld]_{\pi_+} \ar[rd]^{\pi_-} & \CL^{d_2} \ar@{.>}[d] & \fZ_{(x,y,x)} \ar[ld]_{\rho_+} \ar[rd]^{\rho_-} & \CL^{d_3} \ar@{.>}[d] \\
			\fZ_1 \ar[d]_{p_+ \times p_S^x \times p_S^y} & & \fZ_1 & & \fZ_1 \ar[d]^{p_-} \\
			\CM \times S \times S & & & & \CM}
		$$	
One should read the diagram above as inducing a functor $D_\CM \rightarrow D_{\CM \times S \times S}$ by composing the pull-back and push-forward functors corresponding to the solid arrows, read right-to-left, all the while at every step tensoring by the line bundle which is displayed with a dotted line above every space. Similarly, $g_2$ is given by the correspondence
		$$
		\xymatrix{& \CL_{\text{mid}}^k \ar@{.>}[d] & & & \\
			\CL^{d_1} \ar@{.>}[d] & \fZ_{(x,y,x)} \ar[ld]_{\rho_+} \ar[rd]^{\rho_-} & \CL^{d_2} \ar@{.>}[d] & \fZ_{(x,x)} \ar[ld]_{\pi_+} \ar[rd]^{\pi_-} & \CL^{d_3} \ar@{.>}[d] \\
			\fZ_1 \ar[d]_{p_+ \times p_S^x \times p_S^y} & & \fZ_1 & & \fZ_1 \ar[d]^{p_-} \\
			\CM \times S \times S & & & & \CM}
		$$
		Because the fiber squares in Proposition \ref{prop:xxyx and xyxx} are derived, we may invoke Proposition \ref{prop:base change} in order to conclude that $g_1$ and $g_2$ are given by the following correspondences:
		$$
		\xymatrix{& \CL_1^{d_1} \CL_2^{d_2} \CL_3^{k} \CL_4^{d_3} \ar@{.>}[d] \\
			& \fZ_{(x,x,y,x)} \ar[ld] \ar[d] \\
			\CM \times S \times S & \CM} \qquad \text{and} \qquad \xymatrix{   \CL_1^{d_1} \CL_2^{k} \CL_3^{d_2} \CL_4^{d_3} \ar@{.>}[d] & \\
			\fZ_{(x,y,x,x)} \ar[d] \ar[rd] & \\
			\CM \times S \times S & \CM}
		$$
		respectively. However, Proposition \ref{prop:birational} applied to $\fY_{-+}$ implies that we may replace the line bundles above on $\fZ_{(x,x,y,x)}$ and $\fZ_{(x,y,x,x)}$ by their pullbacks to $\fY_{-+}$ without changing the correspondences. Therefore, the required natural transformation between $g_1$ and $g_2$ is induced by a map of line bundles
		$$
\CL_1^{d_1} \CL_2^{d_2} \CL_3^{k} \CL_4^{d_3}	 \qquad \text{and} \qquad \CL_1^{d_1} {\CL_2'}^{k} {\CL_3'}^{d_2} \CL_4^{d_3}
		$$
		on $\fY_{-+}$. The construction of this map, as well as the proof of the fact that its cone has the properties stipulated in item (3) of Theorem \ref{thm:main}, are achieved by using the map \eqref{eqn:zero} in the same way as we have already seen in the cases $n=1$ and $n=2$. 
		
	\end{proof}
	
	\subsection{} 
	\label{sub:k theory}
	
	At the level of $K$--theory, the functors \eqref{eqn:category e} and \eqref{eqn:category f} give rise to maps
	\begin{equation}
	\label{eqn:e in k-theory}
	e_{(d_1,\dots,d_n)} : \km \rightarrow \kms
	\end{equation}
	\begin{equation}
	\label{eqn:f in k-theory}
	f_{(d_1,\dots,d_n)} : \km \rightarrow \kms
	\end{equation}
	for all $d_1,\dots,d_n \in \BZ$. We also consider the power series of maps
	\begin{equation}
	\label{eqn:h in k-theory}
	h^\pm(z) = \sum_{k = 0}^\infty \frac {h_k^\pm}{z^{\pm k}} : \km \rightarrow \kms[[z^{\mp 1}]]
	\end{equation}
	defined as the composition
	$$
h^\pm(z) :	\km \xrightarrow{\text{pull-back}} \kms \xrightarrow{\text{multiplication by } \wedge^\bullet \left( \CU (q^{-1}-1)/z \right)} \kms
	$$
Above, recall that $\CU$ is the universal sheaf on $\CM \times S$, and $q = [\omega_S] \in \ks$ (we abuse notation and use the symbol $q$ also for the pull-back of the canonical line bundle to $\kms$). The meaning of the wedge power is the following. Proposition \ref{prop:length 1} allows one to write $[\CU] = [\CV] - [\CW] \in \kms$, where $\CV$ and $\CW$ are locally free sheaves. Since the total exterior power is supposed to be multiplicative, we define
	$$
	\wedge^\bullet \left( \frac {\CU(q^{-1}-1)}z\right) = \frac {\wedge^\bullet \left( \frac {\CV}{zq} \right)\wedge^\bullet \left( \frac {\CW}z \right)}{\wedge^\bullet \left( \frac {\CV}z \right)\wedge^\bullet \left( \frac {\CW}{zq} \right)} \in \kms(z), \text{ where } \wedge^\bullet \left( \frac {\CV}z \right) = \sum_{i=0}^{\text{rank }\CV} \frac {[\wedge^i\CV]}{(-z)^i}
	$$
	We showed in \cite{Univ} that the operators \eqref{eqn:e in k-theory}, \eqref{eqn:f in k-theory}, \eqref{eqn:h in k-theory} satisfy the relations
	\begin{equation}
	\label{eqn:di 1}
	\wedge^\bullet\left( - \frac wz \cdot \CO_\Delta \right) \left(\sum_{k=-\infty}^{\infty} \frac {e_k}{z^k} \right) h^\pm(w) = \wedge^\bullet\left( - \frac zw \cdot \CO_\Delta \right) h^\pm(w) \left(\sum_{k=-\infty}^{\infty} \frac {e_k}{z^k} \right)
	\end{equation}
	\begin{equation}
	\label{eqn:di 2}
	\wedge^\bullet\left( - \frac zw \cdot \CO_\Delta \right) \left(\sum_{k=-\infty}^{\infty} \frac {f_k}{z^k} \right) h^\pm(w) = \wedge^\bullet\left( - \frac wz \cdot \CO_\Delta \right) h^\pm(w) \left(\sum_{k=-\infty}^{\infty} \frac {f_k}{z^k} \right)
	\end{equation}
	\begin{equation}
	\label{eqn:di 3}
	\left[ \sum_{k=-\infty}^{\infty} \frac {e_k}{z^k}, \sum_{l=-\infty}^{\infty} \frac {f_l}{w^l} \right] = \left(\sum_{k=-\infty}^{\infty} \frac {z^k}{w^k} \right) \Delta_* \left(\frac {h^+(z) - h^-(w)}{1-q^{-1}} \right)
	\end{equation}
where $\CO_\Delta$ is the $K$--theory class of the diagonal inside $S \times S$. We will now work out some relations among the operators \eqref{eqn:e in k-theory} and \eqref{eqn:f in k-theory}. \\

\begin{proposition}
\label{prop:excess k-theory}

We have the following identities of operators $\km \rightarrow \kms$:
\begin{equation}
\label{eqn:excess relations}
e_{(d_1,\dots,d_n)} e_{(d_1',\dots,d_{n'}')} \Big|_\Delta = e_{(d_1,\dots,d_n,d_1',\dots,d_{n'}')} - q \cdot e_{(d_1,\dots,d_n-1,d_1'+1,\dots,d_{n'}')}
	\end{equation}
for all $d_1,\dots,d_n,d_1',\dots,d'_{n'} \in \BZ$. The analogous formula to \eqref{eqn:excess relations} holds with the $e$'s replaced by $f$'s, once one replaces the product in the LHS by the opposite product. \\

\end{proposition}

\begin{proof} We will prove \eqref{eqn:excess relations} for $n=n'=1$, and then the general case will be clear, in light of the fact that the operators \eqref{eqn:e in k-theory} arise as successive push-forwards and pull-backs between $\fZ_1$ and $\fZ_2^\bullet$ (cf. the diagram after \eqref{eqn:def functor}). Consider the diagram:
\begin{equation}
\label{eqn:big square}
		\xymatrix{\fZ_2^\bullet \ar[r]^{\pi_-} \ar[d]_{\pi_+} & \fZ_1 \ar[d]^{p_+ \times p_S} \ar[r]^{p_-} & \CM\\
			\fZ_1 \ar[r]^-{p_- \times p_S} \ar[d]_{p_+ \times p_S} & \CM \times S & \\ \CM \times S & &}
\end{equation}
with the notation as in Subsection \ref{sub:main intro}. By a simple diagram chase, we have:
$$
e_d e_{d'}\Big|_\Delta = (p_+ \times p_S)_* \Big(\CL^d \cdot (p_- \times p_S)^* \circ (p_+ \times p_S)_*( \CL^{d'} \cdot p_-^* )\Big) 
$$
By Propositions \ref{prop:xx} and \ref{prop:base change excess} (the latter Proposition applies because all the maps in \eqref{eqn:big square} are closed embeddings followed by smooth maps), we have:
$$
(p_- \times p_S)^* \circ (p_+ \times p_S)_* = \pi_{+*}\left[ \left(1 - \frac {\CL_2 q}{\CL_1} \right) \cdot \pi_-^* \right]
$$
where $p_S : \fZ_2^\bullet \rightarrow S$ in the RHS of the formula above. Therefore, we have:
$$
e_d e_{d'}\Big|_\Delta = (p_+ \times p_S)_* \left\{\CL^d \cdot \pi_{+*} \left[\left(1 - \frac {\CL_2q}{\CL_1} \right) \cdot \pi_{-}^* \left(\CL^{d'} \cdot p_-^* \right) \right]  \right\}
$$
and it is easy to see that the right-hand side is $e_{(d,d')} - q \cdot e_{(d-1,d'+1)}$.

\end{proof}
	
\noindent As an immediate consequence of Theorem \ref{thm:main}, we conclude that the operators \eqref{eqn:def functor k} satisfy the identities \eqref{eqn:new relations} of operators $\km \rightarrow \kmss$, as well as the opposite Lie bracket relations for the commutator of $f_{(d_1,\dots,d_n)}$ with $f_k$. \\
	
	\section{The $\CA$ action}
	\label{sec:a infty acts}
	
	\medskip
	
	\subsection{} We will now recall the algebraic structure that governs the composition of the operators \eqref{eqn:e in k-theory}, \eqref{eqn:f in k-theory}, \eqref{eqn:h in k-theory}. Our presentation will be an adaptation of \cite{BS} and \cite{S}. Consider two formal parameters $q_1$ and $q_2$, let $q=q_1q_2$ and define
	$$
	[n] = 1+q^{-1}+ \dots +q^{-n+1}
	$$
	Throughout this Section, we will often encounter the ring
	\begin{equation}
	\label{eqn:ring loc} 
	\BK = \BZ[q_1^{\pm 1}, q_2^{\pm 1}]^{\sym}_{([1], [2], [3],...)}
	\end{equation}
	where $\sym$ refers to Laurent polynomials which are symmetric in $q_1$ and $q_2$. \\
	
	\begin{definition} 
		\label{def:algebra} 
		
		(\cite{S}) Consider the $\BK$--algebra $\CA$ with generators
		$$
		\{E_k,F_k,H^\pm_l\}_{k \in \BZ, l\in \BN}
		$$
		modulo the following relations
		\begin{multline}
		(z-wq_1)(z-wq_2)\left(z - \frac wq \right) E(z) E(w) = \\ = \left(z - \frac w{q_1}\right)\left(z - \frac w{q_2} \right)\left(z - wq \right) E(w) E(z) \label{eqn:rel 1}
		\end{multline}
		\begin{multline}
		(z-wq_1)(z-wq_2)\left(z - \frac wq \right) E(z) H^\pm(w) = \\ = \left(z - \frac w{q_1}\right)\left(z - \frac w{q_2} \right)\left(z - wq \right) H^\pm(w) E(z) \label{eqn:rel 2}
		\end{multline}
		\begin{equation}
		\label{eqn:rel 3}
		[[E_{k+1},E_{k-1}],E_k] = 0 \qquad \forall k \in \BZ
		\end{equation}
		together with the opposite relations for $F(z)$ instead of $E(z)$, as well as
		\begin{equation}
		\label{eqn:rel 4}
		[E(z), F(w)] = \delta \left(\frac zw \right) (1-q_1)(1-q_2) \left( \frac {H^+(z) - H^-(w)}{1-q^{-1}} \right)
		\end{equation}
		where
		$$
		E(z) = \sum_{k \in \BZ} \frac {E_k}{z^k}, \qquad F(z) = \sum_{k \in \BZ} \frac {F_k}{z^k}, \qquad H^\pm(z) = \sum_{l = 0}^{\infty} \frac {H_l^\pm}{z^{\pm l}}
		$$
		We will set $H^+_0 = 1$ and $H^-_0 = c^{-1}$, and note that $c$ is a central element of $\CA$. \\
		
	\end{definition}
	
	\subsection{} The algebra $\CA$ is known by many names, including the Ding-Iohara-Miki algebra (\cite{DI, M}), the double shuffle algebra (\cite{FHHSY}), the stable limit of trigonometric DAHA (\cite{Ch}), and many other incarnations in representation theory and mathematical physics. However, the main description we will be concerned with is that of the elliptic Hall algebra of \cite{BS}. Specifically, the following is the main Theorem of \cite{S}. \\
	
	\begin{theorem} 
		\label{thm:algebra} 
		
Define the elliptic Hall algebra (\cite{BS}, see \cite[Theorem 3.5]{W univ} for our conventions) to be generated by elements $\{E_{n,k}\}_{(n,k) \in \BZ^2 \backslash (0,0)}$, modulo the relations
		\begin{multline}
		[E_{n,k},E_{n',k'}] = (1-q_1)(1-q_2)  \\ \left( \sum^{t \geq 1, \sum_{i=1}^t n_i = n+n', \sum_{i=1}^t k_i = k+k'}_{(n',k') \ccur (n_1,k_1) \ccur \dots \ccur (n_t,k_t) \ccur (n,k)} E_{n_1,k_1} \dots E_{n_t,k_t} \cdot \emph{coeff} \right) \label{eqn:comm}
		\end{multline}
for any $(n',k') \ccur (n,k)$ (the notation $\ccur$ means that $(n',k')$ can be reached clockwise from $(n,k)$ without crossing the negative $y$ axis) and certain scalars $\emph{coeff} \in \BK$. \\

\noindent Then upon localization with respect to the element $(1-q_1)(1-q_2) \in \BK$, the algebra $\CA$ becomes isomorphic to the elliptic Hall algebra via the assignment
\begin{equation}
		\label{eqn:gan}
E_{-1,k} = E_k, \qquad E_{1,k} = F_k , \qquad \frac {1 + \sum_{l=1}^\infty \frac {E_{0,\pm l}}{(-zq)^{\pm l}}}{1 + \sum_{l=1}^\infty \frac {E_{0,\pm l}}{(-z)^{\pm l}}} = H^\pm(z)
		\end{equation}
		
	\end{theorem}
	
\medskip

\noindent We will henceforth abuse notation and refer to $E_{n,k}$ as elements of $\CA$. If the triangle with vertices $(0,0)$, $(n,k)$, $(n+n',k+k')$ has no lattice points inside and on any of the edges, then \eqref{eqn:comm} takes the following simple form:
	\begin{equation}
	\label{eqn:comm exp} 
	[E_{n,k},E_{n',k'}] = (1-q_1)(1-q_2) E_{n+n',k+k'} 
	\end{equation}
	Formula \eqref{eqn:comm exp} shows how, after inverting $(1-q_1)(1-q_2)$, any $E_{n,k}$ for $\gcd(n,k) = 1$ can be obtained from successive commutators of $E_{\pm 1,k}$ (there is a similar formula in \cite{BS} which deals with the non coprime case, see \cite{W univ} for our conventions). \\
	
	\begin{remark}
		\label{rem:heis}
		
		Another important case of \eqref{eqn:comm} is when $(n,k)$ and $(n',k')$ are proportional. In order to state the relation, let us set for any coprime integers $n,k$:
		$$
		\exp \left(- \sum_{s=1}^\infty \frac {P_{ns,ks}}{s x^s} \right) = 1 + \sum_{s=1}^\infty \frac {E_{ns,ks}}{(-x)^s}
		$$
		Then the elements $P_{n,k} \in \CA$ defined by the formula above are multiples of the generators considered in \cite{BS} (see \cite[Section 2.6]{W} for a formula of the precise multiple). In terms of the $P_{n,k}$'s, formula \eqref{eqn:comm} when $(n,k)$ and $(n',k')$ are proportional reads
		\begin{equation}
		\label{eqn:comm heis}
		[P_{n,k}, P_{n',k'}] = \delta_{n+n'}^0 (1-q_1^s)(1-q_2^s) \cdot \frac {s(-\emph{sign } n)(1-c^{-|n|})}{1-q^{-s}}
		\end{equation}
		where $s = \gcd(n,k)$. In all representations of $\CA$ considered in the present paper, we will set $c = q^r$ for a natural number $r$ (called the ``level" of the representation) and therefore the fraction in \eqref{eqn:comm heis} is an element of $\BK$. As a consequence of this fact, we conclude that the algebra $\CA$ contains a deformed Heisenberg algebra for each slope $\frac kn$. \\
		
	\end{remark}
	
	\subsection{} In \cite[Proposition 6.2]{Shuf}, we defined elements (albeit denoted by $X_{m_1,\dots,m_k}$)
	\begin{equation}
	\label{eqn:new e}
	\Big\{ E_{(d_1,\dots ,d_n)} \Big\}_{d_1,\dots ,d_n \in \BZ} \in \CA
	\end{equation}
	for all $d_1,\dots,d_n \in \BZ$. It is an easy exercise to show that we have the following property for all $d_1,\dots,d_n,d_1',\dots,d'_{n'} \in \BZ$ (see \cite[Proposition 6.5]{Shuf}):
	\begin{equation}
	\label{eqn:eee}
	E_{(d_1,\dots,d_n)} E_{(d_1',\dots,d_{n'}')} = E_{(d_1,\dots,d_n,d_1',\dots,d_{n'}')} - q \cdot E_{(d_1,\dots,d_n-1,d_1'+1,\dots,d_{n'}')}
	\end{equation}
Meanwhile, we have by convention for all $(n,k) \in \BN \times \BZ$ (see \cite[Proposition 2.3]{W univ}):
	\begin{equation}
	\label{eqn:formula e}
	E_{-n,k} = q^{\gcd(n,k)-1} E_{(d^{(n,k)}_1,\dots,d^{(n,k)}_n)}
	\end{equation}
	where
	\begin{equation}
	\label{eqn:d i n k}
	d^{(n,k)}_i = \left \lceil \frac {ki}n \right \rceil - \left \lceil \frac {k(i-1)}n \right \rceil + \delta_i^n - \delta_i^1
	\end{equation}
for all $i$. As shown in the proof of Proposition 2.3 of \emph{loc. cit.}, any $E_{(d_1,\dots,d_n)}$ can be written as a linear combination of products of $E_{-n,k}$'s using relation \eqref{eqn:eee}. \\

\begin{remark}

Analogously, one can define elements $F_{(d_1,\dots,d_n)} \in \CA$ instead of $E_{(d_1,\dots,d_n)}$ by replacing $E_{-n,k}$ with $E_{n,k}$ in \eqref{eqn:formula e}. \\

\end{remark}

\noindent The elements \eqref{eqn:new e} were constructed using the isomorphism between $\CA$ and a double shuffle algebra (see \cite{Shuf}), and this allows us to prove the following. \\
	
	\begin{proposition}
		\label{prop:formula}
		
		For any $d_1,\dots,d_n,k \in \BZ$, we have (note that $E_{(k)} = E_k$)
		\begin{multline}
		[E_{(d_1,\dots,d_n)}, E_k] = (1-q_1)(1-q_2) \\ \sum_{i=1}^n \begin{cases} - \sum_{k \leq a < d_i} E_{(d_1,\dots,d_{i-1},a,d_i+k-a,d_{i+1},\dots,d_n)}  & \text{if } d_i > k   \\ \\ \sum_{d_i \leq a < k} E_{(d_1,\dots,d_{i-1},a,d_i+k-a,d_{i+1},\dots,d_n)} & \text{if } d_i < k \end{cases} 
		\label{eqn:new formula}
		\end{multline}
		There is no summand corresponding to any $i$ such that $d_i = k$. \\
		
	\end{proposition}
	
	\noindent Proposition \ref{prop:formula} may be proved by embedding the negative half of $\CA$ (i.e. the subalgebra generated by $\{E_k\}_{k \in \BZ}$) in the algebra $\BA_{q,t}$ of \cite{CGM} via the last formula of \loccitt \ Alternatively, we recall that that the elliptic Hall algebra $\CA$ acts on the $K$--theory group of the moduli space of rank $r$ framed sheaves on $\BP^2$ (\cite{FT,SV}, see  \cite{Mod} for our conventions). This statement may be interpreted as an equivariant version of the results in the present paper, and our proof carries through in the equivariant case. Therefore, \eqref{eqn:new relations} implies that relation \eqref{eqn:new formula} holds in the level $r$ representation of $\CA$. Since any non-zero element of $\CA$ acts non-trivially in some level $r$ representation for $r$ large enough, this implies that relations \eqref{eqn:new formula} hold in $\CA$. \\
	
	\begin{proposition}
		\label{prop:equivalent}
		
		Relations \eqref{eqn:rel 1} and \eqref{eqn:rel 3} both follow from \eqref{eqn:new formula}. \\
		
	\end{proposition}
	
	\begin{proof} Taking the coefficient of $z^{-m} w^{-n}$ in relation \eqref{eqn:rel 1} shows it is equivalent to
$$
E_{m+3}E_n - \left(q_1+q_2+\frac 1q\right) E_{m+2}E_{n+1} + \left(\frac 1{q_1}+\frac 1{q_2}+q\right)E_{m+1}E_{n+2} - E_m E_{n+3} = 
$$
$$
= E_n E_{m+3} -  \left(\frac 1{q_1}+\frac 1{q_2}+q\right) E_{n+1}E_{m+2} + \left(q_1+q_2+\frac 1q\right) E_{n+2}E_{m+1} - E_{n+3}E_m
$$
for all $m,n \in \BZ$. We may rewrite the relation above as
		\begin{multline*}
		\left[E_{m+3}E_n - E_{m+2}E_{n+1} \left(\frac 1q + 1 + q \right) + E_{m+1}E_{n+2}  \left(\frac 1q + 1 + q \right) - E_{m}E_{n+3} \right] - \\
		- \left[E_{n} E_{m+3} - E_{n+1}E_{m+2} \left(\frac 1q + 1 + q \right) + E_{n+2}E_{m+1} \left(\frac 1q + 1 + q \right) - E_{n+3}E_{m} \right] = \\
		= (1-q_1)(1-q_2)\left( \frac {E_{n+1} E_{m+2}}q - E_{n+2}E_{m+1} - E_{m+2}E_{n+1} + \frac {E_{m+1}E_{n+2}}q \right)
		\end{multline*}
		In the right-hand side, we may convert every $E_m E_n$ into $E_{(m,n)} - q E_{(m-1,n+1)}$ using relation \eqref{eqn:eee}, and so the formula above becomes equivalent to
		$$
		[E_{m+3}, E_n] - \left(\frac 1q + 1 + q \right) [E_{m+2}, E_{n+1}] + \left( \frac 1q + 1 + q \right) [E_{m+1},E_{n+2}] - [E_m, E_{n+3}] = 
		$$
		$$
		= (1-q_1)(1-q_2) \cdot
		$$
		$$
		\left[ (E_{(n+1,m+2)} + E_{(m+1,n+2)}) \left(\frac 1q + q \right) - E_{(n,m+3)} - E_{(n+2,m+1)} - E_{(m+2,n+1)} - E_{(m,n+3)} \right]
		$$
		It is easy to check that the formula above follows by applying \eqref{eqn:new formula} to the four Lie brackets in the left-hand side. As for \eqref{eqn:rel 3}, it follows from the fact that
		$$
		[E_{k-1},E_{k+1}] = E_{(k-1,k+1)} + E_{(k,k)}
		$$
		together with
		$$
		[E_{(k-1,k+1)}+E_{(k,k)},E_k] = [E_{(k-1,k+1)},E_k]+[E_{(k,k)},E_k] = E_{(k-1,k,k+1)}-E_{(k-1,k,k+1)} = 0
		$$
both of which are special cases of \eqref{eqn:new formula}. 
		
	\end{proof}
	
	\subsection{}
	\label{sub:action}
	
	We would like to compare the abstract algebra elements $E_{(d_1,\dots,d_n)} \in \CA$ of the previous Subsections with the explicit maps $e_{(d_1,\dots,d_n)} : \km \rightarrow \kms$ from Subsection \ref{sub:k theory}. The goal would be to say that we have an action of the algebra $\CA$ on $K_\CM$, but this statement requires care: taken literally, having an action means that any element of $\CA$ gives rise to a homomorphism $K_\CM \rightarrow K_\CM$. In formula \eqref{eqn:e in k-theory}, we see that this is not quite the case, so we must adapt the notion of ``action". \\
	
	\begin{definition}
		\label{def:notation}
		
		For any group homomorphisms $x,y : K_\CM \rightarrow K_{\CM \times S}$, define
		$$
		xy|_\Delta = \Big\{ \km \xrightarrow{y} \kms \xrightarrow{x \times \emph{Id}_S} \kmss \xrightarrow{(\emph{Id}_\CM \times \Delta)^*} \kms \Big\}
		$$
		where $\Delta : S \hookrightarrow S \times S$ is the diagonal embedding. Also define
		\begin{multline*}
		[x,y] = \left\{ \km \xrightarrow{y} K_{\CM \times S_2} \xrightarrow{x \times \emph{Id}_{S_2}} K_{\CM \times S_1 \times S_2} \right\} - \\
		-  \left\{ \km \xrightarrow{x} K_{\CM \times S_1} \xrightarrow{y \times \emph{Id}_{S_1}} K_{\CM \times S_1 \times S_2}  \right\}
		\end{multline*}
		where $S_1 = S_2 = S$, but we use different labels in order to emphasize the fact that $x$ takes values in the first factor of $S$ and $y$ takes values in the second factor. \\
		
	\end{definition}
	
	\noindent The notions in Definition \ref{def:notation} satisfy associativity, in the sense that
	\begin{equation} 
	\label{eqn:assoc}
	(xy|_\Delta )z|_\Delta = x(yz|_\Delta)|_\Delta 
	\end{equation}
	as homomorphisms $\km \rightarrow \kms$. Therefore, we will unambiguously use the notation $xyz|_\Delta$. We also have the following version of the Jacobi identity
	\begin{equation} 
	\label{eqn:jacobi 0}
	[[x,y],z]+[[y,z],x]+[[z,x],y] = 0
	\end{equation}
	as homomorphisms $\km \rightarrow \kmsss$. \\
	
	\begin{definition}
		\label{def:diagonal}
		
		If two homomorphisms $x,y : \km \rightarrow \kms$ satisfy 
		\begin{equation}
		\label{eqn:jov}
		[x,y] = \Delta_*(z)
		\end{equation}
		for some homomorphism $z : \km \rightarrow \kms$, we will write
		\begin{equation}
		\label{eqn:def red}
		[x,y]_\ered = z
		\end{equation}
		Note that formula \eqref{eqn:def red} is unambiguous, since $\Delta_* : \ks \rightarrow \kss$ is injective (as it has a left inverse given by projection to the first factor), and so $z$ in \eqref{eqn:jov} is unique. \\
		
	\end{definition}
	
	\noindent We have the following version of the Leibniz rule
	\begin{equation}
	\label{eqn:leibniz}
	[x,yz|_\Delta]_\red = [x,y]_\red z|_\Delta + y [x,z]_\red|_\Delta 
	\end{equation}
	which is an equality of homomorphisms $\km \rightarrow \kms$. Formula \eqref{eqn:leibniz} follows from base change in the derived fiber square
	$$
	\xymatrix{
		S  \ar@{^{(}->}[d]_-\Delta \ar@{^{(}->}[r]^-{\Delta} & S \times S  \ar@{^{(}->}[d]^-{\Delta \times \text{Id}_S} \\
		S \times S \ar@{^{(}->}[r]^-{\text{Id}_S \times \Delta} & S\times S \times S} \qquad \qquad (\text{Id}_S \times \Delta)^* \circ (\Delta \times \text{Id}_S)_* = \Delta_* \circ \Delta^*
	$$
	and we leave the details to the interested reader. Similarly, \eqref{eqn:jacobi 0} implies
	\begin{equation} 
	\label{eqn:jacobi}
	[[x,y]_\red,z]_\red+[[y,z]_\red,x]_\red+[[z,x]_\red,y]_\red = 0
	\end{equation}
	as homomorphisms $\km \rightarrow \kms$. \\
	
	\subsection{} 
	
	For the remainder of the paper, the parameters $q_1,q_2$ over which the algebra $\CA$ is defined will be reinterpreted such that
	$$
	q_1 + q_2 = [\Omega_S^1], \qquad q = q_1q_2 = [\omega_S]
	$$
While $q_1,q_2$ are not themselves well-defined in $\ks$, any symmetric Laurent polynomial in $q_1,q_2$ is indeed well-defined. In other words, we obtain a ring homomorphism
	$$
	\kk \rightarrow K_S
	$$
	where $\kk$ is the ring \eqref{eqn:ring loc}. In order for the ring homomorphism above to be well-defined, we need the quantity $1+q^{-1}+\dots+q^{-n+1}$ to be invertible in $\ks$, and the solution to achieving this is to work with $K$-theory with $\BQ$ coefficients (indeed, $q$ is unipotent in $\ks$, so $1+q^{-1}+\dots+q^{-n+1}$ is equal to $n$ times a unit). Define
\begin{equation}
\label{eqn:red com algebra}
	[X,Y]_\red = \frac {[X,Y]}{(1-q_1)(1-q_2)}
\end{equation}
	for all $X,Y \in \CA$, and note that the right-hand side is well-defined due to the fact that all commutators in the algebra $\CA$ are multiples of $(1-q_1)(1-q_2)$, see \eqref{eqn:comm}. \\
	
	\begin{definition}
		\label{def:action}
		
		An action $\CA \curvearrowright \km$ is an abelian group homomorphism
		\begin{equation}
		\label{eqn:action}
		\CA \stackrel{\Phi}\longrightarrow \emph{Hom}(\km, \kms)
		\end{equation}
		such that for all $X,Y \in \CA$, we have
		\begin{equation}
		\label{eqn:hom}
		\Phi(XY) = \Phi(X)\Phi(Y)|_\Delta
		\end{equation}
		\begin{equation}
		\label{eqn:comm phi}
		[\Phi(X), \Phi(Y)]_\ered = \Phi ([X,Y]_\ered)
		\end{equation}

	\end{definition}
	
	\bigskip
	
	\noindent The compatibility between relations \eqref{eqn:hom} and \eqref{eqn:comm phi} is visible when restricting the latter formula to the diagonal $\Delta :  S \hookrightarrow S \times S$, as follows:
	$$
	\Phi(X)\Phi(Y)|_\Delta - \Phi(Y)\Phi(X)|_\Delta = \Delta^* \Delta_* \left[ \Phi \left( \frac {XY-YX}{(1-q_1)(1-q_2)} \right) \right] = 
	$$
	\begin{equation}
	\label{eqn:compatibility}
	= \Phi \left( \frac {XY-YX}{(1-q_1)(1-q_2)} \right) \cdot \wedge^\bullet(\CN^\vee_{\Delta \hookrightarrow S \times S}) = \Phi(XY-YX)
	\end{equation}
	because the normal bundle $\CN_{\Delta \hookrightarrow S \times S}$ is isomorphic to the tangent bundle to $S$, hence the exterior algebra of its dual equals $(1-q_1)(1-q_2) = [\wedge^\bullet(\Omega_S^1)]$. \\
	
	\begin{theorem}
		\label{thm:a infty acts}
		
		There exists an action $\CA \curvearrowright \km$ as in Definition \ref{def:action}, where
		$$
		\Phi(E_{(d_1,\dots,d_n)}) = e_{(d_1,\dots,d_n)} \qquad \text{and} \qquad \Phi(F_{(d_1,\dots,d_n)}) = f_{(d_1,\dots,d_n)}
		$$
		Moreover, the series $H^\pm(z) \in \CA$ of \eqref{eqn:gan} act by \eqref{eqn:h in k-theory}. \\
		
	\end{theorem}
	
	\begin{proof} Recall from \cite{BS} (see \cite{W univ} for our notations) that a $\kk$-basis of $\CA$ is given by
		\begin{equation}
		\label{eqn:elements}
		E_v = E_{n_1,k_1} \dots E_{n_t,k_t}
		\end{equation}
		as $v = \{(n_1,k_1), \dots ,(n_t,k_t)\}$ goes over all convex paths of lattice points, i.e. such that the slopes $k_i/n_i$ are ordered clockwise starting from the negative $y$ axis (formula \eqref{eqn:comm heis} implies that $E_{n,k}$ and $E_{n',k'}$ commute if $(n,k) \in \BR_{>0}(n',k')$, so we may order them arbitrarily). Thus, we define the action $\Phi : \CA \rightarrow \End(\km, \kms)$ by
		\begin{align}
		&\Phi(E_{-n,k}) = q^{\gcd(n,k)-1} e_{(d^{(n,k)}_1,\dots,d^{(n,k)}_{n})} \quad \text{of \eqref{eqn:e in k-theory}} \label{eqn:phi negative e} \\
		&\Phi(E_{n,k}) = q^{\gcd(n,k)-1} f_{(d_1^{(n,k)},\dots,d^{(n,k)}_{n})} \ \ \quad \text{of \eqref{eqn:f in k-theory}} \label{eqn:phi positive e}
		\end{align}
		for any $n > 0$, $k \in \BZ$, where $d_i^{(n,k)}$ are defined in \eqref{eqn:d i n k}. Moreover, we have
$$
\Phi(E_{0,k}) = \text{multiplication by }\wedge^k \CU
$$
Then requirement \eqref{eqn:hom} forces us to set
$$
\Phi(E_v) = \Phi(E_{n_1,k_1}) \dots \Phi(E_{n_t,k_t}) \Big|_\Delta
$$
Together with $\kk$-linearity, this completely determines $\Phi$. \\
		
		\begin{claim}
			\label{claim:o}
			
			For any lattice points $(n',k') \ccur (n,k)$, we have
			\begin{multline}
			[\Phi(E_{n,k}), \Phi(E_{n',k'})]_\ered = \\
			= \sum^{t \geq 1, \sum_{i=1}^t n_i = n+n', \sum_{i=1}^t k_i = k+k'}_{(n',k') \ccur (n_1,k_1) \ccur \dots \ccur (n_t,k_t) \ccur (n,k)} \Phi(E_{n_1,k_1}) \dots \Phi(E_{n_t,k_t}) \Big|_\Delta \cdot \emph{coeff} \label{eqn:rely}
			\end{multline}
			where the coefficients in the right-hand side match the analogous ones in the right-hand side of \eqref{eqn:comm}, but with $q_1$ and $q_2$ interpreted as the Chern roots of $\Omega_S^1$. \\
			
\end{claim}
		
\noindent To conclude the proof of Theorem \ref{thm:a infty acts}, we must show that relations \eqref{eqn:rely} imply that $\Phi$ satisfies properties \eqref{eqn:hom} and \eqref{eqn:comm phi}; we will only take care of the former of these, as the latter is analogous and so left to the interested reader. We will use the ``straightening" argument of \cite{BS}: one may define the elements \eqref{eqn:elements} of $\CA$ for an arbitrary path $v$ of lattice points. The convexification of $v$ is the path consisting of the same collection of lattice points, but ordered in clockwise order of slope, starting from the negative $y$ axis. The area of a path, denoted by $a(v)$, is defined as the area between the path $v$ and its convexification. Therefore, we will show that
		\begin{equation}
		\label{eqn:kris}
		\Phi(E_v) = \Phi(E_{n_1,k_1}) \dots \Phi(E_{n_t,k_t}) \Big|_\Delta
		\end{equation}
		holds for an arbitrary path $v$, by induction on $a(v)$ (this claim will establish \eqref{eqn:hom}). The base case is when $a(v) = 0$, i.e. $v$ is already convex, in which case \eqref{eqn:kris} is simply the definition of $\Phi(E_v)$. A non-convex path $v$ may be written as
		$$
		v = \Big \{ \dots , (n_i,k_i), (n_{i+1},k_{i+1}), \dots \Big \} 
		$$
		for some $i$ such that $k_i/n_i < k_{i+1}/n_{i+1}$. Let $v'$ denote the path obtained from $v$ by switching the lattice points $(n_i,k_i)$ and $(n_{i+1},k_{i+1})$. It is elementary to observe that $a(v') < a(v)$, and moreover the straightening argument of \cite{BS} shows that
		\begin{equation}
		\label{eqn:george}
		E_v = E_{v'} + (1-q_1)(1-q_2) \sum_{v''} E_{v''} \cdot \text{coeff}
		\end{equation}
		where the sum goes over paths with $a(v'') < a(v)$, and the coefficients are induced from \eqref{eqn:comm}. Because of \eqref{eqn:rely} and the induction hypothesis, we have
	\begin{equation}
		\label{eqn:noel}
		\dots \Phi(E_{n_i,k_i}) \Phi(E_{n_{i+1},k_{i+1}}) \dots \Big|_\Delta = \Phi(E_{v'}) + \wedge^\bullet(\CN^\vee_{\Delta \hookrightarrow S \times S}) \sum_{v''} \Phi(E_{v''}) \cdot \text{coeff}
		\end{equation}
with the coefficients in the RHS being the same ones as in \eqref{eqn:george} (indeed, the two computations \eqref{eqn:george} and \eqref{eqn:noel} take as inputs \eqref{eqn:comm} and \eqref{eqn:rely}, respectively, so it should be no surprise that they produce the same coefficients). Formula \eqref{eqn:noel} yields the induction step of \eqref{eqn:kris}, which as we have seen establishes \eqref{eqn:hom}. \\
		
\begin{proof} \emph{of Claim \ref{claim:o}}: let us first give a slightly imprecise argument. By Theorem \ref{thm:algebra}, the commutation relations \eqref{eqn:comm} can be deduced from relations \eqref{eqn:rel 1}--\eqref{eqn:rel 4}. As the latter hold with $E_k, F_k,H^\pm_k$ replaced by $e_k = \Phi(E_k)$, $f_k = \Phi(F_k)$, $h^\pm_k = \Phi(H^\pm_k)$ of Subsection \ref{sub:k theory} (this fact was proved in \cite{Univ}), one might conclude the Claim. \\

\noindent The imprecision in the previous paragraph is that, strictly speaking, Theorem \ref{thm:algebra} requires one to invert the scalar $(1-q_1)(1-q_2)$. When $q_1$ and $q_2$ are set equal to the Chern roots of $\Omega_S^1$, this scalar is manifestly a zero divisor, so inverting it is off the table. However, the ace up our sleeve is the operation $(x,y) \mapsto [x,y]_{\text{red}}$ for homomorphisms $x,y : \km \rightarrow \kms$, and its counterpart \eqref{eqn:red com algebra} for elements of $\CA$. We thus revisit the argument of \cite{S}, which consisted of the following main steps: \\

\begin{enumerate}[leftmargin=*]

\item Define $E_{n,k} \in \CA$ as iterated reduced commutators \eqref{eqn:red com algebra} of $E_{\pm 1,l} \in \CA$ via \eqref{eqn:comm exp} (and its generalization to the case when $\gcd(n,k) > 1$, see \cite{W univ}). In the geometric case, we emulate this procedure and define $\Phi(E_{n,k}) : \km \rightarrow \kms$ as iterated reduced commutators \eqref{eqn:def red} of $\Phi(E_{\pm 1,l}) : \km \rightarrow \kms$. \\

\item Show that relations \eqref{eqn:comm} among the elements $E_{n, k} \in \CA$ thus constructed can be obtained from relations \eqref{eqn:rel 1}--\eqref{eqn:rel 4}. While combinatorially quite involved, this part of the argument of \cite{S} only uses the Lebniz rule and Jacobi identities. By \eqref{eqn:leibniz} and \eqref{eqn:jacobi}, the same argument therefore applies equally well to show that relations \eqref{eqn:rely} among the operators $\Phi(E_{n, k}) : \km \rightarrow \kms$ can be obtained from the geometric counterparts of relations \eqref{eqn:rel 1}--\eqref{eqn:rel 4}, which we know to hold. \\

\end{enumerate}

\noindent The only thing that remains to be proved is that the operators $\Phi(E_{n,k})$ constructed as in item (1) above match the operators in the right-hand sides of \eqref{eqn:phi negative e} and \eqref{eqn:phi positive e}. This happens because relations \eqref{eqn:eee} and \eqref{eqn:new formula} in $\CA$ are matched by relations \eqref{eqn:excess relations} and \eqref{eqn:new relations} between operators $\km \rightarrow \kms$.
	
	\end{proof}	
		
	\end{proof}
	
	\section{Dimension estimates}
	\label{sec:app}
	
\medskip	

\subsection{}
\label{sub:app intro}

The main purpose of the present Section is to obtain certain estimates on the dimensions of various strata of the schemes $\fZ_\lambda$ of Definition \ref{def:z}, when $n = |\lambda| \leq 4$. In what follows, the \textbf{dimension} of a scheme will refer to the maximum of the dimensions of its irreducible components. Consider the map
$$
\fZ_\lambda \xrightarrow{\pi} \CM \times S^{n}
$$
which takes a flag \eqref{eqn:z} to $(\CF_n, x_1, \dots, x_n)$. The target of this map is smooth but disconnected, with the dimension of its connected components depending linearly on the second Chern class of $\CF_n$. Thus, the phrase ``dimension (resp. number of irreducible components) of $\fZ_\lambda$" will refer to the dimension (resp. number of irreducible components) of $\pi^{-1}(C)$ for any fixed connected component $C$ of $\CM \times S^n$. The analogous terminology will apply to locally closed subsets of $\fZ_\lambda$. \\
	
	\subsection{} 
	\label{sub:quot}
	
\noindent	Any stable sheaf $\CF$ of rank $r > 0$ on a smooth surface is torsion-free, hence it injects into its double dual $\CF \hookrightarrow \CF^{\vee\vee}$. The double dual is a reflexive sheaf, hence locally free (over a smooth surface). Therefore, we have an injection
	\begin{equation}
	\label{eqn:double dual}
	\CF \hookrightarrow \CV
	\end{equation}
	with $\CV$ locally free. Assumption A implies that $\CF$ is stable iff $\CV$ is stable (see \cite[Proposition 5.5]{Univ}), and since stable sheaves over an algebraically closed field are simple, this implies that the inclusion \eqref{eqn:double dual} is unique up to constant multiple. The quotient $\CV/\CF$ is supported at finitely many distinct closed points $x_1,\dots,x_k \in S$, hence
	\begin{equation}
	\label{eqn:finite length}
	\CV/\CF \cong \CQ_1 \oplus \dots \oplus \CQ_k
	\end{equation}
	where $\CQ_i$ is a finite length sheaf supported at the closed point $x_i \in S$. The length $d_i$ of the sheaf $\CQ_i$ will be called the \textbf{defect} of $\CF$ at the point $x_i$, since it measures how far $\CF$ is from being locally free in the vicinity of $x_i$. \\
	
	\begin{definition}
		\label{def:quot 1}
		
		Let $\equot_d$ denote the moduli space of quotients of a rank $r$ locally free sheaf $\CV$, which have length $d$ and are supported at a given closed point $x \in S$. \\
		
	\end{definition}
	
	\noindent As the notation suggests, the variety $\quot_d$ does not depend (up to isomorphism) on the choices of $\CV,x,S$, as long as the latter is a smooth surface. \\
	
	\begin{theorem}
		\label{thm:quot 1}
		
		(\cite{Ba, EL}) If $d > 0$, $\equot_d$ is irreducible of dimension $rd-1$. \\
		
	\end{theorem}
	
	\noindent In order to produce a stable sheaf $\CF$ with prescribed defects at distinct points $x_1, \dots, x_k \in S$, one must start from a stable locally free sheaf $\CV$ with invariants $(r,c_1)$, and modify it according to a point in the irreducible algebraic variety $\quot_{d_1} \times \dots \times \quot_{d_k}$. This means that the locally closed subscheme which parameterizes sheaves with defects $d_1,\dots,d_k$ at given distinct points $x_1,\dots,x_k$ (and no defect anywhere else) has codimension $r(d_1+\dots+d_k) + k$ in $\CM \times S^k$. \\

	\subsection{} 
	\label{sub:connected}
	
	Recall the schemes introduced in Definition \ref{def:z}, namely
	\begin{equation}
	\label{eqn:flag defect}
	\fZ_\lambda = \Big\{(\CF_0 \subset_{p_1} \CF_1 \subset_{p_2} \dots \subset_{p_n} \CF_n) \text{ such that } p_i = p_j \text{ if } i \sim j \text{ in }\lambda \Big\}
	\end{equation}
	where $\lambda$ is a set partition of size $n$. Assume that the set partition $\lambda$ consists of $k$ distinct parts with multiplicities $n_1,\dots,n_k$, and we will use the notation $x_1,\dots,x_k$ for the points of $S$ corresponding to the distinct parts of $\lambda$. All the sheaves in a flag \eqref{eqn:flag defect} have isomorphic double duals, hence they are all contained in one and the same locally free sheaf $\CV$. Therefore, we let
\begin{equation}
\label{eqn:defect z}
	\fZ_\lambda^{\defect d_1,\dots,d_k} \subset \fZ_\lambda
\end{equation}
	be the locally closed subscheme where the support points $x_1,\dots,x_k \in S$ of the flag \eqref{eqn:flag defect} are all distinct, and the sheaf $\CF_n$ has defect $d_i$ at the point $x_i$ for all $i$ (note that $\CF_n$ may have arbitrary defect at points different from $x_1,\dots,x_k$). \\
	
	\begin{definition}
		\label{def:quot 2}
		
		Let $\equot_{d,\dots,d+n}$ denote the moduli space of quotients
		$$
		\CV \twoheadrightarrow \CQ_{d+n} \twoheadrightarrow \dots \twoheadrightarrow \CQ_d
		$$
		where $\CV$ is locally free, and the sheaves $\CQ_{d+n},\dots,\CQ_d$ have lengths $d+n,\dots,d$, respectively, and are supported at one and the same given closed point $x \in S$. \\
		
	\end{definition}
	
	\noindent We will give an explicit construction of the schemes $\quot_{d,\dots,d+n}$ in the next Subsection, as well as some dimension estimates when $n \leq 4$ (as opposed from the scheme in Definition \ref{def:quot 1}, their dimensions grow wildly for large $n$, and they are far from being irreducible). The construction will not depend on $\CV, x, S$, and it will be clear that the same principle as in the previous Subsection applies. Namely, in order to construct a flag \eqref{eqn:flag defect} which lies in the subscheme \eqref{eqn:defect z}, one \\
	
	\begin{enumerate}[leftmargin=*]
		
		\item chooses a locally free sheaf $\CV$ \\
		
		\item chooses distinct points $x_1,\dots,x_k \in S$ \\
		
		\item chooses a point in the variety
		$$
		\quot_{d_1, \dots, d_1 + n_1} \times \dots \times \quot_{d_k, \dots, d_k + n_k}
		$$
($n_i$ denotes the multiplicity of $x_i$ in $\lambda$) which determines the sheaves $\CF_0,\dots,\CF_n$ as modifications of $\CV$ in the vicinity of the chosen points $x_1,\dots,x_k$. \\
		
	\end{enumerate}
	
\noindent More rigorously, the discussion above implies the following. \\

\begin{lemma}
\label{lem:locally isomorphic}

The scheme $\fZ^{\emph{def }d_1,\dots,d_k}_\lambda$ is locally isomorphic to
$$
\CM \times S^k \times \equot_{d_1, \dots, d_1 + n_1} \times \dots \times \equot_{d_k, \dots, d_k + n_k}
$$
\end{lemma}

\bigskip	
	
	\noindent As a consequence of Lemma \ref{lem:locally isomorphic}, we obtain the following result (recall the conventions of Subsection \ref{sub:app intro} when discussing dimension and the number of irreducible components of locally closed subsets of $\fZ_\lambda$, such as $\fZ_\lambda^{\defect d_1,\dots,d_k}$). \\
	
	\begin{corollary}
		\label{cor:def 1}
		
For any sequence of defects $d_1,\dots,d_k \geq 0$ and any set partition $\lambda$ whose parts have multiplicities $n_1,\dots,n_k \in \BN$, we have
		\begin{multline}
		\dim \fZ_\lambda^{\edefect d_1,\dots,d_k} = \econst + r(c^\efirst_2 + c_2^\elast) + k + \\
		+ \sum_{i=1}^k \left( \dim \equot_{d_i,\dots,d_i+n_i} + 1 - 2rd_i - rn_i \right) \label{eqn:def 1}
		\end{multline}
		where $c_2^\efirst$, $c_2^\elast$ are the second Chern classes of the sheaves labeled $\CF_0$ and $\CF_n$ in \eqref{eqn:flag defect}, and $\econst$ is the same constant that appears in \eqref{eqn:dim m}. \\
		
	\end{corollary}
	
	\begin{proof} Because the second Chern class of the locally free sheaf $\CV$ is
		$$
		c_2 = c_2^\last - \sum_{i=1}^k d_i = c_2^\first - \sum_{i=1}^k (d_i+n_i)
		$$
		the dimension of the moduli space parametrizing the locally free sheaf $\CV$ is
		$$
		\const + r(c_2^\first + c_2^\last) - r \sum_{i=1}^k (2d_i+n_i)
		$$
		The choice of the support points $x_1,\dots,x_k$ adds $2k$ dimensions, and the choice of flag $\CF_0 \subset \dots \subset \CF_n \subset \CV$ with the given defects adds $\sum_{i=1}^k \dim \quot_{d_i,\dots,d_i+n_i}$ dimensions to our estimate (according to item (3) above Lemma \ref{lem:locally isomorphic}). \\ 	
	\end{proof}
	
	\subsection{}
	\label{sub:adhm}
	
	Let us now describe the schemes $\quot_d$ and $\quot_{d,\dots,d+n}$ explicitly. Since these do not depend on the choice of the locally free sheaf $\CV$, or of a closed point $x$ on a smooth surface $S$, we may as well consider $\CV = \CO^{\oplus r}$, $S = \BA^2$ and $x = (0,0)$. Therefore, we have
	\begin{align*}
	&\quot_d = \Big\{\CO_{\BA^2}^{\oplus r} \twoheadrightarrow \CQ, \text{length }\CQ = d, \text{supp } \CQ = (0,0) \Big\} \\
	&\quot_{d,\dots,d+n} = \Big\{\CO_{\BA^2}^{\oplus r} \twoheadrightarrow \CQ_{d+n} \twoheadrightarrow \dots \twoheadrightarrow \CQ_d, \text{length }\CQ_i = i, \text{supp } \CQ_i = (0,0) \Big\}
	\end{align*}
	As vector spaces, we may consider isomorphisms $\CQ_i \cong \BC^i$, and the $\CO_{\BA^2} = \BC[x,y]$ module structure on $\CQ_i$ can be packaged by providing two nilpotent commuting endomorphisms $X,Y \in \End(\BC^i)$. A surjective homomorphism $\CO_{\BA^2}^{\oplus r} \twoheadrightarrow \CQ_i$ is the same datum as $r$ vectors $v_1,\dots,v_r \in \BC^i$ which are cyclic. Recall that cyclicity means that the vector space $\BC^i$ is generated by polynomials in $X$ and $Y$ acting on linear combinations of $v_1,\dots,v_r$. In other words, let us construct the $i \times r$ matrix
	$$
	v = (v_1,\dots,v_r) \in \Hom(\BC^r, \BC^i)
	$$
and let $\End_0(\BC^i)$ denote the vector space of nilpotent endomorphisms of $\BC^i$. Then
	\begin{multline*}
	\quot_d = \Big\{ (X,Y,v) \in \End_0(\BC^d) \times \End_0(\BC^d) \times \\
	\times \Hom(\BC^r, \BC^d) \text{ s.t. } [X,Y] = 0, \ v \text{ cyclic} \Big\}/G
	\end{multline*}
	where we take the quotient by the $G = GL_d$ action in order to factor out the ambiguity of the chosen isomorphism $\CQ \cong \BC^d$. Explicitly, $g \in G$ acts by
	$$
	(X,Y,v) \leadsto (gXg^{-1}, gYg^{-1}, gv)
	$$
	and the cyclicity of $v$ implies that the $G$ action is free. The quotient by $G$ is geometric, because the fact that we restrict to the open subscheme of cyclic triples $(X,Y,v)$ is precisely equivalent to restricting to the locus of stable points (as in geometric invariant theory) with respect to the inverse of the determinant character. Similarly,
	\begin{multline}
	\label{eqn:flag of quot}
	\quot_{d,\dots,d+n} = \Big\{ (X,Y,v) \in \End_0(\BC^{d+n} \twoheadrightarrow \dots \twoheadrightarrow \BC^d) \times \End_0(\BC^{d+n} \twoheadrightarrow \dots \twoheadrightarrow \BC^d) \times \\
	\times  \Hom(\BC^r, \BC^{d+n}) \text{ s.t. } [X,Y] = 0, \ v \text{ cyclic} \Big\}/P \qquad \qquad
	\end{multline}
	where $\End_0(\BC^{d+n} \twoheadrightarrow \dots \twoheadrightarrow \BC^d)$ denotes the vector space of nilpotent endomorphisms of $\BC^{d+n}$ which preserve a fixed flag of quotients of the specified dimensions (and $P \subset GL_{d+n}$ denotes the subgroup of automorphisms with the same property). If one is apprehensive about taking the quotient of a quasiprojective algebraic variety by the non-reductive group $P$, then one must employ the following trick. Instead of fixing a flag of quotients, let it vary and replace the datum $(X,Y,v)$ by
	$$
	(X,Y,v, \BC^{d+n} \twoheadrightarrow V_{d+n-1} \twoheadrightarrow \dots \twoheadrightarrow V_d)
	$$
	where $\dim V_i = i$ and the maps $X, Y \in \End_0(\BC^{d+n})$ are required to preserve the flag $\{V_i\}_{d \leq i < d+n}$. Then the $P$ quotient should be replaced by a $GL_{d+n}$ quotient. We will refrain from making this modification, so as to keep the presentation clear. \\
	
	\subsection{}
	\label{sub:stack 4}
	
	We fix a basis of $\BC^i$ so that the endomorphism $X$ takes the form
	\begin{equation}
	\label{eqn:blocks}
	\left( \begin{array}{cc|ccc} 
	* & * & 0 & 0 & 0 \\
	* & * & 0 & 0 & 0 \\ \hline
	* & * & 0 & 0 & 0 \\
	* & * & * & 0 & 0 \\
	* & * & * & * & 0 \end{array} \right) \qquad = \qquad \left( \begin{array}{c|c} X_0 & 0 \\ \hline x & X' \end{array} \right)
	\end{equation}
	where the block in the top left is an $d \times d$ matrix, and the lower triangular block in the bottom right is $n \times n$. We will use the notations $X_0, x, X'$ for the three blocks, as pictured above, and employ the analogous notations for $Y$. By analogy with the constructions in Subsection \ref{sub:adhm}, we define the affine algebraic variety
	$$
	\comm_n = \Big\{ (X,Y) \in \End_0(\BC^n \twoheadrightarrow \dots \twoheadrightarrow \BC^1)  
	\times \End_0(\BC^n \twoheadrightarrow \dots \twoheadrightarrow \BC^1) \text{ s.t. } [X,Y] = 0 \Big\}
	$$
	and the stack
	$$
	\stack_n = \comm_n/B
	$$
	where $B \subset GL_n$ denotes the Borel subgroup of lower triangular matrices. The variety $\comm_n$ has been studied in the literature, and we refer the reader to \cite{Basili} for a survey, where the author proves that the dimension and number of irreducible components of $\comm_n$ grow wildly as $n$ increases. However, we are interested only in the particular cases $n \in \{1,2,3,4\}$, when this variety is behaved rather mildly. \\
	
	\begin{proposition}
		\label{prop:basili}
		
		When $n \leq 3$, the variety $\ecomm_n$ is irreducible of dimension
		$$
		\frac {n(n+1)}2 - 1 = \dim B - 1
		$$
		When $n = 4$, it has the dimension above, but consists of two irreducible components. \\
		
	\end{proposition}
	
	\noindent For example, we have
	$$
	\comm_3 = \text{Spec } \BC[x_{21}, x_{31}, x_{32}, y_{21}, y_{31}, y_{32}] / (x_{32}y_{21} - y_{32}x_{21})
	$$
	the intuition being that $\{ x_{ij}, y_{ij} \}_{3 \geq i>j \geq 1}$ are the entries of the matrices $X$ and $Y$, and the expression $x_{32}y_{21} - y_{32}x_{21}$ is the only non-zero entry of the commutator $[X,Y]$. Similarly, $\comm_4$ is a subvariety of affine space with 12 coordinates $\{ x_{ij}, y_{ij} \}_{4 \geq i>j \geq 1}$, cut out by 3 equations. Its irreducible components are given by \\
	
	\begin{enumerate}
		
		\item the subvariety $Z_1 \subset \comm_4$ cut out by $x_{32} = y_{32} = 0$. Note that two of the three sub-diagonal coefficients of $[X,Y]$ vanish identically on $Z_1$, thus
		$$
		\dim Z_1 = 10-1 = 9
		$$
		
		\item the subvariety $Z_2 \subset \comm_4$ defined as the closure of the locus $(x_{32}, y_{32}) \neq (0,0)$, for which it is not hard to observe that
		$$
		\dim Z_2 = 12-3 = 9
		$$
		
	\medskip	
		
	\end{enumerate} 
	
	\noindent It is straightforward to prove that $\comm_4$ is generically reduced. Its singular locus is precisely the intersection $Z_1 \cap Z_2$, which has codimension 1 in $\comm_4$. \\
	
	\subsection{} Consider the natural map (notation as in \eqref{eqn:blocks})
	\begin{equation}
	\label{eqn:fiber map}
	\xymatrix{\quot_{d,\dots,d+n} \ar[d]^\zeta \\ \quot_d \times \stack_n} \qquad \qquad (X,Y,v) \mapsto (X_0,Y_0,v_0) \times (X',Y')
	\end{equation}
	where $v_0 \in \Hom(\BC^r,\BC^d)$ is the projection of the cyclic vector $v \in \Hom(\BC^r,\BC^{d+n})$ onto the quotient $\BC^{d+n} \twoheadrightarrow \BC^d$ (it is easy to see that projection preserves cyclicity). \\
	
	\begin{proposition}
		\label{prop:fiber}
		
		We have the following dimension estimate:
		\begin{equation}
		\label{eqn:fiber estimate}
		\dim \equot_{d,\dots,d+n} \leq \dim \equot_d + \dim \estack_n + r(d+n) 
		\end{equation}
		for any $d \geq 0$ and $n \in \{1,2,3,4\}$. \\
				
	\end{proposition}
	
	\noindent Theorem \ref{thm:quot 1} and Proposition \ref{prop:basili} tell us that $\quot_d$ and $\stack_n$ are equidimensional of dimensions $rd-1+\delta_d^0$ and $-1$, respectively, so Proposition \ref{prop:fiber} yields
	$$
	\dim \quot_{d,...,d+n} \leq 2rd + rn - 2 + \delta_d^0 
	$$
	for all $n \leq 4$. This leads to the following improvement of Corollary \ref{cor:def 1}. \\
	
	\begin{corollary}
		\label{cor:def 2}
		
(a)		For any sequence of defects $d_1,\dots,d_k \geq 0$ and any set partition $\lambda$ whose parts all have multiplicity $\leq 4$, we have
		\begin{equation}
		\label{eqn:def 2}
		\dim \fZ_\lambda^{\edefect d_1,\dots,d_k} \leq \econst + r(c^\efirst_2 + c_2^\elast) + k 
		\end{equation}
		and equality holds if and only if $d_1 = \dots = d_k = 0$. \\
		
		\noindent (b) The number of irreducible components of $\fZ_\lambda^{\edefect 0,\dots,0}$ of top dimension is $2^\#$, where $\#$ is the number of parts of $\lambda$ of multiplicity 4. \\
		
	\end{corollary}

\begin{proof} Statement (a) is an immediate consequence of \eqref{eqn:def 1} and \eqref{eqn:fiber estimate}. For statement (b), we invoke Lemma \ref{lem:locally isomorphic} to reduce the problem to the fact that $\quot_{0,\dots,n}$ has 1 irreducible component of top dimension if $n \in \{1,2,3\}$, and 2 such components if $n=4$. To see this, recall that $\quot_{0,\dots,n}$ is an open subset of a rank $rn$ affine bundle over $\stack_n$. As shown in Proposition \ref{prop:basili}, the latter stack is irreducible if $n\leq 3$, while if $n=4$ its two irreducible components were described in items (1) and (2) of the previous Subsection. The open condition that determines $\quot_{0,\dots,n}$ in the said affine bundle is given by requiring that $v$ of \eqref{eqn:flag of quot} be cyclic, and it is easy to see that this open condition does not ``miss" any of the two components.

\end{proof}
	
	\begin{proof} \emph{of Proposition \ref{prop:fiber}:} Consider the locally closed subset
		\begin{equation}
		\label{eqn:locally closed}
		S_{n,d,r,i} = \Big\{ \dim \ \zeta^{-1} \Big( (X_0,Y_0,v_0) \times (X',Y') \Big) = i \Big\} \subset \quot_d \times \stack_n 
		\end{equation}
		To prove the Proposition, it suffices to show that
		\begin{equation}
		\label{eqn:want codim}
		\text{codim } S_{n,d,r,i} \geq i - r(d+n)
		\end{equation}
		for all $n,d,r,i$. The fiber of $\zeta$ over a point
		\begin{equation}
		\label{eqn:fix a point}
		(X_0,Y_0,v_0) \times (X',Y') \in \quot_d \times \stack_n 
		\end{equation}
		consists of a choice of \\
		
\begin{enumerate}[leftmargin=*]	

\item[(1)] $n \times d$ blocks $x,y$ as in the bottom left corner of \eqref{eqn:blocks}, which satisfy
\begin{equation}
		\label{eqn:equations}
		x \cdot Y_0 + X' \cdot y = y \cdot X_0 + Y' \cdot x
		\end{equation}
		and \\
		
		\item[(2)] $rn$ extra coordinates that upgrade $v_0 \in \Hom(\BC^r, \BC^d)$ to $v \in \Hom(\BC^r, \BC^{d+n})$ \\
		
		\end{enumerate}
		
		\noindent The data in (1)--(2) above provide $2dn+rn$ affine coordinates, and the quotient by the unipotent part of the group $P$ (i.e. the block strictly lower triangular matrices, in the notation of \eqref{eqn:blocks}) subtracts $dn$ dimensions, on account of the $P$-action being free. After accounting for all the generators and relations above, we see that
		\begin{equation}
		\label{eqn:dim zeta}
		\dim \ \zeta^{-1}(\text{point \eqref{eqn:fix a point}}) = n(r+d) - \# \left( \text{of independent equations in \eqref{eqn:equations}} \right)
		\end{equation}
		Therefore, condition \eqref{eqn:want codim} is trivial unless
		\begin{equation}
		\label{eqn:box 1}
		\boxed{n>r}
		\end{equation}
		Given that we fixed $X_0,Y_0,X',Y'$, one may think of \eqref{eqn:equations} as a linear map
		$$
		\Hom(\BC^d,\BC^n) \times \Hom(\BC^d,\BC^n) \xrightarrow{\mu} \Hom(\BC^d,\BC^n)
		$$
		\begin{equation}
		\label{eqn:mu}
		\mu(x,y) = (x \cdot Y_0 - Y' \cdot x) - (y \cdot X_0 - X' \cdot y)
		\end{equation}
		The number of independent equations in \eqref{eqn:equations} is equal to $\dim \text{ Im }\mu$. \\
		
\begin{claim}
\label{claim:independent}

With the notation as above, we have
		\begin{equation}
		\label{eqn:dim mu 1}
		\dim \emph{Im } \mu \geq n \cdot \dim (\emph{Im }X_0 + \emph{Im }Y_0)
		\end{equation}
		
\end{claim}

\medskip

\begin{proof} Let us consider the flag of quotients $\BC^n \stackrel{\pi_n}\twoheadrightarrow \BC^{n-1} \stackrel{\pi_{n-1}}\twoheadrightarrow \dots \stackrel{\pi_1}\twoheadrightarrow \BC^1$ which is preserved by the nilpotent maps $X'$ and $Y'$, and consider the commutative diagram
$$
\xymatrix{\Hom(\BC^d,\BC^i) \times \Hom(\BC^d,\BC^i) \ar[r]^-{\mu_i} \ar[d] & \Hom(\BC^d,\BC^i) \ar[d] \\
\Hom(\BC^d,\BC^{i-1}) \times \Hom(\BC^d,\BC^{i-1}) \ar[r]^-{\mu_{i-1}} & \Hom(\BC^d,\BC^{i-1})}
$$
Above, the horizontal arrows are defined by the same formula as \eqref{eqn:mu}, and the vertical arrows are the operators of composition with $\pi_i$. Let $K = \text{Ker }\pi_i \cong \BC$. Since all the maps in the commutative diagram are linear, it suffices to show that
\begin{equation}
\label{eqn:one step ineq}
\dim \Big( \text{Im } \mu_i \cap \text{Hom}(\BC^d,K) \Big) \geq  \dim (\text{Im }X_0 + \text{Im }Y_0)
\end{equation}
Because $X'|_{\BC^i}$ and $Y'|_{\BC^i}$ annihilate $K$, then for any
$$
(x,y) \in \Hom(\BC^d,K) \times \Hom(\BC^d,K) \subset \Hom(\BC^d,\BC^i) \times \Hom(\BC^d,\BC^i)
$$
we have $\mu_i(x,y) = x Y_0 - y X_0$. Then we fix a basis of $\text{Im }X_0 + \text{Im }Y_0$ consisting of linearly independent vectors in $\text{Im }X_0 \cup \text{Im }Y_0$. We may define $x,y \in \Hom(\BC^d, K) \cong (\BC^d)^\vee$ to take any values on these basis vectors, which precisely implies \eqref{eqn:one step ineq}. 

\end{proof}		
		
\noindent The essential feature which allowed the proof of Claim \ref{claim:independent} to work was that $X'$ and $Y'$ are strictly lower triangular. Since the commuting nilpotent matrices $X_0$ and $Y_0$ also preserve a full flag of subspaces of $\BC^d$, a completely analogous argument yields
		\begin{equation}
		\label{eqn:dim mu 2}
		\dim \text{Im } \mu \geq d \cdot \dim (\text{Im }X' + \text{Im } Y')
		\end{equation}
		so we obtain the estimate
		\begin{multline*}
		\# \left( \text{of independent equations in \eqref{eqn:equations}} \right) \geq \\ \geq \max( n \cdot \dim (\text{Im }X_0 + \text{Im }Y_0), d \cdot \dim (\text{Im }X' + \text{Im } Y') )
		\end{multline*}

\medskip
		
		\begin{claim}
			\label{claim:estimate 1}
			
For any $4 \geq n > r$, the locus of points \eqref{eqn:fix a point} such that
$$
\max( n \cdot \dim (\emph{Im }X_0 + \emph{Im }Y_0), d \cdot \dim (\emph{Im }X' + \emph{Im } Y') ) = j \in \BN
$$
has codimension $\geq d(n-r)-j$ in $\equot_d \times \estack_n$, with the following exception. There is an irreducible component $G$ of the locus above for
$$
n=4, \ d=2, \ r=1, \ j = 4
$$
which has codimension $= 2(4-1)-5$ instead of $\geq 2(4-1)-4$. However, outside of codimension 1, points of $G$ have the property that $\dim \emph{Im }\mu \geq 5$ instead of $\geq 4$. \\
			
		\end{claim}
		
\noindent It is straightforward to see that formula \eqref{eqn:dim zeta} and Claim \ref{claim:estimate 1} imply the required inequality \eqref{eqn:want codim} (even the presence of $G$ does not violate this fact, because even though $G$ has dimension one larger than expected, outside of codimension 1 points of $G$ have the property that $\dim \text{Im } \mu$ is one larger than expected). 	\\

\begin{proof} \emph{of Claim \ref{claim:estimate 1}:} We may assume $d>0$, since the $d=0$ case is trivial. Because the matrices $X_0$ and $Y_0$ have a cyclic vector $v_0 \in \Hom(\BC^r, \BC^d)$, then
		$$
		\dim (\text{Im }X_0 + \text{Im }Y_0) \geq d-r
		$$
		Indeed, the existence of a proper subspace of $\BC^d$ that contains $\text{Im }X_0$, $\text{Im }Y_0$ and the $r$ columns of $v_0$ contradicts the cyclicity of $v_0$. The fact that $n(d-r) \geq d(n-r)$ establishes Claim \ref{claim:estimate 1} when $n\leq d$. Therefore, we are left to deal with the case
		\begin{equation}
		\label{eqn:box 2}
		\boxed{n>d}
		\end{equation}
		Taken together with the inequality \eqref{eqn:box 1}, as well as the fact that we only consider $n\in \{1,2,3,4\}$, this means that we are left with finitely many cases. \\
		
		\begin{claim}
			\label{claim:estimate 2}
			
			The locally closed substack $L_{n,\lambda} \subset \estack_n$ consisting of $(X',Y')$ with the property that $\dim (\emph{Im } X' + \emph{Im } Y') = \lambda$ has codimension
			$$
			\text{codim } L_{2,\lambda} = \begin{cases} 2 & \text{if } \lambda = 0 \\ 0 & \text{if } \lambda = 1 \end{cases}, \text{codim } L_{3,\lambda} = \begin{cases} 5 & \text{if } \lambda = 0 \\ 1 & \text{if } \lambda = 1  \\ 0 & \text{if } \lambda = 2 \end{cases}, \text{codim } L_{4,\lambda} = \begin{cases} 9 & \text{if } \lambda = 0 \\ 3 & \text{if } \lambda = 1  \\ 1 & \text{if } \lambda = 2  \\ 0 & \text{if } \lambda = 3 \end{cases}
			$$
			
		\end{claim} 
		
		\medskip
		
		\begin{claim}
			\label{claim:estimate 3}
			
			The locally closed subscheme $M_{d,\mu} \subset \equot_d$ consisting of $(X_0,Y_0,v_0)$ such that $\dim ( \emph{Im } X_0 + \emph{Im }Y_0 ) = \mu$ is empty unless $\mu \in \{d-r,...,d-1\}$. Moreover,
			$$
			\text{codim } M_{1,\mu} \geq \begin{cases} 0 & \text{if } \mu = 0 \end{cases}, \text{codim } M_{2,\mu} \geq \begin{cases} 3 & \text{if } \mu = 0 \\ 0 & \text{if } \mu = 1\end{cases}, \text{codim } M_{3,\mu} \geq \begin{cases} 8 & \text{if } \mu = 0 \\ 2 & \text{if } \mu = 1 \\ 0 & \text{if } \mu = 2 \end{cases}
			$$
			
			\medskip
			
		\end{claim} 
		
		\noindent The proof of Claims \ref{claim:estimate 2} and \ref{claim:estimate 3} are straightforward exercises, which we leave to the interested reader. For example, the $\mu = 0$ case of Claim \ref{claim:estimate 3} holds because the subscheme $\{X_0 = Y_0 = 0\} \subset \quot_d$ has codimension $d^2-1$. Indeed, this subscheme is nothing but the Grassmannian of $d$ dimensional quotients of $r$ dimensional space, so it has dimension $d(r-d)$, whereas $\quot_d$ has dimension $rd-1$. Note that Claim \ref{claim:estimate 1} reduces to the inequality
		\begin{equation}
		\label{eqn:seal}
		\text{codim } L_{n,\lambda} + \text{codim } M_{d,\mu} \geq d(n-r)- \text{max}(d \lambda, n\mu)  
		\end{equation}
for all $0 \leq \lambda < n$, $\text{max}(0,d-r) \leq \mu < d$ and $1 \leq d,r < n \leq 4$. We leave it to the interested reader to show that \eqref{eqn:seal} follows from Claims \ref{claim:estimate 2} and \ref{claim:estimate 3}, with one exception: the inequality fails for $n = 4$, $d = 2$, $r = 1$, $\lambda = 2$, $\mu = 1$, and in fact, in this case the difference between LHS and RHS is equal to $-1$. The failure of the inequality is due to the locally closed subset $G \subset  \quot_2 \times \stack_4$ consisting of points $(X_0,Y_0,v_0) \times (X',Y')$ such that $X_0, Y_0$ are generic, but $\text{Im } X' + \text{Im }Y'$ has dimension 2. After a change of basis, points of $G$ take the form
		$$
		X = \left( \begin{array}{cc|cccc} 
		0 & 0 &  &  &  &  \\
		s & 0 &  &  &  &  \\ \hline
		&  & 0 & 0 & 0 & 0 \\
		&  & x_1 & 0 & 0 & 0\\
		&  & x_2 & 0 & 0 & 0\\ 
		&  & x_3 & x_4 & x_5 & 0\\\end{array} \right) \qquad Y = \left( \begin{array}{cc|cccc} 
		0 & 0 &  &  &  &  \\
		t & 0 &  &  &  &  \\ \hline
		&  & 0 & 0 & 0 & 0 \\
		&  & y_1 & 0 & 0 & 0\\
		&  & y_2 & 0 & 0 & 0\\ 
		&  & y_3 & y_4 &y_5 & 0\\\end{array} \right)
		$$
		with $(s,t) \neq (0,0)$ and $y_1x_2 = x_1y_2$, $y_1x_4 + y_2x_5 = x_1y_4 + x_2y_5$. It is easy to see that $G$ has a single irreducible component of top dimension (compare it with $Z_1$ of Subsection \ref{sub:stack 4}). Moreover, $G$ contains the point $x_1 = \dots = x_5 = y_1 = \dots = y_5 = 1$, for which it is elementary to show that $\dim \text{Im } \mu \geq 5$. The fact that the latter inequality holds on $G$ outside of codimension 1 is a consequence of lower semicontinuity of rank. This establishes the final sentence of Claim \ref{claim:estimate 1}.
		
	\end{proof} \end{proof}
	
	\subsection{} 
	\label{sub:smooth locus}
	
	Since the locally closed subsets $\fZ_\lambda^{\defect d_1,...,d_k}$ stratify the scheme $\fZ_\lambda$, Corollary \ref{cor:def 2} implies that 
	\begin{equation}
	\label{eqn:def 3}
	\dim \fZ_\lambda = \const + r(c_2^\first + c_2^\last) + k
	\end{equation}
whenever $|\lambda| \leq 4$. Keeping in mind Definition \ref{def:exp dim}, we conclude the following. \\
	
	\begin{corollary}
	\label{cor:exp dim}
	
	If $|\lambda| \leq 4$, then $\fZ_\lambda$ has expected dimension. The number of its irreducible components of top dimension is
	$$
	\begin{cases} 1 &\text{if } \lambda \neq (x,x,x,x) \\ 2 &\text{if } \lambda = (x,x,x,x) \end{cases}
	$$ 
		
	\end{corollary}
	
\medskip	
	
\noindent Let us now consider $\fZ_{(x,y)}$, which parametrizes flags $(\CF_0 \subset_x \CF_1 \subset_y \CF_2)$. It was shown in \cite[Claim 4.22]{W univ} that the singular locus of $\fZ_{(x,y)}$ is given by
	\begin{equation}
	\label{eqn:triple}
	\fZ_{(x,x)}^{\text{split}} = \Big\{ (\CF_0 \subset_x \CF_1 \subset_x \CF_2) \text{ such that } \CF_2/\CF_0 \text{ is split} \Big\}
	\end{equation} 
	where split means that the length 2 sheaf $\CF_2/\CF_0$ supported at $x$ is a direct sum of length 1 skyscraper sheaves $\BC_x$. We will now estimate the dimension of \eqref{eqn:triple}. \\
	
\begin{proposition}	
\label{prop:dim split}	
	
$\fZ_{(x,x)}^{\emph{split}}$ has dimension 
$$
\econst + r(c^\efirst_2 + c_2^\elast) - 1
$$
and a single irreducible component of top dimension. \\
	
\end{proposition}

\begin{proof} Consider the stratification into locally closed subsets
	$$
	\fZ_{(x,x)}^{\text{split}} = \bigsqcup_{d=0}^{\infty} \fZ_{(x,x)}^{\text{split}, \text{def } d}
	$$
	where the $d$-th stratum parameterizes those flags as in \eqref{eqn:triple} where $\CF_2$ has defect $d$ at $x$. It suffices to show that
\begin{equation}
\label{eqn:want dim split}
\dim \fZ_{(x,x)}^{\text{split}, \text{def } d} \leq \const + r(c^\first_2 + c_2^\last) - 1
\end{equation}
with equality if and only if $d=0$.  The splitting condition in \eqref{eqn:triple} may be realized as follows. Consider the codimension 2 substack
		\begin{equation}
	\label{eqn:sublocus of stack}
\stack_2' = \Big \{X' = Y' = 0 \Big \} \subset \stack_2
	\end{equation}
	and the following fiber square where the map on the right is \eqref{eqn:fiber map}:
$$
\xymatrix{\quot_{d,d+1,d+2}' \ar[d]_{\zeta'} \ar@{^{(}->}[r] & \quot_{d,d+1,d+2} \ar[d]^\zeta \\
		\quot_d \times \stack_2' \ar@{^{(}->}[r] & \quot_d \times \stack_2}
$$
It is easy to see that 
$$
\fZ_{(x,x)}^{\text{split}, \text{def } d} \quad \text{is to} \quad \quot_{d,d+1,d+2}' \qquad \text{as} \qquad \fZ_{(x,x)}^{\text{def } d} \quad \text{is to} \quad \quot_{d,d+1,d+2}
$$
(as per Lemma \ref{lem:locally isomorphic}). With this in mind, \eqref{eqn:want dim split} is a consequence of the following analogue of Proposition \ref{prop:fiber}:
\begin{align*}
\dim \quot_{d,d+1,d+2}' &\leq \dim \quot_d + \dim \stack_2' + r(d+2) \\ &= \dim \quot_d + \dim \stack_2 - 2 + r(d+2)
\end{align*}
To prove the inequality above, one follows the proof of Proposition \ref{prop:fiber} closely, but with $X'  = Y' = 0$ throughout. In particular, \eqref{eqn:box 1} and \eqref{eqn:box 2} reduce the problem to the case $r = d = 1$, in which case Claim \ref{claim:estimate 1} is trivial. As for the statement of a single irreducible component of top dimension, this is because one only has equality in \eqref{eqn:want dim split} for $d=0$, and one argues as in the proof of Corollary \ref{cor:def 2}.

\end{proof}

\section{Geometric properties}
\label{sec:geom}

\medskip

\subsection{} We will now use the dimension estimates from the previous Section to obtain various geometric properties of the varieties $\fZ_{\lambda}$, which we have invoked in Section \ref{sec:mod}. \\
	
	\begin{proposition}
		\label{prop:flat morphism}
		
If $\fZ_\lambda$ is Cohen-Macaulay of expected dimension (see Definition \ref{def:exp dim}), then the map $p_S^i : \fZ_\lambda \rightarrow S$ which remembers the $i$-th support point is flat $\forall i$. \\
		
	\end{proposition}
	
	\begin{proof} Since $S$ is smooth and $\fZ_\lambda$ is Cohen-Macaulay, the Miracle Flatness Theorem asserts that all we need to show is that the fibers of the morphism $p_S^i : \fZ_\lambda \rightarrow S$ all have dimension equal to $\dim \fZ_\lambda - 2$. By upper semicontinuity, it suffices to show that all the fibers have dimension $ \leq \dim \fZ_\lambda - 2$. This is proved in the same way as Corollary \ref{cor:def 1}, since fixing one of the support points of the flag merely replaces the number $k$ in the right-hand side of \eqref{eqn:def 1} by $k-2$.  
		
	\end{proof}

\begin{proof} \emph{of Proposition \ref{prop:x}:} See \cite[Proposition 2.10]{Univ} for smoothness, and Corollary \ref{cor:exp dim} for the dimension and irreducibility statements. 

\end{proof}

\begin{proof} \emph{of Proposition \ref{prop:xx}:} See \cite[Proposition 4.21]{W univ} for smoothness, and Corollary \ref{cor:exp dim} for the dimension and irreducibility statements. Let us now show that the fiber square \eqref{eqn:square xx} is derived with excess. By Proposition \ref{prop:tower}, we may subdivide this fiber square into
\begin{equation}
\label{eqn:break up two squares}
		\xymatrix{
			\fZ_{(x,x)} \ar[d]_-{\pi_+} \ar@{^{(}->}[r]^-{\iota'} & \BP_{\fZ_{(x)}}(\Gamma^{*}(\CV_1)) \ar[d] \ar@{->>}[r]^-{\rho'} & \fZ_{(x)} \ar[d] \\
			\fZ_{(x)} \ar@{^{(}->}[r]^-{\iota} & \BP_{\CM \times S}(\CV_1) \ar@{->>}[r]^-{\rho} & \CM \times S}
\end{equation}
		where $\Gamma : \fZ_{(x)} \rightarrow \fZ_{(x)} \times S$ is the graph of the map $p_S : \fZ_{(x)} \rightarrow S$. The rightmost fiber square in \eqref{eqn:break up two squares} is already derived, because the maps $\rho$ and $\rho'$ are smooth. However, the leftmost fiber square in \eqref{eqn:break up two squares} is not derived, because
\begin{equation}
\label{eqn:actually equality}
\text{codim } \iota - \text{codim } \iota' = \dim \fZ_{(x,x)} - \dim \fZ_{(x)} - \dim \fZ_{(x)} + \dim (\CM \times S) = 1 
\end{equation}
(this follows from the dimension statements of Propositions \ref{prop:x} and \ref{prop:xx}, which have already been proved). Recall from Proposition \ref{prop:tower} that $\iota$ is cut out by the regular section
		$$
		\rho^*(\CW_1) \hookrightarrow \rho^*(\CV_1) \twoheadrightarrow \CO(1)
		$$
		hence is a complete intersection of codimension equal to the rank of $\CW_1$. On the other hand, $\iota'$ is cut out by the section
$$
\CW_1|_{\Gamma} \rightarrow \CV_1|_{\Gamma} \twoheadrightarrow \CO(1) = \CL_1
$$
In Proposition \ref{prop:tor}, we showed that the section above factors through
		\begin{equation}
		\label{eqn:excess section}
		\frac {\CW_1|_{\Gamma}}{\CL_2 \otimes p_S^*(\omega_S)} \longrightarrow \CO(1) = \CL_1 
		\end{equation}
where the left-hand side is a locally free sheaf of dimension 1 less than the rank of $\CW_1$. By \eqref{eqn:actually equality}, the section \eqref{eqn:excess section} is therefore regular, and this establishes the claim that the square  \eqref{eqn:square xx} is derived with excess bundle $\CL_2 \otimes \CL_1^{-1} \otimes p_S^*(\omega_S)$. 
		
	\end{proof}
	
\begin{proof} \emph{of Propositions \ref{prop:xy}, \ref{prop:xxx}, \ref{prop:xxy and yxx}:} See Corollary \ref{cor:exp dim} for the dimension and irreducibility statements. The fact that the schemes in question are l.c.i. is proved by showing that the fiber squares in the statements of the Propositions are derived. The latter statement is an immediate consequence of Proposition \ref{prop:tower} together with the fact that the spaces in the four corners of each square satisfy the analogue of \eqref{eqn:eq lci} (itself an immmediate consequence of the already proved fact that the northeast, southeast and southwest corners of these squares are l.c.i. of expected dimension, while the northwest corner has expected dimension, see Corollary \ref{cor:exp dim}). 

\end{proof}

	\subsection{}
	\label{sub:cm}
	
We will now prove that the variety $\fZ_{(x,y,x)}$ is Cohen-Macaulay. \\
	
	\begin{proof} \emph{of Proposition \ref{prop:xyx}:} See Corollary \ref{cor:exp dim} for the dimension and irreducibility statements. As a warm-up to proving the fact that $\fZ_{(x,y,x)}$ is Cohen-Macaulay, let us recall the explicit realization of $\fZ_{(x,y)}$ as a local complete intersection. We do so not only because the computation will serve as valuable illustration, but we will need a specific realization of $\fZ_{(x,y)}$ as the zero locus of a regular section. Therefore, consider the smooth scheme $\fZ_{(x)} \times S$ and take the map
		$$
		\xymatrix{\fZ_{(x,y)} \ar[d]^{\pi_+} \\
			\fZ_{(x)} \times S} \qquad \qquad \xymatrix{(\CF_0 \subset_x \CF_1 \subset_y \CF_2) \ar@{.>}[d] \\
			(\CF_0 \subset_x \CF_1, y)}
		$$
		A particular case of Proposition \ref{prop:tower} says that this map can be realized as
		\begin{equation}
		\label{eqn:projectivization one}
		\xymatrix{
			\fZ_{(x,y)} \ar[rd]_{\pi_+} \ar@{^{(}->}[r]^-\iota & \BP := \BP_{\fZ_{(x)} \times S} (\CW_1^\vee \otimes \omega_S) \ar[d] \\
			& \fZ_{(x)} \times S}
		\end{equation}
		where the closed embedding $\iota$ is cut out by the section
		\begin{equation}
		\label{eqn:sigma one}
		\sigma : \CO(-1) \rightarrow \CW_1 \otimes \omega_S^{-1} \rightarrow \CV_1 \otimes \omega_S^{-1}
		\end{equation}
		and $\CO(1)$ denotes the tautological line bundle on the projective bundle in \eqref{eqn:projectivization one}. Above and hereafter, we abuse notation by writing $\CW_1$ for the vector bundle on $\fZ_{(x)} \times S$, as well as for its pull-back to $\BP$. The dimension of $\fZ_{(x)} \times S$ is equal to
		$$
		\const + r(c_2(\CF_0) + c_2(\CF_1)) + 3
		$$
		and therefore the dimension of the projective bundle \eqref{eqn:projectivization one} is equal to
		$$
		\const + r(c_2(\CF_0) + c_2(\CF_1)) + 2 + \text{rank } \CW_1
		$$
		To obtain the closed embedding $\iota$ in \eqref{eqn:projectivization one}, we impose as many equations as the number of coordinates of the section \eqref{eqn:sigma one}, so the expected dimension of $\fZ_{(x,y)}$ is
	\begin{multline*}
		\const + r(c_2(\CF_0) + c_2(\CF_1)) + 2 + \text{rank } \CW_1 - \text{rank } \CV_1 = \\ = \const + r(c_2(\CF_0) + c_2(\CF_2)) + 2
	\end{multline*}
		Since this is equal to the actual dimension of $\fZ_{(x,y)}$ by \eqref{eqn:def 3}, we conclude that the coordinates of the section $\sigma$ form a regular sequence
		\begin{equation}
		\label{eqn:regular}
		\Sigma = \Big\{ \text{regular sequence formed by the coordinates of } \sigma \Big \}
		\end{equation}
In order to study $\Sigma$ in detail, recall that the vector bundle $\CV_1$ is given by \eqref{eqn:we can take}: from now on we will set $n = 0$ in \eqref{eqn:we can take}, in order to keep our formulas legible (otherwise, one would have to often tensor our formulas by line bundles coming from $S$, but this has no substantial effect on our argument). We have a map of locally free sheaves $\phi : \CV_1 \rightarrow \CL_1$ on $\fZ_{(x)}$ that is a particular case of the middle row of diagram \eqref{eqn:big diagram} when $i=1$. Putting all of these constructions together, we may consider the map of line bundles
		\begin{equation}
		\label{eqn:comp one}
		\phi \circ \sigma : \CO(-1) \xrightarrow{\taut} \CW_1 \otimes \omega_S^{-1} \rightarrow \CV_1 \otimes \omega_S^{-1} \xrightarrow{\phi \otimes \text{Id}} \CL_1 \otimes \omega_S^{-1}
		\end{equation}
on $\BP$. Since $\text{Ker }\phi = \CV_0$ is locally free, the map of line bundles $\phi \circ \sigma$ may be thought of as one of the coordinates of the section $\sigma$ in any local trivialization. Therefore, it is one of the entries of the regular sequence $\Sigma$, so let us compute it explicitly on a closed point of the projective bundle \eqref{eqn:projectivization one}. We may work locally, so we assume that $x$ and $y$ are points of $\BA^2 = \text{Spec } \BC[s_1,s_2]$ (alternatively, we need to take generators of the maximal ideals at the closed points $x$ and $y$, but the explanation is analogous) given by coordinates $(x_1,x_2)$ and $(y_1,y_2)$, respectively. As we have seen in the proof of Proposition \ref{prop:tor}, the map $\taut$ in \eqref{eqn:comp one} is identified with
		$$
		\Gamma(S,\CF_2/\CF_1) \otimes \Tor_2(\CO_{\Gamma^y}, \CO_{\Gamma^y}) \rightarrow \Tor_0(\CW_1, \CO_{\Gamma^y})
		$$
		where $\Gamma^y : \BP \hookrightarrow \BP \times S$ is the graph of the map that remembers the point $y$. The Tor groups above may be computed using the following length 2 resolution of $\CO_{\Gamma^y}$:
		$$
		\left[ \CO \xrightarrow{s_2 - y_2, y_1 - s_1} \CO \oplus \CO \xrightarrow{s_1 - y_1, s_2 - y_2} \CO \right] \qis \CO_{\Gamma^y}
		$$
		Therefore, the map denoted $\taut$ in \eqref{eqn:comp one} takes a generator $v$ of the one-dimensional vector space $\Gamma(S,\CF_2/\CF_1)$ to the section $(s_1-y_1) \tau_1 + (s_2 - y_2) \tau_2$, where $\tau_1, \tau_2 \in \Gamma(S,\CF_1)$ have the property that $\tau_1 = (s_2-y_2) f$ and $\tau_2 = (y_1-s_1)f$ for some preimage $f \in \Gamma(S,\CF_2)$ of $v\in \Gamma(S,\CF_2/\CF_1)$ (global sections exist due to our assumption that $n = 0$ in \eqref{eqn:we can take}, otherwise we would need to tensor $\CF_1$ and $\CF_2$ by a very ample line bundle). Therefore, the entire composition \eqref{eqn:comp one} takes the generator $v$ to
		$$
		\text{the image of } (s_1-y_1) \tau_1 + (s_2 - y_2) \tau_2 \text{ in } \Gamma(S,\CF_1/\CF_0)
		$$
		In coordinates, if we identify the rank 2 bundle $\Gamma(S,\CF_2/\CF_0)$ with $\BC^2$, then the $\CO_S = \BC[s_1,s_2]$-module structure on this 2-dimensional space is given by matrices
		\begin{equation}
		\label{eqn:matrices one}
		s_1 = \begin{pmatrix} y_1 & 0 \\ a_1 & x_1 \end{pmatrix} \qquad s_2 = \begin{pmatrix} y_2 & 0 \\ a_2 & x_2 \end{pmatrix}
		\end{equation}
		and so the assignment \eqref{eqn:comp one} takes
		$$
		\begin{pmatrix} 1 \\ 0 \end{pmatrix} \leadsto \begin{pmatrix} 0 \\ (x_1-y_1)a_2 - (x_2-y_2)a_1   \end{pmatrix}
		$$
		Therefore, we conclude that
		\begin{equation}
		\label{eqn:equation}
		(x_1-y_1)a_2 - (x_2-y_2)a_1  
		\end{equation}
		is one of the elements of the regular sequence \eqref{eqn:regular}. Since over local Noetherian rings, one may permute the order of elements of a regular sequence, we will assume \eqref{eqn:equation} to be the last element of the regular sequence. \\
		
		\begin{claim}
			\label{claim:unexpected}
			
			Before imposing the equation \eqref{eqn:equation}, the other elements of the regular sequence cut out a regular local ring. In more mathematical terms, the local rings of $\fZ_{(x,y)}$ are quotients of regular local rings by the single equation \eqref{eqn:equation}. \\
						
		\end{claim}
		
		\noindent Indeed, the Claim follows from the fact (proved in \cite{W univ}) that the tangent spaces to $\fZ_{(x,y)}$ have expected dimension, except at a closed point such that $x = y$ and $\CF_2/\CF_0$ is split, where the dimension of the tangent space jumps by 1. In the local coordinates \eqref{eqn:equation}, this corresponds to $a_1 = a_2 = x_1-y_1 = x_2-y_2 = 0$: therefore, equation \eqref{eqn:equation} fails to cut down the dimension of the tangent spaces on the split locus. Since $\fZ_{(x)}$ is smooth, this means that all other elements in the regular sequence do cut out regular subschemes, and regularity is broken precisely by \eqref{eqn:equation}. \\
		
		\noindent Armed with the discussion above, we are ready to analyze the scheme $\fZ_{(x,y,x)}$:
		$$
		\xymatrix{\fZ_{(x,y,x)} \ar[d]^{\pi_+} \\
			\fZ_{(x,y)}} \qquad \qquad \xymatrix{(\CF_0 \subset_x \CF_1 \subset_y \CF_2 \subset_x \CF_3) \ar@{.>}[d] \\
			(\CF_0 \subset_x \CF_1 \subset_y \CF_2)}
		$$
		A particular case of Proposition \ref{prop:tower} says that this map can be realized as
		\begin{equation}
		\label{eqn:projectivization two}
		\xymatrix{
			\fZ_{(x,y,x)} \ar[rd]_{\pi_+} \ar@{^{(}->}[r]^-\iota & \BP':= \BP_{\fZ_{(x,y)}} (\Gamma^{x*}(\CW_2^\vee \otimes \omega_S)) \ar[d] \\
			& \fZ_{(x,y)}}
		\end{equation}
		where $\Gamma^x : \fZ_{(x,y)} \rightarrow \fZ_{(x,y)} \times S$ is the graph of the map that remembers the point $x$, and the closed embedding $\iota$ is cut out by the section
		\begin{equation}
		\label{eqn:sigma two}
		\sigma : \CO(-1) \rightarrow \Gamma^{x*}(\CW_2 \otimes \omega_S^{-1}) \rightarrow \Gamma^{x*}(\CV_2 \otimes \omega_S^{-1})
		\end{equation}
(we abuse notation by identifying the vector bundles on $\fZ_{(x,y)}$ above with their pull-backs to $\BP'$) where $\CO(1)$ denotes the tautological line bundle on $\BP'$. By \eqref{eqn:def 3}, the dimension of $\fZ_{(x,y)}$ is equal to
		$$
		\const + r(c_2(\CF_0) + c_2(\CF_2)) + 2
		$$
		and therefore the dimension of the projective bundle \eqref{eqn:projectivization two} is equal to
		$$
		\const + r(c_2(\CF_0) + c_2(\CF_2)) + 1 + \text{rank } \CW_2
		$$
		To obtain the closed embedding $\iota$ in \eqref{eqn:projectivization two}, we impose as many equations as the number of coordinates of \eqref{eqn:sigma two}, so the expected dimension of $\fZ_{(x,y,x)}$ is
		$$
		\const + r(c_2(\CF_0) + c_2(\CF_2)) + 1 + \text{rank } \CW_2 - \text{rank } \CV_2 = \const + r(c_2(\CF_0) + c_2(\CF_3)) + 1
		$$
		However, according to \eqref{eqn:def 3}, the actual dimension of $\fZ_{(x,y,x)}$ is 1 bigger than the above expected dimension. Therefore, $\fZ_{(x,y,x)}$ is an almost complete intersection, i.e. cut out by one more equation than its dimension. We have the following criterion for when such schemes are Cohen-Macaulay. \\
		
		\begin{claim}
			\label{claim:cm}
			
			Consider a Cohen-Macaulay local ring $R$ and a collection of elements $f_0,\dots,f_n \in R$ such that the quotient ring $R/(f_0,\dots,f_n)$ has codimension $n$ in $R$. If $R/(f_0,\dots,f_i)$ is Cohen-Macaulay of codimension $i$ in $R$ for some $i \geq 0$, then $R/(f_0,\dots,f_n)$ is Cohen-Macaulay. \\
			
		\end{claim}
		
		\noindent The claim is proved by induction and the well-known fact that if an element $f$ in a Cohen-Macaulay local ring $R$ has the property that $\dim R/(f) = \dim R - 1$, then $f$ is a non-zero divisor and $R/(f)$ is Cohen-Macaulay. We will apply Claim \ref{claim:cm} to our situation by constructing a map of locally free sheaves on $\BP'$:
		\begin{equation}
		\label{eqn:cok}
		\phi : \Gamma^{x*}(\CV_2 \otimes \omega_S^{-1}) \twoheadrightarrow \CN
		\end{equation}
		such that $\CN$ has rank 2 and the closed subscheme
		\begin{equation}
		\label{eqn:subscheme}
\Big\{\phi \circ \sigma = 0 \Big\} \hookrightarrow \BP'
		\end{equation}
		is Cohen-Macaulay of codimension 1. Once we do so, the fact that $\fZ_{(x,y,x)}$ is Cohen-Macaulay of the expected dimension follows from Claim \ref{claim:cm} with $i = 1$. \\
		
		\noindent To construct the map \eqref{eqn:cok}, we go back to diagram \eqref{eqn:big diagram} for $i=2$. Its middle row consists of locally free sheaves, so we may restrict it to $\Gamma^x$:
		$$
		0 \rightarrow \Gamma^{x*}(\CV_1) \rightarrow \Gamma^{x*}(\CV_2) \rightarrow \CL_2 \rightarrow 0
		$$
(the right-most term is $\CL_2$ because we assumed $n=0$ in \eqref{eqn:we can take}, otherwise we could have had to twist by a line bundle). The push-out of the short exact sequence above with respect to the tautological map $\Gamma^{x*}(\CV_1) \twoheadrightarrow \CL_1$ that we have on $\fZ_{(x,y)}$ yields
		$$
		\xymatrix{
			0 \ar[r] & \Gamma^{x*}(\CV_1) \ar@{->>}[d] \ar[r] & \Gamma^{x*}(\CV_2) \ar@{->>}[d]^{\phi} \ar[r] & \CL_2 \ar@{=}[d] \ar[r] & 0 \\
			0 \ar[r] & \CL_1 \ar[r] & \CN \otimes \Gamma^{x*}(\omega_S) \ar[r] & \CL_2 \ar[r] & 0}
		$$
		where $\CN$ is defined by the short exact sequence on the bottom row. This defines the map \eqref{eqn:cok}. We will now write out the entries of the map $\phi$ explicitly in coordinates, just as we did in the discussion immediately preceding Claim \ref{claim:unexpected}. We still work locally, so assume that $x$ and $y$ are points of $\BA^2 = \text{Spec } \BC[s_1,s_2]$ given by coordinates $(x_1,x_2)$ and $(y_1,y_2)$, respectively. Then we identify the rank 3 bundle $\Gamma(S,\CF_3/\CF_0)$ with $\BC^3$, and then the $\CO_S = \BC[s_1,s_2]$-module structure on this 3-dimensional space is given by matrices
		\begin{equation}
		\label{eqn:matrices two}
		s_1 = \begin{pmatrix} x_1 & 0 & 0 \\ b_1 & y_1 & 0 \\ c_1 & a_1 & x_1 \end{pmatrix} \qquad s_2 = \begin{pmatrix} x_2 & 0 & 0 \\ b_2 & y_2 & 0 \\ c_2 & a_2 & x_2 \end{pmatrix}
		\end{equation}
		The matrices \eqref{eqn:matrices one} are precisely the bottom right $2 \times 2$ blocks of \eqref{eqn:matrices two}. Therefore, the composition $\phi \circ \sigma : \CO(-1) \rightarrow \CN$ is given explicitly in coordinates by
		$$
		\begin{pmatrix} 1 \\ 0 \\ 0 \end{pmatrix} \leadsto \begin{pmatrix} 0 \\ b_1 (x_2-y_2) - b_2(x_1-y_1) \\ b_1 a_2 - a_1 b_2\end{pmatrix}
		$$
		Therefore, we conclude that the subscheme \eqref{eqn:subscheme} is obtained by imposing the equations $b_1 (x_2-y_2) = b_2(x_1-y_1)$ and $b_1 a_2 = a_1 b_2$ on the local rings of a projective bundle over the scheme $\fZ_{(x,y)}$. As we have seen in Claim \ref{claim:unexpected}, the local rings of the scheme $\fZ_{(x,y)}$ were obtained by imposing the equation $a_1(x_2-y_2) = a_2(x_1-y_1)$ of \eqref{eqn:equation} in a regular local ring. Therefore, we are in a particular case of the following general situation. \\
		
		\begin{claim}
			\label{claim:cm 2}
			
			Suppose $R$ is a regular ring with given elements $a_1, b_1, a_2, b_2, d_1, d_2$, and let $I = (a_1d_2 - a_2d_1, b_1 d_2 - b_2 d_1, a_1b_2 - b_1 a_2)$. Then $R/I$ is Cohen-Macaulay of codimension 2 if \\
			
			\begin{enumerate}  
				
				\item $a_1b_2 - b_1 a_2$ is not a zero-divisor in $R/(d_1,d_2)$ \\
				
				\item $R/(a_1d_2 - a_2d_1, b_1 d_2 - b_2 d_1)$ is codimension 2 in $R$ \\
				
			\end{enumerate}
			
		\end{claim}
		
		\noindent Let us show that the hypotheses of Claim \ref{claim:cm 2} hold in the local rings of \eqref{eqn:subscheme} with $a_1,b_1,a_2,b_2$ as in \eqref{eqn:matrices two} and $d_1 = x_1 - y_1$, $d_2 = x_2 - y_2$. For item (1), we observe that imposing $d_1 = d_2 = 0$ has the effect of setting the support points $x, y$ equal to each other, which cuts out the smooth subscheme $\fZ_{(x,x)} \hookrightarrow \fZ_{(x,y)}$. Since $\BP'|_{\fZ_{(x,x)}}$ is smooth, it does not have any non-trivial zero-divisors, so must show that
\begin{equation}
\label{eqn:aforementioned}
a_1b_2 - b_1a_2 \neq 0
\end{equation}
in the local rings of $\BP'|_{\fZ_{(x,x)}}$. If \eqref{eqn:aforementioned} failed to hold, then
$$
\#\Big\{\text{equations cutting out }\iota|_{\fZ_{(x,x)}}\Big\} \leq \#\Big\{\text{equations cutting out }\iota \Big\} - 1
$$		
and so
\begin{multline*}
\dim \fZ_{(x,x,x)} = \dim \BP'|_{\fZ_{(x,x)}} - \#\Big\{\text{equations cutting out }\iota|_{\fZ_{(x,x)}}\Big\} \geq \\ \geq \dim \BP' - \#\Big\{\text{equations cutting out }\iota \Big\} = \dim \fZ_{(x,y,x)}
\end{multline*}
which would contradict \eqref{eqn:def 3}. As for item (2), we must show that $b_1d_2 - b_2d_1$ is not a zero-divisor in the local rings of $\BP'$. This is the case because $b_1,b_2$ are linear coordinates in a projective bundle over $\fZ_{(x,y)}$, so the only way $b_1d_2 - b_2d_1$ could be a zero divisor is if $d_1$ and $d_2$ were zero-divisors. If this were the case, then $\fZ_{(x,x)}$ would have the same dimension as $\fZ_{(x,y)}$, which would contradict \eqref{eqn:def 3}. \\
		
		\begin{proof} \emph{of Claim \ref{claim:cm 2}:} By the Auslander-Buchsbaum formula, it is enough to show that the ring $R/I$ has projective dimension 2 as an $R$--module. In fact, we claim that a projective resolution is given by
			$$
			0 \longrightarrow R^2 \xrightarrow{\begin{pmatrix} b_1 & b_2 \\ -a_1 & -a_2 \\ d_1 & d_2 \end{pmatrix}} R^3 \xrightarrow{\begin{pmatrix} d_1a_2 - a_1d_2 & d_1b_2 - b_1d_2 & a_1b_2-a_2b_1 \end{pmatrix}} R \longrightarrow R/I \longrightarrow 0
			$$
Exactness at $R$ and $R/I$ is obvious. As for exactness at $R^2$, if the first map failed to be injective then $a_1b_2 - a_2b_1$, $d_1b_2-b_1d_2$ and $d_1a_2-a_1d_2$ would all be zero-divisors (hence zero in the regular ring $R$) which is not allowed by property (1). It remains to prove exactness at $R^3$. Assume that we have $(x,y,z) \in R^3$ such that
			$$
			(d_1a_2 - a_1d_2)x + (d_1b_2 - b_1d_2) y + (a_1b_2-a_2b_1) z = 0
			$$
			Clearly, $(a_1b_2-a_2b_1)z \in (d_1,d_2)$, so property (1) implies $z = d_1 m + d_2 n$. Therefore, we may rewrite the relation above as
			$$
			(d_1a_2 - a_1d_2)(b_1m + b_2 n + x') + (d_1b_2 - b_1d_2)(-a_1m - a_2 n + y') + (a_1b_2-a_2b_1)(d_1 m + d_2 n) = 0
			$$
			where $x' = x - b_1 m - b_2 n$ and $y' = y + a_1 m + a_2 n$. The relation above reduces to
			$$
			(d_1a_2 - a_1d_2)x' + (d_1b_2 - b_1d_2)y' = 0 
			$$
			and then item (2) implies that we have $x' = (d_1b_2 - b_1d_2)u$ and $y' = -(d_1a_2 - a_1d_2)u$ for some $u$. Therefore, we have
			$$
			\begin{pmatrix}
			x \\ y \\ z 
			\end{pmatrix} = \begin{pmatrix}
			x' + b_1m + b_2 n\\ y' - a_1 m - a_2n  \\ d_1 m + d_2 n 
			\end{pmatrix} = \begin{pmatrix}
			b_1(m-d_2u) + b_2(n+d_1u) \\ -a_1(m-d_2u) -a_2(n+d_1u) \\ d_1(m-d_2u) + d_2(n+d_1u)
			\end{pmatrix}
			$$
			which shows that the triple $(x,y,z)$ came from the image of the $3 \times 2$ matrix. 
			
		\end{proof}
		
	\end{proof}
		
\begin{proof} \emph{of Proposition \ref{prop:xxxx} and \ref{prop:xxyx and xyxx}:} See the Proof of Propositions \ref{prop:xy}, \ref{prop:xxx}, \ref{prop:xxy and yxx}.

\end{proof}		
		
\subsection{}

We will now prove the normality of some of the schemes $\fZ_\lambda$ for $|\lambda| \leq 4$. \\
	
	\begin{proof} \emph{of Proposition \ref{prop:normal}:} Since all the schemes in question are Cohen-Macaulay, it suffices to show that they are singular in codimension $\geq 2$. Let us note that the required statement for $\lambda = (x,y)$ follows from Proposition \ref{prop:dim split}, which shows that the singular locus 
$$
\fZ_{(x,x)}^{\text{split}} \subset \fZ_{(x,y)}
$$
has codimension 3. As for the other $\lambda$'s in \eqref{eqn:lambda}, we will only prove the case $\lambda = (x,x,y,x)$, as the analysis in the other cases is analogous and no more difficult. Below, we will list certain locally closed subsets of
		\begin{equation}
		\label{eqn:flag}
		\fZ_{(x,x,y,x)} = \Big\{ (\CF_0 \subset_x \CF_1 \subset_x \CF_2 \subset_y \CF_3 \subset_x \CF_4) \Big\}
		\end{equation}
		which form a stratification of the variety $\fZ_{(x,x,y,x)}$. To show normality of this variety, one may ignore all strata of codimension 2 and higher. \\
		
		\begin{enumerate}[leftmargin=*]
			
			\item when $x = y$ and $\CF_4$ has non-zero defect $>0$ at $x$, Corollary \ref{cor:def 2} shows that the corresponding locally closed subset of \eqref{eqn:flag} has codimension $\geq 2$, hence can be ignored; \\
			
			\item when $x = y$ and $\CF_4$ is locally free near $x$, the scheme $\fZ_{(x,x,y,x)}$ is locally isomorphic to $\CM \times \Sigma$, where
			\begin{multline*} 
			\Sigma = \Big \{(\CO^{\oplus r} \twoheadrightarrow \CQ_4 \twoheadrightarrow \CQ_3 \twoheadrightarrow \CQ_2 \twoheadrightarrow \CQ_1), \text{ supp } \CQ_4 = \{x,x,y,x\}, \\
			\text{supp } \CQ_3 = \{x,y,x\}, \text{supp } \CQ_2 = \{y,x\}, \text{supp } \CQ_1 = \{x\} \Big \}
			\end{multline*}
			Compare $\Sigma$ with the scheme $\quot_{0,1,2,3,4}$ of Definition \ref{def:quot 2}. We must prove that $\Sigma$ is normal. Since we may work locally, we assume that the base surface is $S = \BA^2$, and we will normalize $x = (0,0)$ and $y = (a,b)$. By analogy with Subsection \ref{sub:adhm}, the scheme $\Sigma$ parameterizes triples $(X,Y,v)$ where
			$$
			X = \begin{pmatrix}
			0 & 0 & 0 & 0 \\
			x_{21} & a & 0 & 0 \\
			x_{31} & x_{32} & 0 & 0 \\
			x_{41} & x_{42} & x_{43} & 0 
			\end{pmatrix}, \qquad Y = \begin{pmatrix}
			0 & 0 & 0 & 0 \\
			y_{21} & b & 0 & 0 \\
			y_{31} & y_{32} & 0 & 0 \\
			y_{41} & y_{42} & y_{43} & 0 
			\end{pmatrix}
			$$
			such that $[X,Y] = 0$, and $v \in \Hom(\BC^r,\BC^4)$ is cyclic for $X,Y$. Therefore, $\Sigma$ is an open subset of an affine bundle over the affine variety 
$$
\overline{\Sigma} = \Big\{(X,Y) \text{ as above, such that } [X,Y] = 0 \Big\}
$$
It suffices to show that $\overline{\Sigma}$ is normal. As an affine variety, it is cut out by
\begin{align*}
&ay_{21} - bx_{21} = 0	\\
&ay_{32} - bx_{32} = 0 \\
&x_{32}y_{21} - y_{32} x_{21} = 0 \\
&x_{43}y_{32} - y_{43} x_{32} + bx_{42} - ay_{42} = 0 \\
&x_{42}y_{21} + x_{43}y_{31} - y_{42}x_{21} - y_{43}x_{31} = 0
\end{align*}
in the 14-dimensional space of entries $a,b,x_{ij}, y_{ij}$. Using Macaulay2, one can check that $\overline{\Sigma}$ is a Cohen-Macaulay irreducible affine variety of dimension $10$. The tangent space to $\overline{\Sigma}$ at a given closed point $(X,Y)$ is the kernel of the map
$$
\mu : (\partial_X, \partial_Y) \mapsto [X,\partial_Y] + [\partial_X,Y]
$$
where $(\partial_X, \partial_Y)$ runs over the 14-dimensional affine space of pairs of matrices with the same pattern of zeroes as $(X,Y)$. Thus, $\overline{\Sigma}$ is smooth at a point $(X,Y)$ iff $\mu$ has 10-dimensional kernel, so let us see when this happens. \\

\noindent Assume first that $(a,b) \neq (0,0)$, and by taking appropriate linear combinations of $X, Y$ we may assume $(a,b) = (1,0)$. Then one can successively solve the equation $\Ker \mu = 0$ for the entries of the matrices $(\partial_X, \partial_Y)$ one by one, concluding that the kernel of $\mu$ fails to be 10-dimensional when $x_{43} = y_{21} = y_{31} = y_{32} = y_{42} = y_{42} = x_{31} - x_{32}x_{21} = 0$. The dimension of this locus is 5, and after we add 2 dimensions to reverse the choice $(a,b) = (1,0)$, we conclude that $\overline{\Sigma}$ is smooth in codimension 3 on the locus $(a,b) \neq (0,0)$. \\

\noindent Now assume $(a,b) = (0,0)$. The corresponding subvariety is precisely $\text{Comm}_4 \subset \overline{\Sigma}$, and we have seen at the end of Subsection \ref{sub:stack 4} that it is 9-dimensional and has two irreducible components $Z_1$ and $Z_2$. Thus, it suffices to take a generic point $(X,Y)$ in each of these components, and show that the map $\mu$ has 10-dimensional kernel at the chosen point. It is straightforward to show that the following choices will do:
$$
	X = \begin{pmatrix}
			0 & 0 & 0 & 0 \\
			1 & 0 & 0 & 0 \\
			0 & 0 & 0 & 0 \\
			0 & 0 & 1 & 0 
			\end{pmatrix}, \qquad Y = \begin{pmatrix}
			0 & 0 & 0 & 0 \\
			0 & 0 & 0 & 0 \\
			1 & 0 & 0 & 0 \\
			0 & 1 & 0 & 0 
			\end{pmatrix}
			$$
			and
			$$
			X = \begin{pmatrix}
			0 & 0 & 0 & 0 \\
			1 & 0 & 0 & 0 \\
			0 & 1 & 0 & 0 \\
			0 & 0 & 1 & 0 
			\end{pmatrix}, \qquad Y = \begin{pmatrix}
			0 & 0 & 0 & 0 \\
			0 & 0 & 0 & 0 \\
			0 & 0 & 0 & 0 \\
			0 & 0 & 0 & 0 
			\end{pmatrix}
			$$

			\item when $x \neq y$ and $\CF_4$ has non-zero defect at both $x$ and $y$, Corollary \ref{cor:def 2} shows that the corresponding locally closed subset has codimension $\geq 2$, hence can be ignored; \\	
			
			\item when $x \neq y$ and $\CF_4$ is locally free near $x$, the scheme $\fZ_\lambda$ is locally isomorphic to $\fZ_1 \times \quot_{0,1,2,3}$, and thus normal (indeed, $\quot_{0,1,2,3}$ is open in an affine bundle over $\stack_3$, and the latter is easily seen to be normal); \\
			
			\item when $x \neq y$ and $\CF_4$ has non-zero defect at $x$ but is locally free near $y$, the scheme $\fZ_\lambda$ is locally isomorphic to $\fZ_{(x,x,x)} \times \quot_1$, so it suffices to show that $\fZ_{(x,x,x)}$ is normal near any point with defect $\geq 1$. To this end, consider the stratification
			$$
			\fZ_{(x,x,x)} = \bigsqcup_{d = 0}^\infty \fZ_{(x,x,x)}^{\defect d}
			$$
			in terms of the defect of $\CF_4$ at $x$. As shown in Corollary \ref{cor:def 2}, the open subset $\fZ_{(x,x,x)}^{\defect 0}$ is the only top-dimensional stratum; as it is locally isomorphic to $\CM \times\quot_{0,1,2,3}$, it is normal. Tracing through the proof of Proposition \ref{prop:fiber} shows that other strata can have codimension 1 only if $r=1$ and $d=1$, in which case the moduli space of stable sheaves may be replaced with the Hilbert scheme of points on $S$. Therefore, it suffices to show that the scheme $\Sigma'$ parametrizing flags of ideals $(I_0 \subset_x I_1 \subset_x I_2 \subset_x I_3)$ is normal near any ideal $I_3$ of defect precisely 1 at $x$. Since the problem is local, we may assume $S = \BA^2$ and $x = (0,0)$, in which case $\Sigma'$ may be described by analogy with Subsection \ref{sub:adhm} as the space of triples $(X,Y,v)$ such that
			$$
			X = \begin{pmatrix}
			a & 0 & 0 & 0 \\
			x_{21} & 0 & 0 & 0 \\
			x_{31} & x_{32} & 0 & 0 \\
			x_{41} & x_{42} & x_{43} & 0 
			\end{pmatrix}, \qquad Y = \begin{pmatrix}
			b & 0 & 0 & 0 \\
			y_{21} & 0 & 0 & 0 \\
			y_{31} & y_{32} & 0 & 0 \\
			y_{41} & y_{42} & y_{43} & 0 
			\end{pmatrix}
			$$
			such that $[X,Y] = 0$, and $v \in \BC^4$ is cyclic for $X,Y$. As before, one may explicitly write down the quadratic equations among the $\{x_{ij}, y_{ij},a,b\}_{4 \geq i > j \geq 1}$ and conclude that $\Sigma'$ is normal. The method of proof is analogous to that in item (2) above, so we leave it as an analogous exercise to the interested reader.
			
		\end{enumerate}
		
	\end{proof}

\subsection{} 
	
	\noindent We will now study the schemes $\fY$, $\fY_-$, $\fY_+$, $\fY_{-+}$ of Subsection \ref{sub:def y}. \\
	
	\begin{proposition}
		\label{prop:y geom}
		
(a) The schemes $\fY$, $\fY_-$, $\fY_+$, $\fY_{-+}$ have expected dimension, i.e. the dimension of the respective spaces on the bottom of \eqref{eqn:pi top} or \eqref{eqn:pi bot}. \\

\noindent (b) The schemes $\fY$, $\fY_-$, $\fY_+$, $\fY_{-+}$ have $1$, $1$, $1$, $2$ irreducible components of expected dimension, respectively. \\
		
	\end{proposition}
	
	\noindent Recall that when we say that $\fY$ (or any of the other 3 schemes) has $\#$ irreducible components of expected dimension, what we actually mean is that it has $\#$ such irreducible components over each connected component of the moduli space $\CM$. \\
	
	\begin{proof} The map $\fY \xrightarrow{\pi^\uparrow} \fZ_{(x,y)}$ is surjective. Over a closed point $(\CF_0 \subset_x \CF_1 \subset_y \CF_2)$ $\in \fZ_{(x,y)}$, the fiber of this map is either a single point, or a copy of $\BP^1$. The latter happens if and only if $x = y$ and $\CF_2/\CF_0$ is split, so we conclude that the only points where the fibers jump are those of $\fZ_{(x,x)}^\split$. Since the locus $\fZ_{(x,x)}^\split$ has codimension $3$ in $\fZ_{(x,y)}$ (see Proposition \ref{prop:dim split}), and the dimensions of the fibers above such points are all 1, this implies that $\fY$ will have the same dimension as $\fZ_2$ (in Subsection \ref{sub:y smooth}, we will show that $\fY$ is actually smooth, so the map $\pi^\uparrow$ can be thought of as the blow-up of the singular locus of $\fZ_{(x,y)}$). The fact that $\fY$ has a single irreducible component of top dimension is clear, since $\fZ_{(x,y)}$ is irreducible. \\
		
		\noindent The proof for the schemes $\fY_-$, $\fY_+$, $\fY_{-+}$ is analogous. For example, the fibers of
		$$
		\pi^\uparrow : \fY_{-+} \longrightarrow \fZ_{(x,x,y,x)} 
		$$ 
		consist of a single point or a copy of $\BP^1$, with the latter situation only happening over closed points
		$$
		(\CF_0 \subset_x \CF_1 \subset_x \CF_2 \subset_y \CF_3 \subset_x \CF_4) \in \fZ_{(x,x,y,x)} \quad \text{s.t. } x = y \text{ and } \CF_3/\CF_1 \text{ is split} 
		$$
		Since the locus of such points is contained in $\fZ_{(x,x,x,x)}$, which has dimension 1 less than $\fZ_{(x,x,y,x)}$ by Propositions \ref{prop:xxxx} and \ref{prop:xxyx and xyxx}, the dimension of $\fY_{-+}$ is the same as that of $\fZ_{(x,x,y,x)}$. As for the irreducible components of top dimension in $\fY_{-+}$, we note that one of them is the closure of the locus $x \neq y$. But recall that
		$$
		\fZ_{(x,x,x,x)} = V_1 \cup V_2 
		$$
		where $V_1$ and $V_2$ lie above the irreducible components $Z_1$ and $Z_2$ of $\stack_4$ (see items (1)--(2) in Subsection \ref{sub:stack 4}). The map $\pi^\uparrow$ has inverse image a single point over the generic point of $V_2$. However, over any point of $V_1$ we have $\CF_3/\CF_1$ split, and so there exists a whole $\BP^1$ in $\fY_{-+}$ above points of $V_1$. As $\dim V_1 = \dim \fZ_{(x,x,x,x)} = \dim \fY_{-+} - 1$, this contributes an irreducible component of top dimension to $\fY_{-+}$. 
		
	\end{proof}
	
	\subsection{} 
	\label{sub:tangent}
	
	In the next Subsection, we will prove that the scheme $\fY$ is smooth. To do so, we will explicitly describe the tangent space to a closed point \eqref{eqn:y} of $\fY$ and compute its dimension. Let us recall that the tangent space to the moduli space $\CM$ at a point $\CF \in \text{Coh}(S)$ is given by
	\begin{equation}
	\label{eqn:kodaira spencer}
	\Tan_{\CF} \CM = \text{Ext}^1(\CF,\CF)
	\end{equation}
	Indeed, the functor-of-points description \eqref{eqn:bijection} implies that a tangent vector at $\CF$ $\in \CM$ is a coherent sheaf on $S \times \text{Spec } \BC[\nu]/(\nu^2)$ which is flat over the second factor and restricts to $\CF$ when one sets $\nu = 0$. In other words, a tangent vector is a coherent sheaf $\CG$ on $S$ with an morphism $\nu : \CG \rightarrow \CG$ that squares to 0, such that
	$$
	\CG/\text{Im }\nu \cong \CF
	$$
	The flatness condition on $\CG$ implies that $\text{Tor}^{\BC[\nu]/(\nu^2)}_1(\BC[\nu]/(\nu), \CG) = 0$, and so
	\begin{equation}
	\label{eqn:ses}
	0 \longrightarrow \CG/\text{Im }\nu \xrightarrow{ \cdot \nu} \CG \longrightarrow \CG/\text{Im }\nu \longrightarrow 0
	\end{equation}
	is a short exact sequence, which precisely gives rise to an element of $\Ext^1(\CF,\CF)$. The moduli space $\CM$ is smooth precisely when the dimensions of the tangent spaces \eqref{eqn:kodaira spencer} are locally constant in $\CF$ (for a more rigorous presentation of the smoothness of the moduli space via obstruction theory, we refer the reader to \cite{HL}). We have
	\begin{equation}
	\label{eqn:dim euler}
	\dim \Hom(\CF,\CF) - \dim \Ext^1(\CF,\CF) + \dim \Ext^2(\CF,\CF) =  \chi(\CF,\CF)
	\end{equation}
	Since stable sheaves are simple (see \cite{HL}), we have
	\begin{align}
	\dim \Hom(\CF,\CF) = 1 & \quad \text{because } \Hom(\CF,\CF) \cong \BC \label{eqn:simple} \\
	\dim \Ext^2(\CF,\CF) = \e & \quad \text{because } \Ext^2(\CF,\CF) \cong \Hom(\CF,\CF \otimes \omega_S)^\vee \label{eqn:serre} 
	\end{align}
	where the latter isomorphism is Serre duality, and the number $\e$ is $1$ or $0$ depending on which situation of Assumption S we are in ($\e = 1$ for $\omega_S \cong \CO_S$ and $\e = 0$ for $c_1(\omega_S) \cdot H < 0$). Therefore, we conclude that
	\begin{equation}
	\label{eqn:dim 0}
	\dim \Ext^1(\CF,\CF) = 1+ \e + \gamma + 2rc_2
	\end{equation}
	where $\chi(\CF,\CF) = -\gamma - 2rc_2$ can be computed using the Hirzebruch-Riemann-Roch theorem, and the constant $\gamma$ only depends on $S,H,r,c_1$. \\
	
	\subsection{} 
	\label{sub:y smooth}
	
	Following \cite[relation (2.23)]{Univ}, the tangent space to $\fZ_1$ at a closed point $(\CF_0 \subset_x \CF_1)$ is the vector space of pairs of the form:
	\begin{equation}
	\label{eqn:doubles}
	(w_0,w_1) \in \text{Ker} \left[ \Ext^1(\CF_0,\CF_0) \oplus \Ext^1(\CF_1,\CF_1) \xrightarrow{\psi} \Ext^1(\CF_0,\CF_1) \right]
	\end{equation}
	where the arrow is the difference of the two natural maps induced by the inclusion $\CF_0 \subset \CF_1$. These maps fit into the diagram below with exact rows and columns:
	\begin{equation}
	\label{eqn:ext diagram 1}
	\xymatrix{
		\Ext^1(\CF_1,\CF_0) \ar[d] \ar[r] & \Ext^1(\CF_0, \CF_0) \ar[d] \ar@{->>}[r] & \Ext^2(\BC_x,\CF_0) \ar[d] \\
		\Ext^1(\CF_1, \CF_1) \ar[r] \ar@{->>}[d] & \Ext^1(\CF_0,\CF_1) \ar[r] \ar[d] \ar@{.>}[dr] & \Ext^2(\BC_x, \CF_1) \ar@{->>}[d] \\
		\Ext^1(\CF_1, \BC_x) \ar[r] & \Ext^1(\CF_0,\BC_x) \ar@{->>}[r] & \Ext^2(\BC_x, \BC_x)}
	\end{equation}
	where we write $\CF_1/\CF_0 = \BC_x$ for the skyscraper sheaf at the closed point $x \in S$. The dimensions of the Ext spaces in the diagram above may be computed as in the previous Subsection:
	\begin{align*}
	&\dim \Ext^1(\CF_0, \CF_0) = 1 + \e + \gamma + 2rc_2^\first \\
	&\dim \Ext^1(\CF_1, \CF_1) = 1 + \e + \gamma + 2rc_2^\last \\
	&\dim \Ext^1(\CF_0, \CF_1) = 1 + \gamma + r(c_2^\first + c_2^\last) \\
	&\dim \Ext^1(\CF_1, \CF_0) = \e + \gamma + r(c_2^\first + c_2^\last)
	\end{align*}
	where $c_2^\first$ and $c_2^\last$ are the second Chern classes of the sheaves denoted by $\CF_0$ and $\CF_1$, respectively. The fact that the kernel of $\psi$ in \eqref{eqn:doubles} has the expected dimension $1 + \e + \gamma + r(c_2^\first + c_2^\last) + 1$ then follows from the elementary facts below: \\
	
	\begin{enumerate}
		
		\item the image of $\psi$ coincides with the kernel of the dotted arrow \\
		
		\item the target $\Ext^2(\BC_x, \BC_x)$ of the dotted arrow is 1 dimensional \\
		
		\item the dotted arrow is non-zero if and only if $\e = 0$ \\
		
	\end{enumerate} 
	
	\noindent These facts were proved in \cite{Univ}. As shown in \loccitt, a pair as in \eqref{eqn:doubles} contains precisely the same information as a commutative diagram with exact rows:
	\begin{equation}
	\label{eqn:diagram extensions}
	\xymatrix{0 \ar[r] & \CF_1 \ar[r] & \CG_1 \ar[r] & \CF_1 \ar[r] & 0 \\
		0 \ar[r] & \CF_0 \ar[r] \ar@{^{(}->}[u] & \CG_0 \ar[r] \ar@{^{(}->}[u] & \CF_0 \ar[r] \ar@{^{(}->}[u]& 0}
	\end{equation}
	In this language, the differential of the map $p_S : \fZ_1 \rightarrow S$ is given by
	\begin{equation}
\label{eqn:explicit differential}	\text{diagram \eqref{eqn:diagram extensions}} \quad \xrightarrow{dp_S} \quad
	\Big( 0 \rightarrow \CF_1/\CF_0 \rightarrow \CG_1/\CG_0 \rightarrow \CF_1/\CF_0 \rightarrow 0 \Big) \in \Tan_x S
	\end{equation}
	where we use the fact that $\CF_1/\CF_0 \cong \BC_x$ and the fact that there exists a canonical isomorphism $\Tan_x S = \Ext^1(\BC_x, \BC_x)$. A diagram in the kernel of $dp_{S}$ is one in which the extension $\CG_1/\CG_0$ splits, which precisely means that the diagram \eqref{eqn:diagram extensions} allows one to insert an extra row:
	$$
	\xymatrix{0 \ar[r] & \CF_1 \ar[r] & \CG_1 \ar[r] & \CF_1 \ar[r] & 0 \\
		0 \ar[r] & \CF_0 \ar[r] \ar@{^{(}->}[u] & \CH \ar[r] \ar@{^{(}->}[u] & \CF_1 \ar[r] \ar@{=}[u]& 0 \\
		0 \ar[r] & \CF_0 \ar[r] \ar@{=}[u] & \CG_0 \ar[r] \ar@{^{(}->}[u] & \CF_0 \ar[r] \ar@{^{(}->}[u]& 0}
	$$
	The ability to insert the middle row into the diagram above is equivalent to saying that the pair \eqref{eqn:doubles} comes from one and the same element in the vector space $\Ext^1(\CF_1,\CF_0)$ situated in the top left corner of diagram \eqref{eqn:ext diagram 1}. We conclude that:
	\begin{equation}
	\label{eqn:surjects}
	\Ext^1(\CF_1,\CF_0) \twoheadrightarrow \text{Ker }dp_S
	\end{equation}
	is a surjective map. Since the dimension of $\Ext^1(\CF_1,\CF_0)$ space is 2 less than that of the tangent space to $\fZ_1$, this implies that the map $dp_S$ is surjective. \\
	
	\begin{proof} \emph{of Proposition \ref{prop:y}:} Since we already proved the dimension and irreducibility statements in Proposition \ref{prop:y geom}, it remains to prove that $\fY$ is smooth. By analogy with the discussion above, we showed in \cite[relation (4.37)]{W univ} that the tangent space to $\fZ_2$ at a closed point $(\CF_0 \subset_x \CF_1 \subset_y \CF_2)$ consists of triples of the form
		\begin{multline}
		(w_0,w_1,w_2) \in \text{Ker } \Big[ \Ext^1(\CF_0,\CF_0) \oplus \Ext^1(\CF_1,\CF_1) \oplus \Ext^1(\CF_2,\CF_2) \\
		\longrightarrow \Ext^1(\CF_0,\CF_1) \oplus \Ext^1(\CF_1,\CF_2) \Big] 
		\label{eqn:triples}
		\end{multline}
		where the arrow is the alternating sum of the four natural maps induced by the inclusions $\CF_0 \subset \CF_1 \subset \CF_2$. In \loccitt, we also showed that
		$$
		\dim \text{ space of triples \eqref{eqn:triples}} = 1 + \e + \gamma + r(c_2^\first + c_2^\last) + 2 + \delta_{\CF_2/\CF_0}^{\text{split}}
		$$
		where $c_2^\first$ and $c_2^\last$ are the second Chern classes of the sheaves denoted by $\CF_0$  and $\CF_2$, respectively (meanwhile, the Kronecker $\delta$ symbol is 1 if $x = y$ and $\CF_2/\CF_0$ is split, i.e. $\cong \BC_x^{\oplus 2}$, and 0 otherwise). We conclude that the dimensions of the tangent spaces to $\fZ_2$ jump by 1 precisely on the split locus. The differential of the map
		\begin{equation}
		\label{eqn:projection 2}
		\fZ_2 \xrightarrow{p_S^1 \times p_S^2} S \times S, \qquad \qquad (\CF_0 \subset_x \CF_1 \subset_y \CF_2) \mapsto (x,y)
		\end{equation}
		admits a presentation analogous to \eqref{eqn:explicit differential}. It was shown in \cite{W univ} that the differential $dp_S^1 \times dp_S^2$ is surjective if and only if either $x \neq y$ or $x  = y$ and $\CF_2/\CF_0$ is split. \\
		
		\noindent By combining the discussion above with the moduli functor presentation of the scheme $\fY$, we see that $\Tan_{(\CF_0 \subset \CF_1,\CF_1' \subset \CF_2)} \fY$ is the space of quadruples
		$$
		(w_0,w_1,w_1',w_2) \in \Ext^1(\CF_0,\CF_0) \oplus \Ext^1(\CF_1,\CF_1) \oplus \Ext^1(\CF_1',\CF_1') \oplus \Ext^1(\CF_2,\CF_2)
		$$
		which satisfy the four properties below: \\
		
		\begin{enumerate}[leftmargin=*]
			
			\item $w_0$ and $w_1$ (or $w_1'$) map to the same element of $\Ext^1(\CF_0,\CF_1)$ (or $\Ext^1(\CF_0,\CF_1')$) \\
			
			\item $w_1$ (or $w_1'$) and $w_2$ map to the same element of $\Ext^1(\CF_1,\CF_2)$ (or $\Ext^1(\CF_1',\CF_2)$) \\
			
			\item $dp_S(w_0,w_1) = dp_S(w_1',w_2) \in \Ext^1(\BC_x,\BC_x)$ \\
			
			\item $dp_S(w_0,w_1') = dp_S(w_1,w_2) \in \Ext^1(\BC_y,\BC_y)$ \\
			
		\end{enumerate}
		
		\noindent By analogy with \eqref{eqn:doubles} and \eqref{eqn:ext diagram 1}, consider the vector space
		\begin{equation}
		\label{eqn:doubless}
		A = \text{Ker} \left[ \Ext^1(\CF_0,\CF_0) \oplus \Ext^1(\CF_2,\CF_2) \xrightarrow{\psi'} \Ext^1(\CF_0,\CF_2) \right]
		\end{equation}
		where $\psi'$ is the difference of the two natural maps in the diagram below with exact rows and columns:
		\begin{equation}
		\label{eqn:ext diagram 2}
		\xymatrix{
			\Ext^1(\CF_2,\CF_0) \ar[d] \ar[r] & \Ext^1(\CF_0, \CF_0) \ar[d] \ar@{->>}[r] & \Ext^2(\CQ,\CF_0) \ar[d] \\
			\Ext^1(\CF_2, \CF_2) \ar[r] \ar@{->>}[d] & \Ext^1(\CF_0,\CF_2) \ar[r] \ar[d] \ar@{.>}[dr] & \Ext^2(\CQ, \CF_2) \ar@{->>}[d] \\
			\Ext^1(\CF_2, \CQ) \ar[r] & \Ext^1(\CF_0,\CQ) \ar@{->>}[r] & \Ext^2(\CQ,\CQ)}
		\end{equation}
		(here $\CF_2/\CF_0 = \CQ$ is a length 2 sheaf which is filtered by $\BC_x$ and $\BC_y$). By analogy with our analysis of \eqref{eqn:ext diagram 1}, it is easy to show that the image of the map $\psi'$ in \eqref{eqn:doubless} coincides with the kernel of the dotted arrow, and so
		\begin{equation}
		\label{eqn:dim a}
		\dim A = 1 + \e + \gamma + r(c_2^\first + c_2^\last) + \dim \Ext^2(\CQ,\CQ)
		\end{equation}
		The dimension of $\Ext^2(\CQ,\CQ)$ is 4 if $x = y$ and $\CQ$ is split, and 2 otherwise. Items (1)--(4) above imply that we have a Cartesian diagram of vector spaces:
		\begin{equation}
		\label{eqn:square y}
		\xymatrixcolsep{4pc} \xymatrix{\Tan_{(\CF_0 \subset \CF_1,\CF_1' \subset \CF_2)} \fY \ar[r]^-{a'} \ar[d]_-{b'} & \Tan_{(\CF_0 \subset_y \CF_1' \subset_x \CF_2)} \fZ_2 \ar[d]^-{(b,dp_S^2, dp_S^1)} \\
			\Tan_{(\CF_0 \subset_x \CF_1 \subset_y \CF_2)} \fZ_2 \ar[r]^-{(a,dp_S^1, dp_S^2)} & A \oplus \Tan_x S \oplus \Tan_y S}
		\end{equation}
		where the maps $a, a'$ forget $w_1$ and the maps $b,b'$ forget $w_1'$. \\
		
		\begin{claim} 
			\label{claim:smooth}
			
			The map $a$ is injective, unless $x = y$ in which case $\emph{Ker }a$ is one-dimensional and spanned by $(0,w_1,0)$, where $w_1$ represents the following extension:
			\begin{equation}
			\label{eqn:ext smooth}
			0 \longrightarrow \CF_1 \xrightarrow{(\text{inclusion},0)} \CF_2 \oplus_{\BC_x} \CF_1 \xrightarrow{(0,\text{projection})} \CF_1 \longrightarrow 0
			\end{equation}
			(the middle space requires fixing isomorphisms $\CF_2/\CF_1 \cong \CF_1/\CF_0 \cong \BC_x$). The image of the entension \eqref{eqn:ext smooth} under $dp_S^1 \times dp_S^2$ is equal to $(v,v) \in \eTan_x S \oplus \eTan_x S$, where $v \in \emph{Ext}^1(\BC_x,\BC_x)$ is the class of the extension $0 \rightarrow \CF_1/\CF_0 \rightarrow \CF_2/\CF_0 \rightarrow \CF_2/\CF_1 \rightarrow 0$. \\
			
		\end{claim}
		
		\noindent We will first show how Claim \ref{claim:smooth} allows us to prove that all tangent spaces to $\fY$ have dimension $\leq$ than:
\begin{equation}
\label{eqn:expected number}
\dim \fY \stackrel{\text{Prop. \ref{prop:y geom}}}=	\dim \fZ_2 \stackrel{\text{Prop. \ref{prop:xy}}}= 1 + \e + \gamma + r(c_2^\first + c_2^\last) + 2 
\end{equation}
(which would conclude the proof of Proposition \ref{prop:y}) and then prove the Claim. \\

\noindent \emph{\textbf{Case 1:}} when $x \neq y$ or $x = y$ and $\CF_2/\CF_0$ is not split (i.e. we are at a smooth point of $\fZ_2$), Claim \ref{claim:smooth} implies that $(a,dp_S^1,dp_S^2)$ is injective. Then the fact that diagram \eqref{eqn:square y} is Cartesian implies that the map $a'$ is also injective, which implies
$$
\dim \Tan_{(\CF_0 \subset \CF_1,\CF_1' \subset \CF_2)} \fY \leq \dim \Tan_{(\CF_0 \subset_y \CF_1' \subset_x \CF_2)} \fZ_2	
$$	
Because of Proposition \ref{prop:y geom} and the fact that we are on the smooth locus of $\fZ_2$, the dimensions of the two tangent spaces must be equal. \\
					
\noindent \emph{\textbf{Case 2:}} when $x = y$ and $\CF_2/\CF_0$ is split, the dimensions of the vector spaces in \eqref{eqn:square y} are
$$
		\xymatrixcolsep{4pc} \xymatrix{? \ar[r]^-{a'} \ar[d]_-{b'} & \BC^{d+1} \ar[d]^-{(b,dp_S^2, dp_S^1)} \\
			\BC^{d+1} \ar[r]^-{(a,dp_S^1, dp_S^2)} & \BC^{d+2} \oplus \BC^2 \oplus \BC^2}
$$
where $d$ is the number in the right-hand side of \eqref{eqn:expected number}. The goal is to show that the Cartesian product of the diagram, namely the vector space $?$, has dimension $d$. If $\CF_1 = \CF_1'$, then the northeast and southwest corners are naturally identified, as are the maps $a$ and $b$. By Claim \ref{claim:smooth}, we can decompose $\BC^{d+1} = \BC \oplus \BC^d$, where $a(\BC) = 0$ and $a|_{\BC^d}$ is injective. Moreover, $dp_S^1(\BC) = dp_S^2(\BC) = 0$, which implies that $dp_S^1 \times dp_S^2|_{\BC^d}$ is surjective (this follows from the surjectivity of $dp_S^1 \times dp_S^2$, proved in \cite[relation (4.47)]{W univ}). We conclude that a point of $?$ is of the form
$$
(l,l',v) \in \BC \oplus \BC \oplus \BC^d
$$
such that $dp_S^1(v) = dp_S^2(v)$. The latter equality imposes two non-trivial linear conditions on $v$, so we conclude that the dimension of $?$ is $1+1+d-2 = d$. \\

\noindent If $\CF_1 \neq \CF_1'$, then it suffices to prove that $\text{Im } a$ and $\text{Im } b$ are transversal $d$-dimensional subspaces of $A$. Consider a point of $A$ given by a pair of extensions
			\begin{equation}
			\label{eqn:diagram extension 2}
			\xymatrix{0 \ar[r] & \CF_2 \ar[r] & \CG_2 \ar[r] & \CF_2 \ar[r] & 0 \\
				0 \ar[r] & \CF_0 \ar[r] \ar@{^{(}->}[u] & \CG_0 \ar[r] \ar@{^{(}->}[u] & \CF_0 \ar[r] \ar@{^{(}->}[u]& 0}
			\end{equation}
			The diagram above induces an extension at the level of quotients
			\begin{equation}
			\label{eqn:white}
			0 \rightarrow \CQ \rightarrow \CH \rightarrow \CQ \rightarrow 0
			\end{equation}
			where $\CQ = \CF_2/\CF_0$ and $\CH = \CG_2/\CG_0$. The pair of extensions \eqref{eqn:diagram extension 2} lies in $\text{Im }a$ iff
			\begin{equation}
			\label{eqn:condition}
			\CH \text{ has a length 2 subscheme compatible with } \CF_1/\CF_0 \subset \CQ
			\end{equation}
			(and similarly for $\text{Im }b$, if we replace $\CF_1$ by $\CF_1'$). Fix a vector space isomorphism:
			$$
			\CF_2/\CF_0 \cong \BC^2
			$$ 
			with respect to which $\CF_1/\CF_0$ is the first standard coordinate line and $\CF_1'/\CF_0$ is the second coordinate line. The rank 4 coherent sheaf $\CH$ is determined on the local neighborhood of $x \in S$ by two commuting $4 \times 4$ matrices $X$ and $Y$, whose only non-zero entries are allowed to be in the bottom left $2 \times 2$ block, as below:
			$$
			X = \begin{pmatrix}
			0 & 0 & 0 & 0 \\
			0 & 0 & 0 & 0 \\
			x_1 & x_2 & 0 & 0 \\
			x_3 & x_4 & 0 & 0 
			\end{pmatrix} \qquad \text{and} \qquad Y = \begin{pmatrix}
			0 & 0 & 0 & 0 \\
			0 & 0 & 0 & 0 \\
			y_1 & y_2 & 0 & 0 \\
			y_3 & y_4 & 0 & 0 
			\end{pmatrix}
			$$
			Condition \eqref{eqn:condition} is equivalent to $x_3 = y_3 = 0$, while the analogous condition with $\CF_1$ replaced by $\CF_1'$ is equivalent to $x_2 = y_2 = 0$. Taken together, this would prove the desired fact that the images of $a$ and $b$ are transverse codimension two subspaces of $A$, as soon as we prove that the map
			$$
			A \xrightarrow{\alpha} \Ext^1(\CQ,\CQ) \qquad \eqref{eqn:diagram extension 2} \mapsto \eqref{eqn:white} 
			$$
			is surjective. To this end, note that the dimension of $A$ is given by \eqref{eqn:dim a}, $\dim \Ext^2(\CQ,\CQ) = 4$, $\dim \Ext^1(\CQ,\CQ) = 8$. Therefore, it suffices to show that
			$$
			\dim \text{Ker }\alpha \leq 1+\e+\gamma + r(c_2^\first + c_2^\last) - 4
			$$
			By analogy with \eqref{eqn:surjects}, we have a surjective map
			$$
			\Ext^1(\CF_2,\CF_0) \stackrel{\beta}\twoheadrightarrow \text{Ker }\alpha
			$$
			A simple application of Hirzebruch-Riemann-Roch shows that $\dim \Ext^1(\CF_2,\CF_0)$ $ = \e+\gamma + r(c_2^\first + c_2^\last)$. Therefore, it suffices to show that the kernel of the map $\beta$ has dimension $\geq 3$. To this end, consider the following piece of the Ext long exact sequence corresponding to the short exact sequence $0 \rightarrow \CF_0 \rightarrow \CF_2 \rightarrow \CQ \rightarrow 0$:
			$$
			\xymatrix{\Hom(\CF_2,\CF_2) \ar[r] & \Hom(\CF_2,\CQ) \ar[r]^{\rho} \ar[d] & \Ext^1(\CF_2,\CF_0) \ar[r] \ar[d] & \Ext^1(\CF_2,\CF_2) \\
				& \Hom(\CF_0,\CQ) \ar[r] & \Ext^1(\CF_0,\CF_0)}
			$$
			Consider the 4-dimensional subspace $V \subset \Hom(\CF_2,\CQ)$ which consists of a fixed homomorphism with kernel $\CF_0$, composed with an arbitrary endomorphism of $\CQ\cong \BC_x^2$. Any element of $V$ maps to 0 in both $\Ext^1(\CF_0,\CF_0)$ and $\Ext^1(\CF_2,\CF_2)$, so $\rho(V)$ lies inside the kernel of $\beta$. However, the map $\rho$ has a 1-dimensional kernel since $\Hom(\CF_2,\CF_2) = \BC$, so we conclude that $\text{Ker } \beta$ has dimension at least 3. \\

		\begin{proof} \emph{of Claim \ref{claim:smooth}:} It is enough to show that any triple $(0,w_1,0)$ satisfying
			\begin{align*}
			&0 \text{ and }w_1 \text{ map to the same element in }\Ext^1(\CF_0,\CF_1) \\
			&w_1 \text{ and }0 \text{ map to the same element in }\Ext^1(\CF_1,\CF_2)
			\end{align*}
			must have $w_1$ equal to a multiple of the extension \eqref{eqn:ext smooth}. Indeed, the natural long exact sequences imply that it suffices to show that any $w_1 \in \Ext^1(\CF_1,\CF_1)$ which lies in the intersection of the images of $\Ext^1(\BC_x,\CF_1)$ and $\Hom(\CF_1,\BC_y)$ is a multiple of \eqref{eqn:ext smooth}, where $\BC_x = \CF_1/\CF_0$ and $\BC_y = \CF_2/\CF_1$. In other words, if we have a diagram
			\begin{equation}
			\label{eqn:temp diag}			
			\xymatrix{0 \ar[r] & \CF_1 \ar[r] & \CF_2 \ar[r]^-{\text{taut}} & \BC_y \ar[r] & 0 \\
				0 \ar[r] & \CF_1 \ar[r] \ar@{=}[u] \ar@{=}[d] & \CG \ar[r] \ar[u] \ar[d] & \CF_1 \ar[r] \ar[u]^-\alpha \ar[d]^-{\text{taut}} & 0 \\
				0 \ar[r] & \CF_1 \ar[r] & \CH \ar[r]^-\beta & \BC_x \ar[r] & 0 
			}
			\end{equation}
		(the maps denoted ``taut" are the projection maps $\CF_1 \twoheadrightarrow \BC_x$ and $\CF_2 \twoheadrightarrow \BC_y$ that give rise to the flag $\CF_0 \subset \CF_1 \subset \CF_2$) where the middle short exact sequence is the pull-back of both the the top and the bottom short exact sequences, we must show that the middle short exact sequence is a multiple of \eqref{eqn:ext smooth}. As the sheaves $\CF_0,\CF_1,\CF_2, \CH$ all have the same reflexive hull, and since the reflexive hull is stable, we may regard $\CF_0,\CF_1,\CF_2, \CH$ as subsheaves of the same stable vector bundle $\CV$.  \\

\begin{enumerate}[leftmargin=*]

\item If $\CH \neq \CF_2$ (as subsheaves of $\CV$), then the two sides of the inclusion
$$
\CF_1 \subseteq \CH \cap \CF_2
$$
have the same colength as subsheaves of $\CV$, so the inclusion above is an equality. Similarly, the two sides of the inclusion:
$$
\CG \subseteq (\CH \cap \CF_2) \oplus \CF_1 = \CF_1 \oplus \CF_1
$$
have the same colength as subsheaves of $\CV \oplus \CV$, so the inclusion above is an equality. But then the short exact sequence $0 \rightarrow \CF_1 \rightarrow \CG \rightarrow \CF_1 \rightarrow 0$ is split, so $w_1 = 0$. \\

\item if $\CH = \CF_2$ (as subsheaves of $\CV$), then $x = y$. Since $\text{Hom}(\CF_1, \CF_2)$ is one dimensional, any two injections $\CF_1 \hookrightarrow \CF_2$ are scalar multiples of each other, which implies that $\beta = \lambda \cdot \taut$ for some $\lambda \in \BC^\times$ (notations as in \eqref{eqn:temp diag}). Then the extension $\CG$ is equal to $\lambda$ times the extension \eqref{eqn:ext smooth}, as we needed to show. 
			
			\end{enumerate}

		\end{proof}
		
	\end{proof}
	
	\subsection{}
	
	Let us now consider the schemes $\fY_-$, $\fY_+$, $\fY_{-+}$ in relation to $\fY$. \\	
	
	\begin{proof} \emph{of Propositions \ref{prop:y- and y+} and \ref{prop:y-+}:} Since we already proved the dimension and irreducibility statements in Proposition \ref{prop:y geom}, it remains to prove that $\fY_-$, $\fY_+$ and $\fY_{-+}$ are l.c.i. This is a consequence of the claim that the following squares are all derived, which we will now prove:
\begin{equation}
\label{eqn:derived y 1}
		\xymatrix{\fY_+ \ar[r] \ar[d] & \fZ_{(x,x)} \ar[d]  \\
			\fY \ar[r] & \fZ_{(x)}} \qquad \qquad \xymatrix{\fY_- \ar[r] \ar[d] & \fZ_{(x,x)} \ar[d]  \\
			\fY \ar[r] & \fZ_{(x)}} 
\end{equation}
\begin{equation}
\label{eqn:derived y 2}
		\xymatrix{\fY_{-+} \ar[r] \ar[d] & \fZ_{(x,x)} \ar[d] \\
			\fY_- \ar[r] & \fZ_{(x)}} \qquad \qquad \xymatrix{\fY_{-+} \ar[r] \ar[d] & \fZ_{(x,x)} \ar[d] \\
			\fY_+ \ar[r] & \fZ_{(x)}}
\end{equation}
		In all cases above, the arrow on the left is the only map one can write which forgets a single sheaf (in the notation of \eqref{eqn:y}--\eqref{eqn:y bullet}), while the arrow on the right is the unique map which forgets the same sheaf as the arrow on the left. \\

\noindent We will only prove the fact that the first square in \eqref{eqn:derived y 1} is derived, since all other cases are analogous. Consider the map on the left of the square
		\begin{equation}
		\label{eqn:kiyomizu}
		\fY_+ \longrightarrow \fY
		\end{equation}
		With the notation in \eqref{eqn:y} and \eqref{eqn:y plus}, we note that the fibers of this map consist of all ways to append a sheaf $\CF_{-1} \subset_x \CF_0$ to diagram \eqref{eqn:y}. Just like in Proposition \ref{prop:tower}, one sees that the map \eqref{eqn:kiyomizu} factors as
		$$
		\xymatrix{\fY_+ \ar[rd] \ar@{^{(}->}[r]^-\iota & \BP_{\fY} (\Gamma^{x*} (\CV_0)) \ar[d] \\
			& \fY}
		$$
		where $\Gamma^x : \fY \rightarrow \fY \times S$ is the graph of the map $p_S^x$ that records the support point $x \in S$. The closed embedding $\iota$ is cut out by the composition
		$$
		\sigma : \Gamma^{x*}(\CW_0) \longrightarrow \Gamma^{x*} (\CV_0) \longrightarrow \CO(1)
		$$
		and just like in Proposition \ref{prop:tor}, one may show that the section $\sigma$ factors through a locally free sheaf of rank 1 less, as follows:
		$$
		\sigma : \Gamma^{x*}(\CW_0) \twoheadrightarrow \frac {\Gamma^{x*}(\CW_0)}{\CL_1 \otimes p_S^{x*}(\omega_S)} \xrightarrow{\sigma'} \CO(1)
		$$
		(the argument requires the fact that $p_S^x$ is flat, which is proved by estimating the dimensions of its fibers, akin to the proof of Proposition \ref{prop:flat morphism}). Because of \eqref{eqn:inequality}, we obtain:
		$$
		\dim \fY_+ - \dim \fY \geq r
		$$
		However, Proposition \ref{prop:y} implies that $\fY$ is smooth, while Proposition \ref{prop:y geom} implies that $\dim \fY_+ = \dim \fY + r$. Therefore, we actually have equality in the inequality above. This is a particular case of Definition \ref{def:regular}, hence the section $\sigma'$ is regular. However, this is precisely the same section that describes the map $\fZ_{(x,x)} \rightarrow \fZ_{(x)}$. By Definition \ref{def:lci}, this precisely says that the first fiber square in \eqref{eqn:derived y 1} is derived.
		
	\end{proof}

	\begin{proof} \emph{of Proposition \ref{prop:reduced}:} The scheme $\fY$ is reduced because it is smooth. As for the other schemes, they are local complete intersections, so it suffices to prove that their generic points are reduced. In the case of $\fY_-$ and $\fY_+$, they are irreducible, and the generic point corresponds to a diagram \eqref{eqn:y minus}--\eqref{eqn:y plus} with $x \neq y$. Near such a point, $\fY_-$ and $\fY_+$ are isomorphic to $\fZ_{(y,x,x)}$ and $\fZ_{(x,x,y)}$, respectively. Since the latter schemes are normal (due to Proposition \ref{prop:normal}), reducedness follows. \\
		
		\noindent The same argument applies to the irreducible component of $\fY_{-+}$ which is the closure of the locus of diagrams \eqref{eqn:y bullet} with $x \neq y$. As for the other irreducible component, we recall that it corresponds to diagrams \eqref{eqn:y bullet} with $x = y$ and $\CF_3/\CF_1$ a split length 2 sheaf. Therefore, the second component is locally isomorphic to
		$$
		V_1 \times \BP^1
		$$
		where $V_1 \subset \fZ_{(x,x,x,x)}$ is the irreducible component consisting of
		$$
		(\CF_0 \subset_x \CF_1 \subset_x\CF_2 \subset_x\CF_3 \subset_x\CF_4)
		$$
		such that $\CF_3/\CF_1$ is a split length 2 sheaf. It suffices to show that $V_1$ is generically reduced. As a consequence of Corollary \ref{cor:def 2}, the generic point of $V_1$ corresponds to $\CF_4$ locally free. Lemma \ref{lem:locally isomorphic} implies that near such a point, $V_1$ is locally isomorphic to the smooth moduli space $\CM_4 = \{ \CF_4 \}$ times the component $Z_1/B \subset \comm_4/B$ that we studied in Subsection \ref{sub:stack 4}. As we noted therein, $Z_1$ is generically reduced, so we are done. \\
		
	\end{proof}

\end{document}